\documentclass[12pt]{amsart}

\usepackage{graphicx} 
\usepackage{scrextend}
\usepackage{amsfonts, amsmath, amssymb,amsthm,comment}  
\usepackage{times,enumerate}
\usepackage{mathdots}%
\usepackage{nccmath}
\usepackage{mathtools}
\usepackage[shortlabels]{enumitem}
\usepackage{multicol}
    {\end{pmatrix}\end{medsize}}%
\usepackage{dsfont,bbm}
\usepackage{rotating}
\usepackage{lscape}
\usepackage{url}
\usepackage{multicol}
\usepackage{thmtools, thm-restate}
\usepackage{kbordermatrix}
\usepackage{blkarray}
\usepackage{multicol}
\usepackage[margin=2.5cm]{geometry}
 \usepackage[foot]{amsaddr}
 \usepackage{hyperref}
\usepackage{cancel}

\usepackage[usenames, dvipsnames]{xcolor}
\DeclareGraphicsExtensions{.pdf,.png,.jpg}

\hypersetup{
    colorlinks,
    linkcolor={blue},
    citecolor={blue},
    urlcolor={blue}
}

\usepackage[section]{algorithm}
\usepackage{algpseudocode}

\DeclareMathOperator{\modulo}{mod}

\DeclareMathOperator{\Span}{span}

\DeclareMathOperator{\Rank}{rank}

\DeclareMathOperator{\Card}{card}

\DeclareMathOperator{\rev}{rev}

\makeatletter
\def\widebreve{\mathpalette\wide@breve}
\def\wide@breve#1#2{\sbox\z@{$#1#2$}%
     \mathop{\vbox{\m@th\ialign{##\crcr
\kern0.08em\brevefill#1{0.8\wd\z@}\crcr\noalign{\nointerlineskip}%
                    $\hss#1#2\hss$\crcr}}}\limits}
\def\brevefill#1#2{$\m@th\sbox\tw@{$#1($}%
  \hss\resizebox{#2}{\wd\tw@}{\rotatebox[origin=c]{90}{\upshape(}}\hss$}
\makeatletter

\newcommand{\RR}{\mathbb R}

\newcommand{\NN}{\mathbb N}
\newcommand{\ZZ}{\mathbb Z}
\newcommand{\CC}{\mathbb C}

\newcommand{\cB}{\mathcal B}

\newcommand{\cV}{\mathcal V}

\newcommand{\cZ}{\mathcal Z}
\newcommand{\cC}{\mathcal C}

\newcommand{\benu}{\begin{enumerate}}
\newcommand{\eenu}{\end{enumerate}}
\newcommand{\bop}{\begin{opomba}}
\newcommand{\eop}{\end{opomba}}

\newcommand{\supp}{\mathrm{supp}}

\newtheorem{theorem}{Theorem}[section]
\newtheorem{corollary}[theorem]{Corollary}

\newtheorem{proposition}[theorem]{Proposition}

\theoremstyle{definition}

\newtheorem{example}[theorem]{Example}

\newcommand{\mbb}{\mathbb}
\newcommand{\mbf}{\mathbf}

\usepackage[section]{algorithm}
\usepackage{algpseudocode}

\theoremstyle{remark}
\newtheorem{remark}[theorem]{Remark}

\numberwithin{equation}{section}

\begin{document}

\title{The truncated moment problem on curves $y=q(x)$ and $yx^\ell=1$}

\author[Alja\v z Zalar]{Alja\v z Zalar}

\address{%
Faculty of Computer and Information Science\\
University of Ljubljana\\
Ve\v cna pot 113\\
1000 Ljubljana\\
Slovenia}

\address{%
Faculty of Mathematics and Physics \\
University of Ljubljana\\
Jadranska ulica 19\\
1000 Ljubljana\\
Slovenia}

\address{%
Institute of Mathematics, Physics and Mechanics\\
Jadranska ulica 19\\
1000 Ljubljana\\
Slovenia}

\email{aljaz.zalar@fri.uni-lj.si}

\thanks{Supported by the Slovenian Research Agency grants J1-2453, J1-3004, P1-0288.}


\date{\today}
\maketitle

\begin{abstract}
	In this paper we study the bivariate truncated moment problem (TMP) 
	on curves of the form $y=q(x)$, $q(x)\in \RR[x]$, $\deg q\geq 3$, and $yx^\ell=1$, $\ell\in \NN\setminus\{1\}$.
	For even degree sequences the solution based on the size of moment matrix extensions 
	was first given by Fialkow \cite{Fia11} using the truncated Riesz-Haviland theorem \cite{CF08}
	and a sum-of-squares representations for polynomials, strictly positive on such curves 
	\cite{Fia11,Sto01}.
	Namely, the upper bound on this size is quadratic in the degrees of the
	sequence and the polynomial determining a curve.
	We use a reduction to the univariate setting technique,
	introduced in \cite{Zal21,Zal22a,Zal22b},
	and improve Fialkow's bound to $\deg q-1$ (resp.\ $\ell+1$) for curves $y=q(x)$ (resp.\ $yx^\ell=1$).
	This in turn gives analogous improvements of 
	the degrees in the sum-of-squares representations referred to above.
	Moreover, we get the upper bounds on the number of atoms in the minimal representing measure,
	which are $k\deg q$ (resp.\ $k(\ell+1)$) for curves $y=q(x)$ 
	(resp.\ $yx^\ell=1$) 
	for even degree sequences, 
	while for odd ones they are 
	$k\deg q-\big\lceil\frac{\deg q}{2} \big\rceil$ (resp.\ $k(\ell+1)-\big\lfloor\frac{\ell}{2} \big\rfloor+1$) for curves $y=q(x)$ 
	(resp.\ $yx^\ell=1$).
	In the even case these are counterparts to the result by Riener and Schweighofer \cite[Corollary 7.8]{RS18}, 
	which gives the same bound for odd degree sequences on all plane curves.
	In the odd case their bound is slightly improved on the curves we study.
	Further on, 
	we give another solution to the TMP on the curves studied based on the feasibility
	of a linear matrix inequality, 
	corresponding to the univariate sequence obtained,
	and finally we solve concretely odd degree cases to the TMP on curves $y=x^\ell$, $\ell=2,3$,
	and add a new solvability condition to the even degree case on the curve $y=x^2$.
\end{abstract}

\section{Introduction}
\label{S1}

Given a real $2$--dimensional sequence
	$$\beta^{ (d)}=\{\beta_{0,0},\beta_{1,0},\beta_{0,1},\ldots,\beta_{d,0},\beta_{d-1,1},\ldots,
		\beta_{1,d-1},\beta_{0,d}\}$$
of degree $d$ and a closed subset $K$ of $\RR^2$, the \textbf{truncated moment problem ($K$--TMP)} supported on $K$ for $\beta^{ (d)}$
asks to characterize the existence of a positive Borel measure $\mu$ on $\RR^2$ with support in $K$, such that
	\begin{equation}
		\label{moment-measure-cond}
			\beta_{i,j}=\int_{K}x^iy^j d\mu\quad \text{for}\quad i,j\in \ZZ_+,\;0\leq i+j\leq d.
	\end{equation}
If such a measure exists, we say that $\beta^{ (d)}$ has a representing measure 
	supported on $K$ and $\mu$ is its $K$--\textbf{representing measure ($K$--rm).}

Let $k=\big\lceil\frac{d}{2}\big\rceil$.
In the degree-lexicographic order $1,X,Y,X^2,XY,Y^2,\ldots,X^k,X^{k-1}Y,\ldots,Y^k$ of rows and columns, the corresponding moment matrix to $\beta$
is equal to 
	\begin{equation}	
	\label{281021-1448}
		M_k=M_k(\beta):=
		\left(\begin{array}{cccc}
		M[0,0](\beta) & M[0,1](\beta) & \cdots & M[0,k](\beta)\\
		M[1,0](\beta) & M[1,1](\beta) & \cdots & M[1,k](\beta)\\
		\vdots & \vdots & \ddots & \vdots\\
		M[k,0](\beta) & M[k,1](\beta) & \cdots & M[k,k](\beta)
		\end{array}\right),
	\end{equation}
where
	$$M[i,j](\beta):=
		\left(\begin{array}{ccccc}
		\beta_{i+j,0} & \beta_{i+j-1,1} & \beta_{i+j-2,2} & \cdots & \beta_{i,j}\\
		\beta_{i+j-1,1} & \beta_{i+j-2,2} & \beta_{i+j-3,3} & \cdots & \beta_{i-1,j+1}\\
		\beta_{i+j-2,2} & \beta_{i+j-3,3} & \beta_{i+j-4,4} & \cdots & \beta_{i-2,j+2}\\
		\vdots & \vdots & \vdots & \ddots &\vdots\\
		\beta_{j,i} & \beta_{j-1,i+1} & \beta_{j-2,i+2} & \cdots & \beta_{0,i+j}\\
		\end{array}\right)$$
and for odd $d$, the lower right corner $M[k,k]$ of $M_k(\beta)$ is undefined. 
Until the end of this section we assume that $M_k$ is fully determined, i.e., it corresponds to the even degree sequence $\beta^{(2k)}$.
Let 
$\RR[x,y]_k:=\{p\in \RR[x,y]\colon \deg p\leq k\}$ 
stand for the set of real polynomials in variables $x,y$ of degree at most $k$, where for $p\not\equiv 0$ the degree $\deg p$ stands for the maximal sum $i+j$ over all monomials 
$x^iy^j$ appearing in $p$ with a nonzero coefficient $a_{ij}$, while for $p\equiv 0$, $\deg p=0$.
For every $p(x,y)=\sum_{i,j} a_{ij}x^iy^j\in \RR[x,y]_k$ we define
its \textbf{evaluation} $p(X,Y)$ on the columns of the matrix $M_k$ by replacing each capitalized monomial $X^iY^j$
in $p(X,Y)=\sum_{i,j} a_{ij}X^iY^j$ by the column of $M_k$, indexed by this monomial.
Then $p(X,Y)$ is a vector from the linear span of the columns of $M_k$. If this vector is the zero one, i.e., all coordinates are equal to 0, 
then we say $p$ is a \textbf{column relation} of $M_{k}$.
Recall from \cite{CF96}, that $\beta$ has a rm $\mu$ with the support $\supp(\mu)$ being a 
subset of 
$\cZ(p):=\{(x,y)\in \RR^2\colon p(x,y)=0\}$ if and only if $p$ is a column relation of $M_{k}$.
We say that the matrix $M_{k}$ is \textbf{recursively generated (rg)} if for $p,q,pq\in \RR[x,y]_k$ such that $p$ is a column relation of $M_k$, 
it follows that $pq$ is also a column relation of $M_k$.

A \textbf{concrete solution} to the TMP is a set of necessary and sufficient conditions for the existence of a $K$--rm, that can be tested in 
numerical examples. 
Among necessary conditions, $M_{k}$ must be positive semidefinite (psd) and rg 
\cite{CF96}.
A crucial tool to tackle the TMP, discovered by Curto an Fialkow in 1996, was a \textbf{flat extension theorem (FET)} \cite[Theorem 7.10]{CF96} (see also \cite[Theorem 2.19]{CF05b} and \cite{Lau05} for 
an alternative proof), which states that $\beta^{(2k)}$ admits a $(\Rank M_{k})$--atomic 
rm
if and only if $M_{k}$ is psd and admits a rank-preserving extension to a moment matrix $M_{k+1}$. 
Using the FET as the main tool the bivariate TMP has been concretely solved in
the following cases: 
	$K$ is the variety defined by a polynomial $p(x,y)=0$ with $\deg p\leq 2$ 
	\cite{CF02, CF04, CF05,Fia14};
	$K=\RR^2$, $k=2$ and $M_2$ is invertible \cite{CY16}, first solved nonconstructively in \cite{FN10};
	$K$ is the variety $y=x^3$ \cite{Fia11};
 	$M_{k}$ has a special feature such as \textit{recursive determinateness} \cite{CF13} or \textit{extremality} \cite{CFM08}.
Some special cases have also been solved in \cite{CY15,Yoo17a,Yoo17b} based on the FET and in \cite{Ble15,BF20,DS18,Fia17,Kim14} using different approaches.

References to some classical work on the TMP are monographs \cite{Akh65,AhK62,KN77}, while for a recent development in the area we refer a reader to \cite{Sch17}.
We also mention some variants of the TMPs, which attracted a recent research interest, such as versions of the infinite dimensional TMPs \cite{AJK15,GKM16,IKLS17}, the TMP on subspaces of polynomial algebra \cite{Nie14}, the TMP for commutative $\RR$--algebras \cite{CGIK+}, matrix and operator TMPs \cite{And70,BK12,BZ18,BZ21,BW11,DFMT09,DS02,KT22,KW13}, etc.

In our previous work we introduced a new approach to tackle the singular bivariate TMP, namely \textit{a reduction to the univariate setting technique}.
The idea is to use one of the column relations to transform the problem into the equivalent univariate TMP, 
where also negative moments of the measure could be present or not all moments between the lowest and the highest degree ones are known.
In the case all moments from degree 0 to the highest degree are known, the situation is well understood in terms of the existence and uniqueness of the
rm
and has been solved in full generality \cite{CF91}
for measures with support $\RR$, $[a,\infty)$ or $[a,b]\subset \RR$, $a,b \in \RR$, $a<b$,
as well as for even and odd degree sequences. 
In the presence of negative moments we gave a solution along the lines of the classical case in \cite{Zal22b}, where we note that the existence of the solution even 
in the matrix case was already established by Simonov \cite{Sim06} (but the measure is not constructively obtained and the number of atoms in a minimal measure does not directly follow from this more general approach). 
Using these results we presented \cite{BZ21,Zal21,Zal22a,Zal22b} alternative solutions with shorter proofs compared to the original ones 
to the TMPs on the curves $xy=0$, $y=x^3$, $y^2=y$, $xy=1$,
but also obtained solutions to new cases, namely on the curve $y^2=x^3$, on the union of three parallel lines and on $xy^2=1$.

The motivation for this paper was to use a reduction technique to the TMP on curves of the form $y=q(x)$ and $yq(x)=1$, where $q\in \RR[x]$.
In \cite[Section 6]{Fia11}, Fialkow gave a solution to the TMP on these curves for even degree sequences
in terms of the bound on the degree $m$ for which the existence of a positive extension $M_m$ of $M_{k}$ is equivalent to the existence of a rm.
Namely, his bound is quadratic in $k$ and $\deg q$.
Using a reduction technique we are able to decrease his bound in the even degree case for all curves of the form $y=q(x)$,
$\deg q\geq 3$, to $\deg q-1$ and for curves of the form $yx^\ell=1$, $\ell\in \NN\setminus\{1\}$, to $\ell+1$, which is our first main result.
Moreover, the reduction technique also works in the odd degree case.
A corollary to this improved bounds are also improvements of the sum-of-squares representations for polynomials, strictly positive on such curves, by decreasing the degrees of the polynomials in the representation.
Our second main result are the upper bounds on the number of atoms in the minimal 
rm,
which are for curves $y=q(x)$, $\deg q\geq 3$, equal to $k\deg q$ in the even and $k\deg q-\big\lceil\frac{\deg q}{2} \big\rceil$ in the odd case and for curves $yx^\ell=1$, $\ell\geq 2$, equal to $k(\ell+1)$ in the even and 
$k(\ell+1)-\big\lfloor\frac{\ell}{2} \big\rfloor+1$ in the odd case. In the even case these results 
are counterparts to the result of Riener and Schweighofer \cite[Corollary 7.8]{RS18}, who proved that for \textit{all} plane curves, odd degree sequence have at most $k\deg q$ atoms in the minimal measure.
For curves of the above form we improve their bound slightly in the odd degree case.
The third main result of the paper is another solution to the TMPs studied, which is based on the feasibility
of a linear matrix inequality corresponding to the univariate sequence obtained. 
Moreover, we give concrete solutions to the odd degree TMPs on the curves $y=x^2$ and $y=x^3$ and an alternative solution to the even degree case on $y=x^2$ 
with a new solvability condition, which will be crucially needed in the solution of the TMP on the reducible curve $y(y-x^2)=0$ in our forthcoming work.

\subsection{Reader's guide} 

The paper is organized as follows.
In Section \ref{251122-2143} we fix some further notation and known results on the TMP, which will be used in the proofs of our results.
In Section \ref{251122-2146} we give two solutions to the $K$--TMP for $K=\{(x,y)\in\RR^2\colon y=q(x)\}$, $q\in \RR[x]$, $\deg q\geq 3$,
one based on the size of psd extensions of the moment matrix needed (see Theorems \ref{240822-1317} and \ref{240822-1317-odd}
for the even and the odd degree cases, respectively) and the other one based on the feasibility question of a certain
linear matrix inequality (see Theorem \ref{121122-1341}). Theorems \ref{240822-1317} and \ref{240822-1317-odd} 
also give bounds on the number of atoms in a minimal $K$--rm.
Moreover, Theorem \ref{240822-1317} 
gives a Positivstellensatz on $K$ as a corollary (see Corollary \ref{290822-1100}). Further on, we solve concretely the TMPs
on the curve $y=x^2$ (see Theorems \ref{090622-1549} and \ref{090622-1549-odd} for the even and the odd degree cases, respectively) 
and on $y=x^3$ for the odd case (see Theorem \ref{21122-0855-odd}).
In Section \ref{251122-2158} we give the corresponding results to the ones from Section \ref{251122-2146} 
for curves $yx^\ell=1$, $\ell\in \NN\setminus\{1\}$.
Theorems  \ref{280822-2053} and \ref{151122-1348} are the counterparts of Theorems \ref{240822-1317} and \ref{240822-1317-odd}, respectively, 
Corollary \ref{290822-1052} of Corollary \ref{290822-1100} and Theorem \ref{14122-1729} of Theorem \ref{121122-1341}.\\




\section{Preliminaries}
\label{251122-2143}

In this section we fix some terminology, notation and present some tools needed in the proofs of our main results.

We write $\RR^{n\times m}$ for the set of $n\times m$ real matrices. For a matrix $M$ 
we call the linear span of its columns a \textbf{column space} and denote it by $\cC(M)$.
The set of real symmetric matrices of size $n$ will be denoted by $S_n$. 
For a matrix $A\in S_n$ the notation $A\succ 0$ (resp.\ $A\succeq 0$) means $A$ is positive definite (pd) (resp.\ positive semidefinite (psd)).

In the rest of this section let $d\in \NN$ and $\beta=\beta^{ (d)}=\{\beta_{i,j}\}_{i,j\in \ZZ_+,\; 0\leq i+j\leq d}$ be a bivariate sequence of degree $d$.

\subsection{Moment matrix}
Let $k=\big\lceil\frac{d}{2}\big\rceil$ and $M_k=M_k(\beta)$ be the moment matrix of $\beta$ (see \eqref{281021-1448}).
Let $Q_1, Q_2$ be subsets of the set $\{X^iY^j\colon  i,j \in \ZZ_+,\; 0\leq i+j\leq k\}$.
We denote by 
$(M_k)|_{Q_1,Q_2}$ 
the submatrix of $M_k$ consisting of the rows indexed by the elements of $Q_1$
and the columns indexed by the elements of $Q_2$. In case $Q:=Q_1=Q_2$, we write 
$(M_k)|_{Q}=(M_k)|_{Q,Q}$
for short. 


\begin{remark}
\label{301122-1948}
Whenever $Q_1,Q_2$ will be subsets of $\{x^iy^j\colon  i,j \in \ZZ_+,\; 0\leq i+j\leq k\}$ in the rest of the paper, in the notation $(M_k)|_{Q_1,Q_2}$ all monomials from $Q_1$, $Q_2$ are meant capitalized, i.e., $x^iy^j\mapsto X^iY^j$.
\end{remark}

\subsection{Atomic measures}
For $x\in \RR^m$, $\delta_x$ 
stands for the Dirac measure supported on $x$.
By a \textbf{finitely atomic positive measure} on $\RR^m$ we mean a measure of the form $\mu=\sum_{j=0}^\ell \rho_j \delta_{x_j}$, 
where $\ell\in \NN$, each $\rho_j>0$ and each $x_j\in \RR^m$. The points $x_j$ are called 
\textbf{atoms} of the measure $\mu$ and the constants $\rho_j$ the corresponding \textbf{densities}.

\subsection{Riesz functional}
The functional $L_{\beta}:\mbb{R}[x,y]_{\leq d}\to \RR$ , defined by 
$$
	L_{\beta}(p):=\sum_{\substack{i,j\in \ZZ_+,\\ 0\leq i+j\leq d}} a_{i,j}\beta_{i,j},\qquad \text{where}\quad p=
	\sum_{\substack{i,j\in \ZZ_+,\\ 0\leq i+j\leq d}} a_{i,j}x^iy^j,
$$
is called the \textbf{Riesz functional of the sequence $\beta$}.

\subsection{Affine linear transformations} \label{affine linear-trans}

Let $K\subseteq \RR^2$. The existence of a $K$--rm
for $\beta$ is invariant under invertible affine linear transformations (alt\textit{s}) of the form 
\begin{equation}
	\label{130922-1215}
		\phi(x,y)=(\phi_1(x,y), \phi_2(x,y)):=(a+bx+cy,d+ex+fy),\; (x,y)\in \RR^{2},
\end{equation}
$a,b,c,d,e,f\in \RR$ with $bf-ce \neq 0$ in the sense which we now explain.
We denote by $\widetilde \beta$ the $2$--dimensional sequence defined by
	$$\widetilde \beta_{i,j}=L_{\beta}\big(\phi_1(x,y)^i \cdot \phi_2(x,y)^j\big),$$
where
	$L_{\beta}$ is the Riesz functional of $\beta$.

\begin{proposition}
	[{\cite[Proposition 1.9]{CF05}}] 
		\label{251021-2254}
	Assume the notation above and let $d=2k$.
	\begin{enumerate}
		\item $M_k(\beta)$ is psd if and only if $M_k(\widetilde \beta)$ is psd.
		\item $\Rank M_k(\beta)=\Rank M_k(\widetilde \beta)$.
		\item $M_k(\beta)$ is rg if and only if $M_k(\widetilde \beta)$ is rg.
		\item
			\label{291021-2333} 
				$\beta$ admits a $r$--atomic $K$--rm
				if and only if 
				$\widetilde \beta$ admits a $r$--atomic $\phi(K)$--rm.
	\end{enumerate}
\end{proposition}

In case $d=2k-1$ is odd, the block $M[k,k]$ of $M_k(\beta)$ is undefined. We say that $M_k(\beta)$ is \textbf{psd completable} if there exists an extension $\beta^{(2k)}$ of $\beta$ 
such that $M_k(\beta^{(2k)})$ is psd. 

\begin{proposition}
		\label{151122-1555}
	Assume the notation above and let $d=2k-1$, $k\in \NN$.
	\begin{enumerate}
		\item 
			$M_k(\beta)$ is psd completable 
			if and only if 
			$M_k(\widetilde \beta)$ is psd completable.
		\item Let $r\in \NN$.
			There exists an extension $\beta^{(2k)}$ of $\beta$ such that $\Rank M_k(\beta^{(2k)})=r$
			if and only if	
			there exists an extension $\widetilde \beta^{(2k)}$ of $\widetilde \beta$ such that $\Rank M_k(\widetilde\beta^{(2k)})=r$.
		\item Let $r\in \NN$.
			There exists an extension $\beta^{(2k)}$ of $\beta$ such that $M_k(\beta^{(2k)})$ is rg
			if and only if	
			there exists an extension $\widetilde \beta^{(2k)}$ of $\widetilde \beta$ such that $M_k(\widetilde\beta^{(2k)})$ is rg.
		\item
			\label{151122-1321-pt4} 
				$\beta$ admits a $r$--atomic $K$--rm
				if and only if 
				$\widetilde \beta$ admits a $r$--atomic $\phi(K)$--rm.
	\end{enumerate}
\end{proposition}

\begin{proof}
	Proposition \ref{151122-1555} follows easily from Proposition \ref{251021-2254} by defining the extension $\widetilde \beta^{(2k)}$ of $\widetilde \beta$ from
	the extension $\beta^{(2k)}$ of $\beta$ using the same transformation $\phi$ together with the Riesz functional $L_{\beta^{(2k)}}$ of the extension.	
	Similarly, for the other direction one uses $\phi^{-1}$ together with the Riesz functional $L_{\widetilde \beta^{(2k)}}$ of the extension.	
	For \eqref{151122-1321-pt4} we notice that any $r$--atomic $K$--rm	
 	of the sequence $\beta$ generates the extension $\beta^{(2k)}$
	and then use \eqref{291021-2333} of Proposition \ref{251021-2254}.
\end{proof}

\subsection{Hankel matrices and univariate sequences}\label{SubS2.2}
Let $k\in \NN$.
For 
		$v=(v_0,\ldots,v_{2k} )\in \RR^{2k+1}$
we define the corresponding Hankel matrix as
	\begin{equation}\label{vector-v}
		A_{v}:=\left(v_{i+j} \right)_{i,j=0}^k
					=\left(\begin{array}{ccccc} 
							v_0 & v_1 &v_2 & \cdots &v_k\\
							v_1 & v_2 & \iddots & \iddots & v_{k+1}\\
							v_2 & \iddots & \iddots & \iddots & \vdots\\
							\vdots 	& \iddots & \iddots & \iddots & v_{2k-1}\\
							v_k & v_{k+1} & \cdots & v_{2k-1} & v_{2k}
						\end{array}\right)
					\in S_{k+1}.
	\end{equation}
Let
	$\mbf{v_j}:=\left( v_{j+\ell} \right)_{\ell=0}^k$ be the $(j+1)$--th column of $A_{v}$, $0\leq j\leq k$.
	In this notation, we have that
		$$A_{v}=\left(\begin{array}{ccc} 
								\mbf{v_0} & \cdots & \mbf{v_k}
							\end{array}\right).$$
As in \cite{CF91}, the \textbf{rank} of $v$, denoted by $\Rank v$, is defined by
	$$\Rank v=
	\left\{\begin{array}{rl} 
		k+1,&	\text{if } A_{v} \text{ is nonsingular},\\
		\min\left\{i\colon \bf{v_i}\in \Span\{\bf{v_0},\ldots,\bf{v_{i-1}}\}\right\},&		\text{if } A_{v} \text{ is singular}.
	 \end{array}\right.$$
We denote
\begin{itemize} 
	\item the upper left-hand corner $\left(v_{i+j} \right)_{i,j=0}^m\in S_{m+1}$ of $A_{v}$ of size $m+1$ by 
		$A_{v}(m)$.
	\item the lower right-hand corner $\left(v_{i+j} \right)_{i,j=k-m}^k\in S_{m+1}$ of $A_{v}$ of size $m+1$ by
		$A_{v}[m]$.
\end{itemize}
We write 
	$$v^{(\rev)}:=(v_{2k},v_{2k-1},\ldots,v_0)$$ 
for the \textbf{reversed sequence} of $v$.

A sequence $v=(v_0,\ldots,v_{2k})$ is called
\textbf{positively recursively generated (prg)} if, denoting 
	$r=\Rank v$,
it holds that 
	$A_v(r-1)\succ 0$
and in case $r<k+1$, also
 	$$v_j=\sum_{i=0}^{r-1} \varphi_i v_{j-r+i} \quad \text{for}\quad j=r,\ldots,2k,$$
where		
\begin{equation}
	\label{211222-1818}
	\left(\begin{array}{ccc}\varphi_0 & \cdots & \varphi_{r-1}\end{array}\right):=
						A_{v}(r-1)^{-1}
						\left(\begin{array}{ccc}v_r & \cdots &v_{2r-1}\end{array}\right)^{T}.
\end{equation}

A sequence $v=(v_0,\ldots,v_{2k} )$ is called
\textbf{negatively recursively generated (nrg)}  if, denoting
	$r=\Rank v^{(\rev)}$,
it holds that
	$A_v[r-1]\succ 0$
and in case $r<k+1$, also
	$$v_{2k-r-j}=\sum_{i=0}^{r-1}\psi_i v_{2k-r+1-j+i}
	\quad \text{for}\quad j=0,\ldots,2k-r,$$
where
$$
	\left(\begin{array}{ccc} \psi_0 & \cdots & \psi_{r-1}\end{array}\right):=A_{v}[r-1]^{-1}
	\left(\begin{array}{ccc}v_{2k-2r+1} & \cdots & v_{2k-r}\end{array}\right)^{T}.
$$
\subsection{Univariate truncated moment problem}

Given a real sequence 
	$$\gamma^{ (k_1,k_2)}=(\gamma_{k_1},\gamma_{k_1+1},\ldots,\gamma_{k_2-1},\gamma_{k_2})$$ 
of degree $(k_1,k_2)$, $k_1,k_2\in \ZZ$, $k_1\leq k_2$,
a subset $K$ of $\RR$, the \textbf{truncated moment problem supported on $K$ for $\gamma^{ (k_1,k_2)}$ ($(K,k_1,k_2)$--TMP)}
asks to characterize the existence of a positive Borel measure $\mu$ on $\RR$ with support in $K$, such that
	\begin{equation}
		\label{moment-measure-cond-univ}
			\gamma_{i}=\int_{K}x^i d\mu\quad \text{for}\quad i\in \ZZ,\;k_1\leq i\leq k_2.
	\end{equation}
If such a measure exists, we say that $\gamma^{ (k_1,k_2)}$ has a rm supported on $K$ and $\mu$ is its $K$--\textbf{rm}. 

The $(\RR,0,k)$--TMP with $k\in \ZZ_+$ is the usual \textbf{truncated Hamburger moment problem (THMP) of degree $k$}, which was solved in full generality in \cite{CF91}.
Algorithm \ref{alg-090323-1157} is a numerical procedure to determine the existence and extract a rm. It is an adaptation of \cite[Algorithm 4.2]{Las09}, which 
is an algorithm to extract a rm in the multivariate TMP under the assumption that a rm exists, to the univariate setting with the addition of deciding whether a rm
exists. For the latter the solution from \cite{CF91} is used.

\begin{algorithm}[H]
	\caption{Solution to the $(\RR,0,2k)$--TMP}
	\label{alg-090323-1157}
	\begin{algorithmic}[0]
		\State \textbf{Input:} A univariate sequence $\gamma=(\gamma_0,\gamma_1,\ldots,\gamma_{2k})$.
		\State \textbf{Output:} A negative answer to the existence of a rm for $\gamma$ or 
			the points $x_1,\ldots,x_{r}\in \RR$ and densities $\rho_{1},\ldots,\rho_{r}$ such that $\mu=\sum_{j=1}^{r} \rho_j \delta_{x_j}$ is a $\RR$--rm for $\gamma$.
		\medskip

		\State \textbf{Algorithm:}
		\State \textbf{1.} Try to compute Cholesky factorization $VV^T$ of $A_\gamma$, where $V\in \RR^{(k+1)\times r}$ is a lower triangular matrix. 
			In case of a failure, a rm does not exist.
		\State \textbf{2.1} If $r<k+1$, reduce $V$ to a column echelon form $U$ using Gaussian elimination on columns.
		\State \textbf{2.2} Else $r=k+1$. Choose $\gamma_{2k+1}\in \RR$ arbitrarily and solve a linear system
				$Vw=
				\begin{pmatrix}
					\gamma_{k+1} & \ldots & \gamma_{2k+1}
				\end{pmatrix}^T.
				$
				Reduce 
					$\widetilde V=\begin{pmatrix} V^T & w\end{pmatrix}^T$ 
				to a column echelon form 
					$U$ 
				using Gaussian elimination on columns.
		\State \textbf{3.1.} If $r<k+1$ and $U$ is of the form 
				$U=
				\begin{pmatrix}
					I_r &
					B
				\end{pmatrix}^T$
				for some $B\in \RR^{r\times (k+1-r)}$ or $r=k+1$,
				define $N$ as a $\RR^{r\times r}$ matrix 
				consisting of rows $2,\ldots,r+1$ of $U$.
		\State \textbf{3.2.} If $r<k+1$ and $U$ is not as in 3.1, then a rm for $\gamma$ does not exist.
		\State \textbf{4.} Compute the Schur decomposition $N=QTQ^T$ of $N$.		
				The diagonal elements $x_1,\ldots,x_r$ of $T$ are atoms in a rm for $\gamma.$
		\State \textbf{5.}
				Solve the system 
				$W\rho=
				\begin{pmatrix}
					\gamma_{0} & \ldots & \gamma_{r-1}
				\end{pmatrix}^T,
				$	
				where $W$ is the $r\times r$ Vandermonde matrix with $i$-th row equal to $\begin{pmatrix} x_j^{i-1} \end{pmatrix}_{j=1,\ldots,r}$.
				The coordinates $\rho_j$, $j=1,\ldots,r$, of $\rho$ are the densities of a rm measure for $\gamma$. 
	\end{algorithmic}
\end{algorithm}

\begin{remark}
	\begin{enumerate}	
	\item 
	\textbf{Correctness of Algorithm \ref{alg-090323-1157}:} Step 1 checks whether $A_\gamma$ is psd. Step 2.1 only changes a basis of the column space of $V$.	
	If Step 2.2 applies, then on can choose $\gamma_{2k+1}$ arbitrarily and then compute $\gamma_{2k+2}$ such that for the extended sequence
	$\widetilde \gamma=(\gamma,\gamma_{2k+1},\gamma_{2k+2})$ Step 2.1 applies, i.e., $A_{\widetilde \gamma}$ is psd, 
	$\Rank A_{\widetilde \gamma}=k+1<k+2$ and its Cholesky decomposition is equal to $\widetilde V(\widetilde V)^T$.
     In the situation of Step 3.1, $A_\gamma(r)$ is pd and hence $\gamma$ is prg. If Step 3.2 applies, the latter is not true. 
	By \cite[Theorem 3.9]{CF91}, a rm exists iff $\gamma$ is prg.	
	Since $U(v^{(0,r-1)}_{x_i})=v^{(0,r-1)}_{x_i}$ for every atom $x_i$ in a rm for $\gamma$,
	where 
		$v^{(0,p)}_{x}=\begin{pmatrix}1 & x \ldots & x^{p}\end{pmatrix}^T$,
	it follows that
		$N(v^{(0,r-1)}_{x_i})=x_iv^{(0,r-1)}_{x_i}$
	and the atoms of a rm are the eigenvalues of $N$. These are computed in Step 4 and finally, a Vandermonde system in Step 5 determines the densities.
	\item
	\textbf{Uniqueness of a rm:} 
	If $r<k+1$ and Step 3.1 applies, then the rm for $\gamma$ is unique. Otherwise, in case $r=k+1$, there are infinitely many choices of $(k+1)$--atomic rm\textit{s}. 
	Each choice of $\gamma_{2k+1}$ gives a different one.
	\item 
	\textbf{An adaptation of Algorithm \ref{alg-090323-1157} to solving $(\RR,0,2k+1)$--TMP:} One has to do the following modifications:
	\begin{enumerate}
		\item Let $\widetilde \gamma=(\gamma_0,\gamma_1,\ldots,\gamma_{2k+1})$ and $\gamma$ 
			as in Algorithm \ref{alg-090323-1157}.
		\item If Step 3.1 applies with $r<k+1$, one has to check whether the system
			$Vw=
				\begin{pmatrix}
					\gamma_{k+1} & \ldots & \gamma_{2k+1}
				\end{pmatrix}^T
				$
			is solvable. If not, a rm does not exists. Otherwise it does and compute it as in Algorithm \ref{alg-090323-1157}.
		\item If Step 3.1 applies with $r=k+1$, one does not choose $\gamma_{2k+1}$, since it is already given. Hence, the minimal $(k+1)$--atomic rm is unique.
	\end{enumerate}
	\end{enumerate}
\end{remark}

The $(\RR\setminus\{0\},k_1,k_2)$--TMP with $k_1,k_2\in \ZZ$, $k_1<0<k_2$
is the \textbf{strong truncated Hamburger moment problem (STHMP) of degree $(k_1,k_2)$}. 
For even $k_1$ and $k_2$ the solution is \cite[Theorem 3.1]{Zal22b}, but the technique in the proof
can be extended to establish also the cases, where $k_1$, $k_2$ are not both even.

Let 
	$\RR[x^{-1},x]=
		\left\{ \sum_{i=r_1}^{r_2} a_i x^i \colon a_i \in \RR, r_1,r_2\in \ZZ, r_1\leq r_2\right\}$
be the set of Laurent polynomials.
For $k_1,k_2\in \ZZ$, $k_1\leq k_2$, we denote by
	$V_{(k_1,k_2)}$ 
a vector subspace in $\RR[x^{-1},x]$ generated by the 
set 
	$\{x^{k_1},x^{k_1+1},\ldots,x^{k_2}\}$.
For a sequence $\gamma:=\gamma^{ (k_1,k_2)}$ the functional
	 $L_{\gamma}:V_{(k_1,k_2)}\to \RR$ , defined by 
$$
	L_{\gamma}(p):=\sum_{k_1\leq i\leq k_2} a_{i}\gamma_i,\qquad \text{where}\quad 
p=\sum_{k_1\leq i\leq k_2} a_{i}x^i,
$$
is called the \textbf{Riesz functional of the sequence $\gamma$}.

\begin{remark}[An adaptation of Algorithm \ref{alg-090323-1157} to solving $(\RR,-2k_1,2k_2)$--TMP, $k_1,k_2\in \NN$] 
	One has to do the following modifications (see \cite[Theorem 3.1]{Zal22b}):
	\begin{enumerate}
		\item The input is a sequence $\gamma=(-\gamma_{2k_1},\gamma_{-2k_1+1},\ldots,2k_2)$.
			The output is a negative answer to the existence of a rm for $\gamma$ or 
			the points $x_1,\ldots,x_{r}\in \RR$ and densities $\rho_{1},\ldots,\rho_{r}$ such that $\mu=\sum_{j=1}^{r} \rho_j \delta_{x_j}$ is a $\RR$--rm for $\gamma$.
		\item One forms a sequence
			$\widetilde\gamma=(\widetilde\gamma_{0},\ldots,\widetilde\gamma_{2(k_1+k_2)})$,
			where $\widetilde \gamma_{i}=\gamma_{i-2k_1}$, $i=0,1,\ldots,2(k_1+k_2)$,
			and do all computations from Algorithm \ref{alg-090323-1157} on $\widetilde\gamma$.
		\item If Step 2.2 applies, one chooses $\gamma_{2(k_1+k_2)+1}$ arbitrarily except for the number
			$v^TC^{-1}v$ in case $C$ is invertible, 
			where $C$ is a submatrix of $A_{\widetilde\gamma}$ consisting of rows $1,\ldots,k_1+k_2$ and columns $2,\ldots,k_1+k_2+1$,
			and $v=\begin{pmatrix}\gamma_{-k_1+k_2+1} & \cdots & \gamma_{-2k_2}\end{pmatrix}^T$ is a vector.
		\item If Step 3.1 applies with $r<k_1+k_2+1$, one has to check whether $(r+1)$-th column of $A_{\widetilde \gamma}$ is in the span of	
			columns $2,\ldots,r$. If yes, then a rm does not exist. Otherwise it does. 
			Equivalently, one can compute the Cholesky decomposition $V_1V_1^T$ of the restriction of 
			$A_{\widetilde \gamma}$ to the principal submatrix on rows and columns $2,3,\ldots,r+1$
			and see whether $V_1$ has rank $r$ or not.
		\item If  Step 3.1 applies with $r=k_1+k_2+1$, then a rm exists.
		\item The atoms are computed as in Step 4.
		\item 
			The densities in Step 5	are obtained by solving the system
				$$W\rho=
				\begin{pmatrix}
					\gamma_{-2k_1} & \gamma_{-2k_1+1}\ldots & \gamma_{-2k_1+r-1}
				\end{pmatrix}^T,
				$$	
				where $W$ is the $r\times r$ Vandermonde matrix with $i$-th row equal to $\begin{pmatrix} x_j^{-2k_1+i-1} \end{pmatrix}_{j=1,\ldots,r}$.
				The coordinates $\rho_j$, $j=1,\ldots,r$, of $\rho$ are the densities of a rm measure for $\gamma$.
	\end{enumerate}
\end{remark}


\section{The TMP on the curves $y=q(x)$}
\label{251122-2146}

In this section we study the $K$--TMP for $K$ being a curve of the form $y=q(x)$, $q\in \RR[x]$.
In Subsection \ref{251122-2112} we first give a solution of the $K$--TMP, $\deg q\geq 3$, 
based on the  size  of positive semidefinite extensions of the moment matrix needed and also bound the number of atoms in
the $K$--rm
with the smallest number of atoms (see Theorem \ref{240822-1317} for the even degree and 
Theorem \ref{240822-1317-odd} for the odd degree sequences). 
As a result we obtain a sum-of-squares representation for polynomials, which are strictly positive on $K$ (see Corollary \ref{290822-1100}). 
This improves bounds on the degrees in the previously known result \cite[Proposition 6.3]{Fia11}.
In Subsection \ref{251122-2113} we apply the technique from the proofs of the results from Subsection \ref{251122-2112} 
to give a concrete solution to the TMP on the curve $y=x^2$,
which is an alternative solution to the one from \cite{CF04} in the even case (see Theorem \ref{090622-1549})
and is new in the odd case (see Theorem \ref{090622-1549-odd}). 
In Subsection \ref{251122-2111} we give a solution to the $K$--TMP based on a feasibility of the corresponding linear matrix inequality 
(see Theorem \ref{121122-1341}). Finally, in Subsection \ref{251122-2119} we concretely solve the TMP on the curve $y=x^3$ 
in the odd degree case (see Theorem \ref{21122-0855-odd}).

\subsection{Solution to the TMP in terms of psd extensions of $M_k$, bounds on the number of atoms in the minimal measure and a Positivstellensatz}
\label{251122-2112}

\begin{theorem}[Even case]
	\label{240822-1317}
	Let $K:=\{(x,y)\in \RR^2\colon y=q(x)\}$, where $q\in \RR[x]$ with $\deg q\geq 3$,
	and
		$\beta:=\beta^{(2k)}=(\beta_{i,j})_{i,j\in \ZZ_+,i+j\leq 2k}$ with $k\geq \deg q$.
	The following statements are equivalent:
	\begin{enumerate}
		\item
		 	\label{240822-1317-pt1}
				$\beta$ has a $K$--representing measure.
		\item
		 	\label{240822-1317-pt4}
				$\beta$ has a $s$--atomic $K$--representing measure for some $s$ satisfying 
 					$$\Rank M_k\leq s\leq k\deg q.$$
		\item 
		 	\label{240822-1317-pt2}
				$M_k$ satisfies $Y=q(X)$ and admits a positive semidefinite, recursively generated extension $M_{k+\deg q-2}$.
		\item
		 	\label{240822-1317-pt3}
				$M_k$ satisfies $Y=q(X)$ and admits a positive semidefinite extension 
					$M_{k+\deg q-1}$.
	\end{enumerate}
\end{theorem}

\begin{theorem}[Odd case]
	\label{240822-1317-odd}
	Let $K:=\{(x,y)\in \RR^2\colon y=q(x)\}$, where $q\in \RR[x]$ with $\deg q\geq 3$,
	and 
		$\beta^{(2k-1)}=(\beta_{i,j})_{i,j\in \ZZ_+,i+j\leq 2k-1}$ with $k\geq \deg q$.
	Then the following statements are equivalent:
	\begin{enumerate}
		\item
		 	\label{240822-1317-pt1-odd}
				$\beta$ has a $K$--representing measure.
		\item
		 	\label{240822-1317-pt4-odd}
				$\beta$ has a $s$--atomic $K$--representing measure for some $s$ satisfying 
 					$$\Rank M_{k-1}\leq s\leq k\deg q-\Big\lceil\frac{\deg{g}}{2}\Big\rceil.$$
		\item 
		 	\label{240822-1317-pt2-odd}
				$\beta^{(2k-1)}$ can be extended to a sequence $\beta^{(2k)}$
				such that $M_{k}$ satisfies $Y=q(X)$
				and admits a positive semidefinite, recursively generated extension $M_{k+\deg q-2}$.
		\item
		 	\label{240822-1317-pt3-odd}
				$\beta^{(2k-1)}$ can be extended to a sequence $\beta^{(2k)}$
				such that $M_{k}$ satisfies $Y=q(X)$
				and admits a positive semidefinite extension $M_{k+\deg q-1}$.
	\end{enumerate}
\end{theorem}

\begin{remark}
\label{191122-1055}
\begin{enumerate}
\item\textbf{Previous bounds on the size of extensions in \eqref{240822-1317-pt3} of Theorem \ref{240822-1317}:}
In \cite{CF08}, Curto and Fialkow studied polynomials $p\in\RR[x,y]$ for which the existence of the $\cZ(p)$--rm
is equivalent to the psd moment matrix extension of some size. 
In \cite[Section 6]{Fia11} the author considered polynomials of the form 
	$p(x,y)=y-q(x)$,  
	where $q\in \RR[x]$,
			and proved that 
			a sequence of degree $2k$ admits a $\cZ(p)$--rm, 
			if $M_k$ admits a psd extensions $M_{k+r}$,
			where
					$r=(2k+1)\deg q-k$
   		 \cite[Propositions 6.1, 6.3]{Fia11}.
	The proof of this result relies on the truncated Riesz-Haviland theorem \cite[Theorem 1.2]{CF08}
	and a sum-of-squares representations for polynomials, strictly positive on $\cZ(p)$ 
	(\cite[Proposition 6.3]{Fia11} and \cite[Proposition 5.1]{Sto01}).
Part \eqref{240822-1317-pt3} of Theorem \ref{240822-1317} improves Fialkow's result by decreasing the size of the extensions to 
$r=\deg q-1$. We mention that this was known for the case of the curve $y=x^3$ \cite[Corollary 5.3]{Fia11}.

\item
\label{191122-1055-pt2}
\textbf{Known bounds on the number of atoms in \eqref{240822-1317-pt4} of Theorems \ref{240822-1317}, \ref{240822-1317-odd}}:
In \cite{RS18} the authors also studied odd degree sequences $\beta$, which are moments of a positive Borel measure supported on a plane curve $\cZ(p)$, $p\in \RR[x,y]$, and proved
that every such sequence admits a $(k\deg p)$--atomic $\cZ(p)$--rm
\cite[Corollary 7.6]{RS18}. 
In the proof they
use their variant of B\'ezout's theorem on the number of intersection points of two plane algebraic curves \cite[Theorem 7.3]{RS18}.
Part \eqref{240822-1317-pt4} of Theorem \ref{240822-1317} gives an analogue of \cite[Corollary 7.6]{RS18} for even degree sequences on curves $\cZ(y-q(x))$, $\deg q\geq 3$, while
part \eqref{240822-1317-pt4-odd} of Theorem \ref{240822-1317-odd} improves \cite[Corollary 7.6]{RS18} for curves $\cZ(y-q(x))$, $\deg q\geq 3$, by decreasing the upper bound on the number of atoms needed by
$\lceil\frac{\deg{g}}{2}\rceil$.

\item
\label{191122-1056}
\textbf{Theorem \ref{240822-1317} in case $\deg q=2$:}
 If $\deg q=2$ in Theorem \ref{240822-1317}, then $y=q_2x^2+q_1x+q_0\in \RR[x]$ with $q_2\neq 0$ or equivalently $\frac{1}{q_2}y-q_1x-q_0=x^2$. By applying an alt 
		$\phi(x,y)=(x,\frac{1}{q_2}y-q_1x-q_0)$
	to the sequence $\beta$ we get a sequence $\widetilde \beta$ with the moment matrix $M_k(\widetilde \beta)$ satisfying $Y=X^2$. 
	So it is enough to observe the case of a parabola, which was concretely solved in \cite{CF04} by the use of the 
	FET.
	The technique used in the proof of Theorem \ref{240822-1317} can be used to give an alternative proof of the solution 	from \cite{CF04} and also obtain a new solvability condition (see Theorem \ref{090622-1549} below). 
	This condition will be essentially 
	used in the solution of TMP on the cubic reducible curve $y(y-x^2)=0$ in our forthcoming work, 
	similarly as for the TMP on the union of three parallel lines \cite[Theorem 4.2]{Zal22a}, where we needed such version of the solution to the TMP on the union of two parallel lines \cite[Theorem 3.1]{Zal22a}.
	The upper bound on the number of atoms in a minimal rm
	is $2k+1$ 	
	and this is sharp (e.g., if $M_k$ has only column relations coming from $Y=X^2$ by rg, then it is of rank $2k+1$
	and so every rm 
	must have at least $2k+1$ atoms). 
	So the equivalence $\eqref{240822-1317-pt1}\Leftrightarrow\eqref{240822-1317-pt4}$
	of Theorem \ref{240822-1317} does not extend to $\deg q=2$.
	(The moment $\gamma_{k\deg q-1}$ is not independent from $\beta$ for $\deg q=2$ 
	as opposed to $\deg q>2$ and hence in the last step of the proof below decreasing the number of atoms in 	
	the rm 
 from $k\deg q+1$ to $k\deg q$ cannot be done.)
	Also the equivalence $\eqref{240822-1317-pt1}\Leftrightarrow\eqref{240822-1317-pt2}$ of 
	Theorem \ref{240822-1317} is not true for $\deg q=2$, but we need 
	to replace $k+\deg q-2$ by $k+\deg q-1$ in \eqref{240822-1317-pt2},
	 because we do not get the information about $\gamma_{2k\deg q+1}=\gamma_{4k+1}$ and 
	$\gamma_{2k\deg q+2}=\gamma_{4k+2}$
	from $M_{k+\deg q-2}=M_{k}$ for $\deg q=2$ as opposed to $\deg q>2$.
	However, the equivalence $\eqref{240822-1317-pt1}\Leftrightarrow\eqref{240822-1317-pt3}$ still holds for
	$\deg q=2$ with the argument given in Theorem \ref{090622-1549} below.

\item
\label{191122-1056-pt4}  
	\textbf{Theorem \ref{240822-1317} in case $\deg q\leq 1$:}
	If $\deg q\leq 1$ in Theorem \ref{240822-1317}, then $q(x)=ax+by+c$, $a,b,c\in \RR$.
	If $(a,b)\neq (0,1)$, then the following statements are equivalent:
	\begin{enumerate}
		\item
			\label{271122-1836-pt1}
				$\beta$ has a $K$--rm. 
		\item
			\label{271122-1836-pt2}
		 		$\beta$ has a $s$--atomic $K$--rm  
				for some $s$ satisfying 
 					$$\Rank M_k\leq s\leq k+1.$$
		\item
			\label{271122-1836-pt3}
				$M_k$ satisfies $Y=aX+bY+c$, is psd and rg. 
	\end{enumerate}
	The equivalence $\eqref{271122-1836-pt1}\Leftrightarrow\eqref{271122-1836-pt3}$
	follows from \cite[Proposition 3.11]{CF08}, while the 
	equivalence $\eqref{271122-1836-pt1}\Leftrightarrow\eqref{271122-1836-pt2}$
	follows from the solution \cite[Theorem 3.9]{CF91}
	to the $\RR$--TMP, which corresponds to $(M_k)|_{\{1,X,\ldots,X^k\}}$.
	Namely, if the atoms $x_1,\ldots,x_m$ represent $\beta_{i,0}$, $i=0,\ldots,2k$,
	then the atoms $(x_i,y_i)$, where $y_i=\frac{1}{1-b}(ax_i+c)$ will represent $\beta$
	if $b\neq 1$. If $b=1$ and $a\neq 0$, then we change the roles of $x$ and $y$ in the argument above. 
	If $b=1$ and $a=0$, then
	$y=q(x)$ only makes sense if $c=0$, but in this case there are no relations in the moment matrix
 	and for $k>2$ 
	the solution to the TMP is not known (for $k=2$ the solution is known \cite{FN10,CY16}).
	\item
	\label{120323-0031}
	\textbf{Uniqueness and description of all solutions in Theorems \ref{240822-1317} and \ref{240822-1317-odd}:}
	These questions are nontrivial, being equivalent to the descriptions of psd completions of a partially defined Hankel matrix $A_\gamma$,
	where $\gamma\in \RR^{2k\ell+1}$ is a partially defined univariate sequence, i.e.,  $\gamma_t$ is defined by the formula \eqref{250822-1207}
	below. Namely, the original sequence $\beta$ determines only those $\gamma_t$ for which $t\modulo \ell+\lfloor \frac{t}{\ell}\rfloor\leq 2k$ (resp.\ $2k-1$)
	in the even (resp.\ odd) case.
	However, in a very special case $y=x^3$ the structure of the missing entries is simple enough (only $\gamma_{6k-1}$ is missing) to answer these questions.
	The description is not explicitly stated in \cite[Theorem 3.1]{Zal21}, which solves $y=x^3$ using the univariate reduction technique, 
	but one of the main steps of the proof is \cite[Lemma 2.11]{Zal21}, which actually describes
	all psd completions and among them there are one or two minimal ones (in terms of the rank). However, already for $y=x^4$ describing all psd completions concretely
	and among them minimal ones does not seem to be 
	possible
	due to the structure of missing entries of $A_\gamma$; see Example \ref{121122-1804} below.
\item
	\label{120323-0036}
	\textbf{Complexity of checking conditions in \eqref{240822-1317-pt3} of Theorems \ref{240822-1317} and \ref{240822-1317-odd}:}
	Checking if $Y=q(X)$ is a column relation only requires checking whether the corresponding vector is in the kernel fo $M_k$,	
	while the existence of a psd extension $M_{k+\deg q-1}$ of $M_k$ is a feasibility question of a semidefinite programe (SDP) in the variables $\beta_{i,j}$ 
	with $2k<i+j\leq 2(k+\deg q-1)$, $i,j\in \ZZ_+$, i.e., 
	a SDP with matrices of size $\binom{k+\deg q+1}{2}$ in 
		$\binom{2(k+\deg q)}{2}-\binom{2(k+1)}{2}=(\deg q-1)(4k+2\deg q+1)$ 
	variables.
	The complexity of the SDP feasibility question is still unknown (see \cite{Ram97,KS12}), but for a fixed number of variables or size of matrices it 
	has polynomial time complexity \cite{PK97}.
	By Theorem \ref{121122-1341} below, the feasibility question in \eqref{240822-1317-pt3} of Theorems \ref{240822-1317} and  \ref{240822-1317-odd} is equivalent to the 
	feasibility question of a smaller SDP, i.e., the matrices are of size $k\deg q+2$ and the number of variables is $\frac{1}{2}(\deg q-2)(\deg q-1)+2$.
\end{enumerate}
\end{remark}

\begin{remark}[Basic idea of the proof of the implication $\eqref{240822-1317-pt2}\Rightarrow\eqref{240822-1317-pt4}$ of Theorem \ref{240822-1317}]
\label{110323-1045}
The main steps are the following:
\begin{enumerate}
	\item
		\label{260223-1844}
	 		Due to the column relation $Y=q(X)$ satisfied by $M_k$ and $M_k$ being rg, 
			the column space of $M_{k+\ell-2}$ is spanned by the set $\cB$ of columns, 
			indexed by monomials
				$Y^{i}X^j$, 
					where $i=0,\ldots,k$, $j=0,\ldots,\deg q-1$, and $i+j\leq k+\ell-2$.
	\item 
		\label{260223-1845}
			Writing $q(x)=\sum_{i=0}^\ell q_i x^i$, the point \eqref{260223-1844} implies, 
			on the level of the sequence $\beta_{i,j}$, 
			the following linear relations:
			\begin{equation}
				\label{260223-1853}
					\beta_{i,j}=q_\ell\beta_{i+\ell,j-1}+q_{\ell-1}\beta_{i+\ell-1,j-1}+\ldots+q_0\beta_{i,j-1},
			\end{equation}
			where $i\in \ZZ_+$, $j\in \NN$ and $i+j\leq 2(k+\ell-2)$.
	\item 
		\label{260223-1857}
			Using relations \eqref{260223-1853} one can show that the sequence 
			\begin{equation*}
				\label{260223-1922}	
					\beta_{i,j}, \qquad i,j\in \ZZ_+, \; i+j\leq 2(k+\ell-2),
			\end{equation*}
			can be uniquely parametrized by 
			\begin{equation*}
				\label{260223-1924}
					\gamma_i:=\beta_{i,0}, \qquad i=0,\ldots,2k\ell+2.
			\end{equation*}
			(See \eqref{250822-1207} and Claim 4 in the proof of Theorem \ref{240822-1317} below.)
	\item 
		\label{260223-1921}
			The next step is to notice that the original sequence $\beta_{i,j}$,
			$i,j\in \ZZ_+, i+j\leq 2k$ 
			has a $K$--rm if and only if the univariate sequence
			$\gamma_0,\gamma_1,\ldots,\gamma_{2k\ell}$
			has a $\RR$--rm.
			(See Claim 5 in the proof of Theorem \ref{240822-1317} below.)
	\item 
		\label{260223-1926}
			Now one can use the solution to the $\RR$--TMP due to Curto and Fialkow \cite[Theorem 3.9]{CF91}
			and 
			derive the solution to the $K$--TMP.
			Namely, the sequence 
				$\gamma^{(0,2k\ell)}=(\gamma_0,\gamma_1,\ldots,\gamma_{2k\ell})\in \RR^{2k\ell+1}$
			has a $(\Rank \gamma^{(0,2k\ell)})$--atomic $\RR$--rm iff
			a Hankel matrix 
			$A_{\gamma^{(2k\ell+2)}}$
			is psd for some extension
			$\gamma^{(2k\ell+2)}=(\gamma,\gamma_{2k\ell+1},\gamma_{2k\ell+2})\in \RR^{2k\ell+3}$
			of $\gamma^{(0,2k\ell)}$.
			Observing that the restriction $(M_{k+\ell-2})|_{\cB}$ is equal to 
				$PA_{\gamma^{(2k\ell+2)}}P^T$
			for a certain matrix $P$ 
			(See Claim 6 in the proof of Theorem \ref{240822-1317} below.),
			 translates the condition of $A_{\gamma^{(2k\ell+2)}}$ being psd,
			to $(M_{k+\ell-2})|_{\cB}$ (and hence $M_{k+\ell-2}$) being psd.
	\item
		\label{260223-1948}
			The last step is to decrease the number of atoms in the solution to the 
			$\RR$--TMP for $\gamma^{(0,2k\ell)}$
			by 1 in case $A_{\gamma^{(0,2k\ell)}}$ has full rank, i.e., $\Rank A_{\gamma^{(0,2k\ell)}}=k\ell+1$. 
			This can be achieved in the following three steps:
			\begin{enumerate}
				\item 
					\label{230303-0025}
					Applying an appropriate alt $\phi$ to $\beta,$ one can assume that the coefficient $q_{\ell-1}$ of $q(x)$ is equal to 0.
					(See Claim 1 in the proof of Theorem \ref{240822-1317} below.) 
				\item
					Due to \eqref{230303-0025} none of the moments of the original sequence $\beta_{i,j}$, $i,j\in \ZZ_+$, $i+j\leq 2k$, 
					depends on $\gamma_{2k\ell-1}$.
				\item One can replace $\gamma_{2k\ell-1}$ with $\widetilde \gamma_{2k\ell-1}$ to obtain a sequence
					$\widetilde\gamma^{(0,2k\ell)}$ with a $(k\ell)$--atomic $\RR$--rm.
			\end{enumerate}
\end{enumerate}
\end{remark}

\begin{proof}[Proof of Theorem \ref{240822-1317}] 
	Before starting a proof we do an alt $\phi$ which will be used in the proof of the implication
	$\eqref{240822-1317-pt2}\Rightarrow\eqref{240822-1317-pt4}$ to justify in an easier way
	that the upper bound in 
	\eqref{240822-1317-pt4} is $k\deg q$ instead of $k\deg q+1$.
	We write $\ell:=\deg q$ and let $q(x)=\sum_{i=0}^{\ell}q_i x^i$,
	where $q_\ell\neq 0$ and each $q_{i}\in \RR$.\\

	\noindent\textbf{Claim 1.}
		We may assume that $q_{\ell-1}=0$.\\
	
	\noindent \textit{Proof of Claim 1.}
	Defining $\phi:\RR^2\to\RR^2$ by
		$\displaystyle \phi(x,y)=\Big(x+\frac{q_{\ell-1}}{\ell q_{\ell}},y\Big)=:(\tilde x, y)$, note that the
	relation $y=q(x)$ becomes
		$$
			y
			=q\Big(\tilde x-\frac{q_{\ell-1}}{\ell q_{\ell}}\Big)
			=\sum_{i=0}^{\ell}q_i \Big(\tilde x-\frac{q_{\ell-1}}{\ell q_{\ell}}\Big)^i
			=q_\ell \tilde x^\ell + \Big(\underbrace{-q_\ell \cdot \ell \frac{q_{\ell-1}}{\ell q_{\ell}} +q_{\ell-1}}_{=0} \Big)\tilde x ^{\ell-1}
			+\sum_{i=0}^{\ell-2} \tilde{q}_{i}\tilde x^i,
		$$
	for some $\tilde{q}_{0},\ldots, \tilde{q}_{\ell-2}\in \RR$.
	Since the solution of the $K$--TMP is invariant under applying $\phi$ 
	by Proposition \ref{251021-2254}, the conclusion of Claim 1 follows.
	\hfill$\blacksquare$\\

	Now we start the proof of the theorem.
	The implications 
		$\eqref{240822-1317-pt1}\Rightarrow\eqref{240822-1317-pt3}$ 
	and
		$\eqref{240822-1317-pt4}\Rightarrow\eqref{240822-1317-pt1}$ 	
	are trivial.
	The implication $\eqref{240822-1317-pt3}\Rightarrow\eqref{240822-1317-pt2}$ is \cite[Theorem 3.14]{CF96}.
	It remains to prove the implication $\eqref{240822-1317-pt2}\Rightarrow\eqref{240822-1317-pt4}$.
	Assume that $M_k$ admits a psd, rg extension $M_{k+\ell-2}$.
Let	
	\begin{equation}
		\label{111122-2040}
			\cB=\left\{1,x,\ldots,x^{\ell-1},y,yx,\ldots,yx^{\ell-1},\ldots,y^{k-1},\ldots,y^{k-1}x^{\ell-1},y^{k},y^kx\right\}
	\end{equation}
	be a set of monomials and $V$ a vector subspace in $\RR[x,y]_{k+\ell-2}$ generated by the set $\cB$.
	Since $M_{k+\ell-2}$ satisfies $X^iY^j=X^iq(X)^j$ for every $i,j\in \ZZ_+$
	such that $i+j\ell\leq k+\ell-2$, it follows that the columns from $\cB$	
	span $\cC(M_{k+\ell-2})$.
	Let
		$p(x,y)=\sum_{i,j}p_{ij}x^iy^j\in V$
	be a polynomial and 
		$\widehat p$
	a vector of its coefficients ordered in the basis $\cB$.
	Before we define a univariate polynomial $g_p(x)$ corresponding to $p(x,y)$, 
	we prepare some computations.
	We have that
	\begin{align}\label{240822-1441}
	\begin{split}
		x^i(q(x))^j
		&=
		x^i\Big(\sum_{0\leq i_1,\ldots,i_j\leq \ell} q_{i_1}q_{i_2}\cdots q_{i_j} x^{i_1+\ldots+i_j}\Big)\\
		&=
		\sum_{p=0}^{j\ell}
		\Big(\sum_{\substack{0\leq i_1,\ldots,i_j\leq \ell,\\ i_1+\ldots+i_j=p}}
				q_{i_1}q_{i_2}\cdots q_{i_j}\Big)
		x^{i+p}\\
		&=
		\sum_{s=i}^{i+j\ell} q_{i,j,s}x^s,
	\end{split}
	\end{align}	
	for all $i,j\in \ZZ_+$,
	where 
	\begin{equation}
		\label{111122-1939}
		\displaystyle q_{i,j,s}=
			\left\{
			\begin{array}{rr}
				\displaystyle\sum_{\substack{0\leq i_1,\ldots,i_j\leq \ell,\\ i_1+\ldots+i_j=s-i}}
				q_{i_1}q_{i_2}\ldots q_{i_j},&		\text{if }i\leq s\leq i+j\ell,\\
				0,&		\text{otherwise}.
			\end{array}
			\right.
	\end{equation}

	\noindent Later on we will need the following observation about the numbers $q_{i,j,s}$.\\

	\noindent\textbf{Claim 2.}
	Let $i_1,i_2,j_1,j_2,s\in \ZZ_+$. Then
	\begin{equation}
		\label{091122-2011}
			q_{i_1+i_2,j_1+j_2,s}=\sum_{t=i_1}^{s}q_{i_1,j_1,t}q_{i_2,j_2,s-t}.
	\end{equation}
	
	\noindent \textit{Proof of Claim 2.}
	We write $m_1:=i_1+i_2$ and $m_2:=i_1+i_2+(j_1+j_2)\ell$.
	We separate two cases: $s\in \{m_1,m_1+1,\ldots, m_2\}$ and $s\notin \{m_1,m_1+1,\ldots,m_2\}$.\\

	\noindent Case 1: $s\in  \{m_1,m_1+1,\ldots, m_2\}$. 
	We have that	
	\begin{align*}
		q_{i_1+i_2,j_1+j_2,s}
			&=^1\sum_{\substack{0\leq k_1,\ldots,k_{j_1+j_2}\leq \ell,\\ k_1+\ldots+k_{j_1+j_2}=s-i_1-i_2}}
				q_{k_1}q_{k_2}\cdots q_{k_{j_1}}q_{k_{j_1+1}}\cdots q_{k_{j_1+j_2}}\\
			&=^2\sum_{t=i_1}^{s}
				\sum_{\substack{0\leq k_1,\ldots,k_{j_1}\leq \ell,\\ k_1+\ldots+k_{j_1}=t-i_1}}
				\sum_{\substack{0\leq k_{j_1+1},\ldots,k_{j_2}\leq \ell,\\ k_{j_1+1}+\ldots+k_{j_2}=s-t-i_2}}
				q_{k_1}q_{k_2}\cdots q_{k_{j_1}}q_{k_{j_1+1}}\cdots q_{k_{j_1+j_2}}\\
			&=^3\sum_{t=i_1}^{s}
				\Big(\sum_{\substack{0\leq k_1,\ldots,k_{j_1}\leq \ell,\\ k_1+\ldots+k_{j_1}=t-i_1}}
					q_{k_1}q_{k_2}\cdots q_{k_{j_1}}\Big)
				\Big(\sum_{\substack{0\leq k_{j_1+1},\ldots,k_{j_2}\leq \ell,\\ k_{j_1+1}+\ldots+k_{j_2}=s-t-i_2}}
					q_{k_{j_1+1}}\cdots q_{k_{j_1+j_2}}\Big)\\
			&=^4\sum_{t=i_1}^{s} q_{i_1,j_1,t}q_{i_2,j_2,s-t},\\
	\end{align*}
	where 
		the first equality follows by definition \eqref{111122-1939} of $q_{i_1+i_2,j_1+j_2,s}$,
		in the second we decomposed the sum into three sums,
		in the third we used independence of the inner two sums,
	while
		the last equality follows by definitions \eqref{111122-1939} of $q_{i_1,j_1,t}$ and $q_{i_2,j_2,s-t}$.\\

	\noindent Case 2: $s\notin \{m_1,m_1+1,\ldots, m_2\}$. For $s>m_2$ we have
	$q_{i_1+i_2,j_1+j_2,s}=0$ and 
		$$
		\sum_{t=i_1}^{s}q_{i_1,j_1,t}q_{i_2,j_2,s-t}
		=
		\sum_{t=i_1}^{i_1+j_1\ell}q_{i_1,j_1,t}\underbrace{q_{i_2,j_2,s-t}}_{\substack{=0, \text{ since}\\[0.2em]s-t>i_2+j_2\ell}}
		+
		\sum_{t=i_1+j_1\ell+1}^{s}\underbrace{q_{i_1,j_1,t}}_{\substack{=0, \text{ since}\\[0.2em] t>i_1+j_1\ell}}q_{i_2,j_2,s-t}
		=0,
		$$	
	which implies that \eqref{111122-1939} holds. 
	Similarly, for $s<m_1$ we again have $q_{i_1+i_2,j_1+j_2,s}=0$ and 
		$
		\sum_{t=i_1}^{s}q_{i_1,j_1,t}q_{i_2,j_2,s-t}
		=0,
		$
	since $q_{i_2,j_2,s-t}=0$ for every $t$ due to $s-t<i_2$.
	Also in this case \eqref{111122-1939} holds. 
	\hfill$\blacksquare$\\

	Now we define a univariate polynomial $g_p(x)$ corresponding to $p(x,y)$ 
	by
		$$g_p(x):=p(x,q(x))
			=
			\sum_{i,j}p_{ij} \sum_{s=i}^{i+j\ell} q_{i,j,s}x^s
			=:
			\sum_{s=0}^{k\ell+1} g_{p,s} x^s
			\in 
		\RR[x]_{k\ell+1},$$
	where we used \eqref{240822-1441} in the second equality.
%
	Let $\widehat{g_p}$ be its vector of coefficients in the basis 
	\begin{equation}
		\label{171122-2049}
			\cB_1=\{1,x,\ldots,x^{k\ell+1}\}.
	\end{equation}
	The following claim expresses $\widehat{g_p}$ by $\widehat p$.\\

	\noindent\textbf{Claim 3.}
		It holds that 
		\begin{equation}
			\label{111122-2024}
				\widehat{g_p}=P^T\widehat p,
		\end{equation}
		where
$$
		P=
		\begin{pmatrix}
		I_{\ell} & 0       & \cdots & \cdots & 0 & 0 \\
		P[1,0] & P[1,1] & 0 &   & 0 & \vdots\\
		\vdots & \ddots &  \ddots & \ddots & \vdots & \vdots\\
		\vdots & &  \ddots & \ddots & 0 & \vdots\\
		P[k-1,0] & P[k-1,1] & \cdots & \cdots &P[k-1,k-1] & 0\\
		Q[0] & Q[1] & \cdots & \cdots & Q[k-1] & Q[k]
		\end{pmatrix}\in \RR^{(k\ell+2)\times (k\ell+2)}
		$$
	and
	\begin{align*}
		P[c,d]
		&=
		\begin{pmatrix}
		q_{0,c,d\ell} & q_{0,c,d\ell+1}  & \cdots & q_{0,c,d\ell+\ell-1} \\
		q_{1,c,d\ell} & q_{1,c,d\ell+1}  & \cdots & q_{1,c,d\ell+\ell-1} \\
		\vdots & \cdots & \cdots & \vdots \\
		q_{\ell-1,c,d\ell} & q_{\ell-1,c,d\ell+1}  & \cdots & q_{\ell-1,c,d\ell+\ell-1} \\
		\end{pmatrix}
		\in \RR^{\ell\times \ell}\quad \text{for each }c,d,\\[0.5em]
		Q[d]
		&=
		\begin{pmatrix}
		q_{0,k,d\ell} & q_{0,k,d\ell+1}  & \cdots & q_{0,k,d\ell+\ell-1} \\
		q_{1,k,d\ell} & q_{1,k,d\ell+1}  & \cdots & q_{1,k,d\ell+\ell-1}
		\end{pmatrix}
		\in \RR^{2\times \ell}\quad \text{for }d=0,\ldots,k-1,\\[0.5em]
		Q[k]
		&=
		\begin{pmatrix}
		(q_{\ell})^k & 0\\
		q_{1,k,k\ell} & (q_{\ell})^k
		\end{pmatrix}
		\in \RR^{2\times 2}.
	\end{align*}	
	
	\noindent \textit{Proof of Claim 3.} 
	We write 
		$\vec{v}_{x,y}$ for the vector of monomials $x^iy^j$ from the basis $\cB$ (see \eqref{111122-2040})
	and
		 $\vec{v}_x$ for the vector of monomials from the basis $\cB_1$ (see \eqref{171122-2049}).
	We have that 
		$p(x,y)=\big(\vec{v}_{x,y}\big)^T\widehat p$
	and 
		$g_p(x)=\big(\vec{v}_x\big)^T \widehat g_p$.
	By \eqref{240822-1441} it follows that $\vec{v}_{x,y}=P\vec{v}_{x}$. 
	Hence, the definition of $g_p$ implies that
		$g_p(x)=\big(P\vec{v}_x\big)^T\widehat p=\big(\vec{v}_x\big)^TP^T\widehat p.$
	Thus, $\widehat g_p=P^T\widehat p$, which proves Claim 3.	\hfill$\blacksquare$\\

	Note that 
	\begin{equation}
		\label{111122-2306}
			q_{i,j,i+j\ell}=(q_{\ell})^j\neq 0
	\end{equation}
	and hence we can express $x^{i+j\ell}$ from \eqref{240822-1441} by the formula
	\begin{align}\label{240822-1500}
	\begin{split}
		x^{i+j\ell}=\frac{1}{(q_\ell)^j}\Big(x^i(q(x))^j-\sum_{s=0}^{i+j\ell-1}q_{i,j,s}x^{s}\Big).
	\end{split}
	\end{align}
	We define two univariate sequences
	$$
		\gamma:=\gamma^{(0,2k\ell)}=(\gamma_0,\gamma_1,\ldots,\gamma_{2k\ell})\in \RR^{2k\ell+1},\quad
		\widetilde\gamma:=\gamma^{(2k\ell+2)}=(\gamma,\gamma_{2k\ell+1},\gamma_{2k\ell+2})\in \RR^{2k+3},
	$$
	recursively for $t=0,1,\ldots,2k\ell+2$ by the formula
	 \begin{equation}
		\label{250822-1207}
		\gamma_{t}
				=\frac{1}{(q_\ell)^{\lfloor \frac{t}{\ell}\rfloor}}\Big(\beta_{t\modulo \ell,\lfloor \frac{t}{\ell}\rfloor}-
				\sum_{s=0}^{t-1}q_{t\modulo \ell,\lfloor \frac{t}{\ell}\rfloor,s}\cdot \gamma_{s}\Big).
	\end{equation}
	Note that $t\modulo \ell+\lfloor\frac{t}{\ell}\rfloor\leq \ell-1+2k$ (here we used that $\ell\geq 3$ and thus $\lfloor\frac{t}{\ell}\rfloor\leq 2k$)
	and so	$\beta_{t\modulo \ell,\lfloor \frac{t}{\ell}\rfloor}$ is well-defined being an element of the matrix $M_{k+\ell-2}$
	(since $2(k+\ell-2)\geq \ell-1+2k$ for $\ell\geq 3$).\\

Note that we defined $\gamma_t$ by \eqref{250822-1207} only using $\beta_{i,j}$ with $i<\ell$.
The following claim proves we could define $\gamma_t$ using any $i,j$ with $t=i+j\ell$.\\

\noindent\textbf{Claim 4.} 
		Let $t\in \{0,1,\ldots,2k\ell+2\}$ and $t=i+j\ell$ with $i,j\in\ZZ_+$, $i+j\leq 2(k+\ell-2)$.
		Then 
		\begin{equation}
			\label{111222-2258}
			\gamma_{t}
				=\frac{1}{(q_\ell)^{j}}\Big(\beta_{i,j}-\sum_{s=0}^{t-1}q_{i,j,s}\cdot \gamma_{s}\Big).
		\end{equation}

\noindent \textit{Proof of Claim 4.} 
	If $i<\ell$, then \eqref{111222-2258} follows from definition \eqref{250822-1207} of $\gamma_t$.
	Let $i_0\geq \ell$. Assume \eqref{111222-2258} is true for every $i=0,1,\ldots,i_0-1$
	and $j\in\NN_0$ such that $i+j\leq 2(k+\ell-2)$. 
	We will prove it is true for $i_0$ and every $j\in\NN_0$ such that $i_0+j\leq 2(k+\ell-2)$.
	We have that:
	\begin{align}
	\label{111222-2321}
	\begin{split}
		\beta_{i_0,j}
		&=\frac{1}{q_\ell}\Big(\beta_{i_0-\ell,j+1}-\sum_{s=0}^{\ell-1}q_s\beta_{i_0-\ell+s,j}\Big)\\
		&=
		\frac{1}{q_\ell}\Big(\sum_{s=i_0-\ell}^{i_0+j\ell}q_{i_0-\ell,j+1,s}\gamma_s
		-\sum_{s=0}^{\ell-1}q_s\Big(\sum_{u=i_0-\ell+s}^{i_0+s+(j-1)\ell}q_{i_0-\ell+s,j,u}\gamma_u\Big)\Big)
	\end{split}
	\end{align}
	where in the first equality we used that $\beta^{(2(k+\ell-2))}=(\beta_{i,j})_{i,j\in \ZZ_+,i+j\leq 2(k+\ell-2)}$ is rg and $Y=q(X)$ is a column relation, 
	and in the second we used the induction hypothesis together with the fact that 
	$q_{i,j,s}=0$ if $s<i$ and $q_{i_0-\ell,j+1,i_0+j\ell}=(q_{\ell})^{j+1}$.
	Note that
	\begin{align}
	\label{111222-2322}
	\begin{split}
	&\sum_{s=0}^{\ell-1}q_s
		\Big(\sum_{u=i_0-\ell+s}^{i_0+s+(j-1)\ell}q_{i_0-\ell+s,j,u}\gamma_u\Big)=
	\sum_{s=0}^{\ell-1}q_s
		\Big(\sum_{u=i_0-\ell+s}^{i_0+s+(j-1)\ell}q_{i_0-\ell,j,u-s}\gamma_u\Big)\\
	&=\sum_{s=0}^{\ell-1}
		\Big(\sum_{u=i_0-\ell+s}^{i_0+s+(j-1)\ell}q_sq_{i_0-\ell,j,u-s}\gamma_u\Big)\\
	&=\sum_{u=i_0-\ell}^{i_0+j\ell-1}q_{i_0-\ell,j+1,u}\gamma_u-	
		\sum_{u=i_0-\ell}^{i_0+j\ell-1}q_\ell q_{i_0-\ell,j,u-\ell}\gamma_u
	\end{split}
	\end{align}
	where in the first equality we used that $q_{i_0-\ell+s,j,u}=q_{i_0-\ell,j,u-s}$ by definition \eqref{111122-1939},
	in the second we moved $q_s$ inside the inner sum,
	in the third we used that for a fixed $u$, as $s$ runs from 0 to $\ell-1$, the coefficients at $\gamma_u$ run over all terms
	$q_{i_1}\cdots q_{i_{j+1}}$ such that $i_1+\ldots+i_{j+1}=u-i_0+\ell$ except those where $q_{i_1}=q_\ell$.
	But if $q_{i_1}=q_{\ell}$, all terms $q_{i+2}\cdots q_{i_{j+1}}$ such that $i_2+\ldots+i_{j+1}=u-i_0$ sum up to $q_{i_0-\ell,j,u-\ell}$.
	Using \eqref{111222-2322} in \eqref{111222-2321} we get
	\begin{align}
	\label{111222-2338}
	\begin{split}
		\beta_{i_0,j}
		&=\frac{1}{q_\ell}\Big(q_{i_0-\ell,j+1,i_0+j\ell}\gamma_{i_0+j\ell}+
			\sum_{u=i_0-\ell}^{i_0+j\ell-1}q_\ell q_{i_0-\ell,j,u}\gamma_u\Big)\\
		&=\frac{1}{q_\ell}\Big((q_\ell)^{j+1}\gamma_{i_0+j\ell}+
			\sum_{u=i_0-\ell}^{i_0+j\ell-1}q_\ell q_{i_0-\ell,j,t}\gamma_u\Big)\\
		&=(q_\ell)^{j}\gamma_{i_0+j\ell}+
			\sum_{u=i_0-\ell}^{i_0+j\ell-1}q_{i_0-\ell,j,u}\gamma_u.
	\end{split}
	\end{align}
	where in the second equality we used that 
	$q_{i_0-\ell,j+1,i_0+j\ell}=(q_\ell)^{j+1}$	
	and in the last we put $\frac{1}{q_\ell}$ inside the bracket.
	But \eqref{111222-2338} is \eqref{111222-2258} for $t=i_0+j\ell$.
	\hfill $\blacksquare$\\
	
	By the following claim solving the $K$--TMP for $\beta$ is equivalent to solving the $\RR$--TMP for $\gamma$.\\

	\noindent\textbf{Claim 5.} 
		Let $u\in \NN$.
		A sequence 
			$\gamma$ admits a $u$--atomic $\RR$--rm
		if and only if 
			$\beta$ admits a $u$--atomic $K$--rm.\\

\noindent \textit{Proof of Claim 5.} 
First we prove the implication $(\Rightarrow)$.
Let 	
	$x_1,\ldots, x_u,$ 
be the atoms in the $\RR$--rm for $\gamma$ with the corresponding densities 
	$\rho_1,\ldots,\rho_u$.
We will prove that the atoms $(x_1,q(x_1)),\ldots, (x_u,q(x_u))$ with densities $\rho_1,\ldots,\rho_p$ are the $K$--rm for $\beta$.
We use induction on the index $i$ in $\beta_{i,j}$, where $i+j\leq 2k$. 
For $i<\ell$ and any $j$ such that $i+j\leq 2k$
we have that
	\begin{align*}
	  \beta_{i,j}
		&=^1
		(q_\ell)^j \gamma_{i+j\ell}+\sum_{s=0}^{i+j\ell-1}q_{i,j,s}\gamma_s\\
		&=^2
		(q_\ell)^j \Big(\sum_{p=0}^u\rho_p (x_p)^{i+j\ell}\Big)+
		\sum_{s=0}^{i+j\ell-1}q_{i,j,s}\Big(\sum_{p=0}^u\rho_p (x_p)^{s}\Big)\\
		&=^3\sum_{p=0}^u\rho_p (q_\ell)^j (x_p)^{i+j\ell}+
		\sum_{p=0}^u\Big(\rho_p  \sum_{s=0}^{i+j\ell-1}q_{i,j,s}(x_p)^{s}\Big)\\
		&=^4\sum_{p=0}^u\Big(\rho_p \Big((q_\ell)^j (x_p)^{i+j\ell}+\sum_{s=0}^{i+j\ell-1}q_{i,j,s}(x_p)^{s}\Big)\Big)\\
		&=^5\sum_{p=0}^u\Big(\rho_p (x_p)^i (q(x_p))^j\Big),
	\end{align*}
where we used \eqref{250822-1207} with $t=i+j\ell$ in the first equality noticing that
	$$i+j\ell=i+j+j(\ell-1)\leq 2k+2k(\ell-1)\leq 2k\ell,$$ 
implying well-definedness of $\gamma_s$ by $s$ being bounded above by $2k\ell$, 
the definitions of $\rho_p$, $x_p$ in the second equality,
we interchanged the order of summation in the third and fourth equalities
and in the last we used \eqref{240822-1500} for $x=x_p$.
So the atoms $(x_1,q(x_1)),\ldots, (x_u,q(x_u))$ with densities $\rho_1,\ldots,\rho_p$ indeed represent $\beta_{i,j}$
for $i<\ell$ and any $j$ such that $i+j\leq 2k$.
We now assume that this holds for all $i=0,\ldots,m$ and $j$ such that $i+j\leq 2k$, where $m\geq \ell-1$ 
and prove it for $i=m+1$ and any $j\leq 2k-i$.
We have that
	\begin{align*}
		\beta_{m+1,j}
		&=^1
		\frac{1}{q_\ell}\big(\beta_{m+1-\ell,j+1}-\sum_{s=0}^{\ell-1}q_s \beta_{m+1-\ell+s,j} \big)\\
		&=^2
		\frac{1}{q_\ell}\Big(\Big(\sum_{p=0}^u \rho_p(x_{p})^{m+1-\ell} (q(x_p))^{j+1}\Big)-
			\sum_{s=0}^{\ell-1}q_s \Big(\sum_{p=0}^u \rho_p(x_{p})^{m+1-\ell+s} (q(x_p))^{j}\Big) \big)\\
		&=^3	
		\frac{1}{q_\ell}\sum_{p=0}^u\Big(\rho_p\Big( (x_{p})^{m+1-\ell} (q(x_p))^{j+1}-\sum_{s=0}^{\ell-1}q_s (x_{p})^{m+1-\ell+s} (q(x_p))^{j}\Big) \Big)\\
		&=^4	
		\frac{1}{q_\ell}\sum_{p=0}^u\Big(\rho_p(x_{p})^{m+1-\ell} (q(x_p))^{j}\Big( q(x_p)-\sum_{s=0}^{\ell-1}q_s (x_{p})^{s}\Big) \Big)\\
		&=^5	
		\sum_{p=0}^u\Big(\rho_p (x_{p})^{m+1} (q(x_p))^{j}\Big),
	\end{align*}
	where we used that $\beta$ is rg in the first equality, the induction hypothesis in the second,
	in the third we interchanged the order of summation,
	factored out $(x_{p})^{m+1-\ell} (q(x_p))^{j}$ in the fourth 
	and 
	in the last we used that 
		$q(x)-\sum_{s=0}^{\ell-1}q_s x^{s}=q_\ell x^\ell$ 
	by definition of $q$.
	This proves the implication $(\Rightarrow)$.\\

	It remains to prove the implication $(\Leftarrow)$.
	Let 	
		$(x_1,q(x_1)),\ldots, (x_u,q(x_u))$	
	be the atoms in the $K$--rm for $\beta$ with the corresponding densities 
		$\rho_1,\ldots,\rho_u$.
	We will prove that the atoms 
		$(x_1,\ldots,x_u)$ 
	with densities 
		$\rho_1,\ldots,\rho_p$ 
	are the $\RR$--rm for $\gamma$.
	We use induction on the index $t$ in $\gamma_t$. For $t=0$  the claim is trivial,
	since $\gamma_0=\beta_{0,0}=\sum_{p=0}^u \rho_u.$
We now assume that the claim holds for all $t-1$ with $0\leq t-1\leq 2k\ell-1$ 
and prove it for $t$.
We have that
	\begin{align*}
	  \gamma_{t}
		&=^1\frac{1}{(q_\ell)^{\lfloor \frac{t}{\ell}\rfloor}}\Big(\beta_{t\modulo \ell,\lfloor \frac{t}{\ell}\rfloor}-
				\sum_{s=0}^{t-1}q_{t\modulo \ell,\lfloor \frac{t}{\ell}\rfloor,s}\cdot \gamma_{s}\Big)\\
		&=^2\frac{1}{(q_\ell)^{\lfloor \frac{t}{\ell}\rfloor}}\Big(
				\sum_{p=0}^u\rho_p (x_p)^{t\modulo \ell} (q(x_p))^{\lfloor \frac{t}{\ell}\rfloor}	
				-
				\sum_{s=0}^{t-1}q_{t\modulo \ell,\lfloor \frac{t}{\ell}\rfloor,s}\cdot 
				\Big(\sum_{p=0}^u \rho_p (x_p)^{s}\Big)\Big)\\
		&=^3\frac{1}{(q_\ell)^{\lfloor \frac{t}{\ell}\rfloor}}
				\sum_{p=0}^u\Big(
				\rho_p \Big((x_p)^{t\modulo \ell} (q(x_p))^{\lfloor \frac{t}{\ell}\rfloor}	
				-
				\sum_{s=0}^{t-1}q_{t\modulo \ell,\lfloor \frac{t}{\ell}\rfloor,s}\cdot 
				(x_p)^{s}\Big)\Big)\\
		&=^4\frac{1}{(q_\ell)^{\lfloor \frac{t}{\ell}\rfloor}}
				\sum_{p=0}^u\Big(
				\rho_pq_{t\modulo \ell, \lfloor \frac{t}{\ell}\rfloor, t\modulo \ell+\lfloor \frac{t}{\ell}\rfloor\ell} 
				\cdot (x_p)^{t\modulo \ell+\lfloor \frac{t}{\ell}\rfloor\ell}\Big)\\
		&=^5\frac{1}{(q_\ell)^{\lfloor \frac{t}{\ell}\rfloor}}
				\sum_{p=0}^u\Big(
				\rho_p (q_\ell)^{\lfloor \frac{t}{\ell}\rfloor} (x_p)^{t}\Big)\\
		&=^6\sum_{p=0}^u \rho_p (x_p)^{t},
	\end{align*}
where we used the definition \eqref{250822-1207} of $\gamma_t$ in the first equality,
the definitions of $\rho_p$, $x_p$ and the induction hypothesis in the second equality,
we interchanged the order of summation in the third equality,
used \eqref{240822-1441} for $(i,j)=(t\modulo \ell, \lfloor \frac{t}{\ell}\rfloor)$ in the fourth equality
and
the observation \eqref{111122-2306} for $(i,j)=(t\modulo \ell, \lfloor \frac{t}{\ell}\rfloor)$ in
the fifth equality.
This proves the implication $(\Leftarrow)$.
	\hfill $\blacksquare$\\

	Let	$(M_{k+\ell-2})|_{\cB}$ be the restriction of $M_{k+\ell-2}$ to the rows and columns indexed by monomials (capitalized)
	from $\cB$. The following claim gives an explicit connection between $(M_{k+\ell-2})|_{\cB}$ and the Hankel matrix 	
	$A_{\widetilde\gamma}$ of the sequence $\widetilde\gamma$.\\
	
	\noindent\textbf{Claim 6.} 
	We have that
	\begin{equation}
		\label{250822-1104}
			(M_{k+\ell-2})|_{\cB}=PA_{\widetilde\gamma}P^T.
	\end{equation}

	\noindent \textit{Proof of Claim 6.}
	Let 
		$p(x,y)=\sum_{i,j}p_{ij}x^iy^j\in V$ 
	and
		$r(x,y)=\sum_{i,j}r_{ij}x^iy^j\in V$ 
	be polynomials from the vector subspace $V$ 
	and 
		$\widehat p$, $\widehat r$ 
	vectors of their coefficients ordered in the basis $\cB$ (see \eqref{111122-2040}).
	Let $\widetilde \beta:=\beta^{(2(k+\ell-2))}$.
	Then we have
	\begin{align*}
		(\widehat r)^T\left((M_{k+\ell-2})|_{\cB}\right)\widehat p
		&=^1 L_{\widetilde\beta}(pr)
		=L_{\widetilde \beta}\big(\sum_{i_1,i_2,j_1,j_2} p_{i_1j_1}r_{i_2j_2}x^{i_1+i_2}y^{j_1+j_2}\big)\\
		&=^2 \sum_{i_1,i_2,j_1,j_2} p_{i_1j_1}r_{i_2j_2}\beta_{i_1+i_2,j_1+j_2}\\
		&=^3 \sum_{i_1,i_2,j_1,j_2} p_{i_1j_1}r_{i_2j_2}
			\left(\sum_{s=i_1+i_2}^{i_1+i_2+(j_1+j_2)\ell} q_{i_1+i_2,j_1+j_2,s}\gamma_s\right)\\
		&=^4 \sum_{i_1,i_2,j_1,j_2} p_{i_1j_1}r_{i_2j_2}
			\left(\sum_{s=i_1+i_2}^{i_1+i_2+(j_1+j_2)\ell} \left(\sum_{t=i_1}^{s} 
				q_{i_1,j_1,t}q_{i_2,j_2,s-t}\right)\gamma_s\right)\\
	 	&=^5\sum_{i_1,i_2,j_1,j_2} 
			\left(\sum_{s=i_1+i_2}^{i_1+i_2+(j_1+j_2)\ell} \left(\sum_{t=i_1}^{s} 
				p_{i_1j_1} q_{i_1,j_1,t}r_{i_2j_2}q_{i_2,j_2,s-t}\right)\gamma_s\right)\\
	 	&=^6L_{\widetilde\gamma}\left(\sum_{i_1,i_2,j_1,j_2} 
			\left(\sum_{s=i_1+i_2}^{i_1+i_2+(j_1+j_2)\ell} \left(\sum_{t=i_1}^{s} 
				p_{i_1j_1} q_{i_1,j_1,t}r_{i_2j_2}q_{i_2,j_2,s-t}\right)x^s\right)\right)\\
	 	&=^7L_{\widetilde\gamma}\left(\sum_{i_1,i_2,j_1,j_2} 
			\left(\sum_{s=i_1+i_2}^{i_1+i_2+(j_1+j_2)\ell} \sum_{t=i_1}^{s} 
				p_{i_1j_1} q_{i_1,j_1,t}x^t \cdot r_{i_2j_2}q_{i_2,j_2,s-t}x^{s-t}\right)\right)\\
		&=^8L_{\widetilde\gamma}\Big(\sum_{i_1,i_2,j_1,j_2} 
			\Big(\sum_{t=i_1}^{i_1+j_1\ell} p_{i_1j_1} q_{i_1,j_1,t}x^t\Big)
			\Big(\sum_{u=i_2}^{i_2+j_2\ell} r_{i_2j_2} q_{i_2,j_2,u}x^{u}\Big)\\
	 	&=^9L_{\widetilde\gamma}\Big( 
			\Big(\underbrace{\sum_{i_1,j_1}\sum_{t=i_1}^{i_1+j_1\ell} p_{i_1j_1} q_{i_1,j_1,t}x^t}_{g_p(x)}\Big)
			\Big(\underbrace{\sum_{i_2,j_2}\sum_{s=i_1+i_2}^{i_2+j_2\ell} r_{i_2j_2} q_{i_2,j_2,s}x^s}_{g_r(x)}\Big)\Big)\\
		&=^{10}\widehat{g_r}^T A_{\widetilde\gamma}\widehat{g_p}
		= (P^T\widehat{r})^TA_{\widetilde\gamma} (P^T\widehat p)
		=\widehat r^T (P A_{\widetilde\gamma} P^T)\widehat p,
	\end{align*}
	where 
		in the first line we used the correspondence between the moment matrix and the Riesz functional $L_{\widetilde \beta}$,
		the definition $L_{\widetilde \beta}$ in the second,
		Claim 4 in the third,
		Claim 2 in the fourth,
		we moved the factor $p_{i_1j_1}r_{i_2j_2}$ into the inner sum in the fifth,
		used the definition of $L_{\widetilde\gamma}$ in the sixth,
		split $x^s$ into two parts and moved it into the inner sum in the seventh,
		decomposed a double sum into the product of two sums in the eight using that 
		$q_{i_1,j_1,t}$ is nonzero only for $t\leq i_1+j_1\ell$ and $q_{i_2,j_1,u}$ is nonzero only for $u\leq i_2+j_2\ell$,
		decomposed a sum into the product of two sums using independence of the factors
		in the ninth line,
		in the tenth  we used the correspondence between $A_{\widetilde \gamma}$ and the Riesz functional $L_{\widetilde \gamma}$,
	where $\widehat{g_p}$, $\widehat{g_r}$ are the 
		vectors of coefficents of $g_p$ and $g_r$ in the basis $\cB_1$ (see \eqref{171122-2049})
		and 
			also Claim 3.
	Since $p$ and $q$ were arbitrary from $V$, this proves Claim 6.
	\hfill$\blacksquare$\\

	Note that 
	$$P[c,c]
		=
		\begin{pmatrix}
		(q_\ell)^c & 0  & \cdots & \cdots & 0 \\
		q_{1,c,c\ell} & (q_\ell)^c & 0 &  & \vdots \\
		\vdots & \ddots & \ddots & \ddots & \vdots\\
		\vdots &  & \ddots & \ddots & 0 \\
		q_{\ell-1,c,c\ell} & \cdots & \cdots & q_{\ell-1,c,c\ell+\ell-1} & q_{\ell}^c \\
		\end{pmatrix}
		\in \RR^{\ell\times \ell}\quad \text{for }c=1,\ldots,k-1.
	$$
	Since $P$ is a lower triangular matrix with all nonzero diagonal entries, it is invertible.
	Claim 5 implies that
	\begin{equation}
		\label{250822-1105-v2}
			A_{\widetilde\gamma}=P^{-1}\left((M_{k+\ell-2})|_{\cB}\right)(P^{-1})^T.
	\end{equation}
Since $(M_{k+\ell-2})|_{\cB}$ is psd, it follows from \eqref{250822-1105-v2}
	that $A_{\widetilde\gamma}$ is also psd. 
We separate two cases. Either $A_{\widetilde\gamma}$ is pd or $A_{\widetilde\gamma}$ is singular.
	In the first case in particular $A_{\widetilde\gamma}(k\ell)=A_{\gamma}$ is pd, while in the second case
	$A_{\widetilde\gamma}(k\ell)$ is psd and prg by \cite[Theorem 2.6]{CF91}. 
	By \cite[Theorem 3.9]{CF91},
		$\gamma$ 
	admits a $(\Rank A_{\gamma})$--atomic $\RR$--rm.
	Since $\Rank M_k\leq \Rank A_{\gamma}\leq k\ell+1$, 
	using Claim 4 the following holds:
	\begin{enumerate}[(2')]
	\item
		 	\label{240822-1317-pt4'-special-case}
				$\beta$ has a $s$--atomic $K$--rm for some $s$ satisfying 
 				\begin{equation}
					\label{121122-0945-special-case}
						\Rank M_k\leq s\leq A_{\gamma}\leq k\ell+1.
				\end{equation}
	\end{enumerate}
	To obtain \eqref{240822-1317-pt4} of Theorem \ref{240822-1317} we need to decrease the upper bound in
     \eqref{121122-0945-special-case} by 1. Note that the bound $k\ell+1$ occurs only in the case 
	$A_{\gamma}$	is pd, which we assume in the rest of the proof.
	We denote by $\gamma(\mathrm z)$ a sequence obtained from the sequence $\gamma$ by replacing $\gamma_{2k\ell-1}$ with a variable $\mathrm z$.
	The matrix $A_{\gamma(\mathrm z)}$ is a partially pd matrix and by \cite[Lemma 2.11]{Zal21} there exist two choices of $\mathrm z$, 
	which we denote by $z^{\pm}$, such that $A_{\gamma(z^{\pm})}$ is psd and has rank $k\ell$.
	Since $\Rank A_{\gamma(z^{\pm})}(k\ell-1)=\Rank A_{\gamma(z^{\pm})}=k\ell$, 
	the sequence
	$\gamma(z^{\pm})$ is prg 
	and 
	by \cite[Theorem 3.9]{CF91} 
	it  
	admits a $k\ell$--atomic $\RR$--rm.
	If none of the moments $\beta_{i,j}$ of the sequence $\beta$ depends on $\gamma_{2k\ell-1}$, 
	the $\RR$--rm for $\gamma(z^{\pm})$ will generate a $K$--rm
	for $\beta$ as in the proof of Claim 4.  
	By \eqref{240822-1441}, the only moment from $\beta$, which could depend on $\gamma_{2k\ell-1}$,
	is $\beta_{0,2k}$. Note that if $q_{0,2k,2k\ell-1}=0$, then also $\beta_{0,2k}$ is independent from the value of
	$\gamma_{2k\ell-1}$.
	We have that 
		$$q_{0,2k,2k\ell-1}
			=\sum_{\substack{0\leq i_1,\ldots,i_{2k}\leq \ell,\\ i_1+\ldots+i_{2k}=2k\ell-1}}
			q_{i_1}q_{i_2}\ldots q_{i_{2k}}
			=2k (q_\ell)^{2k-1}q_{\ell-1},$$
	where in the first equality we used the definition \eqref{111122-1939} of $q_{0,2k,2k\ell-1}$,
	while in the second we used the fact that $i_1+\ldots+i_{2k}=2k\ell-1$ could be fullfilled only if 
	$2k-1$ indices $i_j$ are $\ell$ and one is $\ell-1$. So $q_{0,2k,2k\ell-1}=0$ iff $q_{\ell-1}=0$.	
	But this is true by Claim 1 and concludes the proof of Theorem \ref{240822-1317}.
\end{proof}

To prove Theorem \ref{240822-1317-odd} for $q(x)=x^\ell$, $\ell\geq 3$, only
a little adaptation of the last part of the proof of Theorem \ref{240822-1317} is needed, which we now explain.

\begin{proof}[Proof of Theorem \ref{240822-1317-odd}]
	The implications 
		$\eqref{240822-1317-pt1-odd}\Rightarrow\eqref{240822-1317-pt3-odd}$ 
	and
		$\eqref{240822-1317-pt4-odd}\Rightarrow\eqref{240822-1317-pt1-odd}$ 	
	are trivial.
	The implication $\eqref{240822-1317-pt3-odd}\Rightarrow\eqref{240822-1317-pt2-odd}$ follows from \cite[Theorem 3.14]{CF96}.
	It remains to prove the implication $\eqref{240822-1317-pt2-odd}\Rightarrow\eqref{240822-1317-pt4-odd}$.
Following the proof of Theorem \ref{240822-1317} everything remains the same until \ref{240822-1317-pt4'-special-case}.
It remains to justify that the upper bound in \eqref{121122-0945-special-case} can be decreased to $m:=k\ell-\lceil\frac{\ell}{2}\rceil$.
If $\Rank A_{\gamma}\leq m$, then we are already done. From now on we assume that $r:=\Rank A_{\gamma}>m$. Since $\gamma$ admits a $\RR$--rm, which we
denote by $\mu$, $\gamma$ is prg 
and $\Rank \gamma=\Rank{A_{\gamma}}=\Rank A_{\gamma}(r-1)$
by \cite[Theorem 3.9]{CF91}.
Hence, $A_{\gamma}(r-1)$ is pd and in particular also its submatrix $A_{\gamma}(m)$ is pd.
We denote by $\gamma(\mathrm z_1,\ldots,\mathrm z_{\ell})$ a sequence obtained from the sequence $\gamma$ by replacing $\gamma_{(2k-1)\ell+1},\gamma_{(2k-1)\ell+2},\ldots,\gamma_{2k\ell}$ with variables $\mathrm z_1,\ldots,\mathrm z_{\ell}$.
The sequence 
	$\gamma^{(0,(2k-1)\ell)}:=(\gamma_0,\ldots,\gamma_{(2k-1)\ell})$ 
is represented by $\mu$, being a subsequence of $\gamma$.
If $\ell$ is even, then $(2k-1)\ell$ is also even and by \cite[Theorem 3.9]{CF91}, 
$\gamma^{(0,(2k-1)\ell)}$ has a 
$(\Rank \gamma^{(0,(2k-1)\ell)})$--atomic $\RR$--rm.
Otherwise $\ell$ is odd, $(2k-1)\ell$ is also odd and by \cite[Theorem 3.1]{CF91}, 
$\gamma^{(0,(2k-1)\ell)}$ has a 
$(\Rank \gamma^{(0,(2k-1)\ell-1)})$--atomic $\RR$--rm, 
where $\gamma^{(0,(2k-1)\ell-1)}:=(\gamma_0,\ldots,\gamma_{(2k-1)\ell-1})$.
We denote the measure obtained in this way by $\mu_1$
and generate its moment sequence $\gamma(z_1,\ldots,z_{\ell})$, where $z_1,\ldots,z_{\ell}$
are the moments of degrees $(2k-1)\ell+1,\ldots, 2k\ell$.
Hence,
$\Rank A_{\gamma(z_1,\ldots,z_{\ell})}$ is equal to 
$\Rank A_{\gamma^{(0,(2k-1)\ell)}}$ for even $\ell$
and 
$\Rank A_{\gamma^{(0,(2k-1)\ell-1)}}$ for odd $\ell$.
Since $(2k-1)\ell=2m$ for even $\ell$ and $(2k-1)\ell-1=2m$ for odd $\ell$,
$\mu_1$ is $m$--atomic (since $A_{\gamma}(m)\succ 0$ by assumption).
	If none of the moments $\beta_{i,j}$ of the sequence $\beta^{(2k-1)}$ depends on $\gamma_{(2k-1)\ell+1},\gamma_{(2k-1)\ell+2},\ldots,\gamma_{2k\ell},$
	then $\mu_1$ will generate a $K$--rm
	for $\beta^{(2k-1)}$ as in the proof of Claim 5 of Theorem \ref{240822-1317}. 
	But by definition \eqref{250822-1207}, none of the moments of $\beta^{(2k-1)}$ depends on $\gamma_{(2k-1)\ell+1},\gamma_{(2k-1)\ell+3},\ldots,\gamma_{2k\ell}$,
	which concludes the proof of Theorem \ref{240822-1317-odd}.
\end{proof}

A corollary to Theorem \ref{240822-1317} is an improvement of the bounds on the degrees of sums of squares in the Positivstellensatz \cite[Corollary 6.3]{Fia11} for the curves of the form $y=q(x),$ $q\in\RR[x]$, $\deg q\geq 3.$

\begin{corollary}
	\label{290822-1100}
	Let $K:=\{(x,y)\in \RR^2\colon y=q(x)\}$, where $q\in\RR[x]$ satisfies $\deg q\geq 3$.
	Let $k \geq \deg q$.
	If $r(x,y)\in\RR[x,y]_{2k}$ is strictly positive on $K$, then $r$ admits a decomposition
		$$r(x,y)=
			\sum_{i=1}^{\ell_1} f_{i}(x,y)^2
			+(y-q(x))\sum_{i=1}^{\ell_2}g_i(x,y)^2
			-(y-q(x))\sum_{i=1}^{\ell_3}h_i(x,y)^2,$$
	where $\ell_1,\ell_2,\ell_3\in \ZZ_+$,
		$f_i,g_i\in \RR[x,y]$
	and 
		$$
			\deg f_i^2\leq 2m,
			\; \deg ((y-q(x))g_i^2)\leq 2m,		
			\;\deg ((y-q(x))h_i^2)\leq 2m
		$$
	with
		$m=k+\deg q-1.$
\end{corollary}

\begin{proof}
	By the equivalence $\eqref{240822-1317-pt1}\Leftrightarrow\eqref{240822-1317-pt2}$
	of Theorem \ref{240822-1317}, the set $K$ has the property $(R_{k,\deg q-2})$ in the notation of \cite[p.\ 2713]{CF08}.
	Now the result follows by \cite[Theorem 1.5]{CF08}.
\end{proof}

\begin{remark}
The bound on $m$ in Theorem \ref{290822-1100} from \cite[Corollary 5.4]{Fia11}
is quadratic in $k$ and $\deg q$, namely $(2k+1)\deg q$. 
\end{remark}

\subsection{Solution to the parabolic TMP}
\label{251122-2113}

The following is a concrete solution to the parabolic TMP, first solved in \cite{CF04}. 
We give an alternative proof together with a new solvability condition, i.e., 
\eqref{040722-pt4} below, where the variety condition is removed. 

\begin{theorem}[Solution to the parabolic TMP, even case]
	\label{090622-1549}
	Let
		 $$K:=\{(x,y)\in \RR^2\colon y=q_2x^2+q_1x+q_0\},$$ 
	where $q_0,q_1,q_2\in \RR$ and $q_2\neq 0$, be the parabola and 
	$\beta:=\beta^{(2k)}=(\beta_{i,j})_{i,j\in \ZZ_+,i+j\leq 2k}$, where $k\geq 2$. 
	The following statements are equivalent:
	\begin{enumerate}
		\item
			\label{040722-pt1} 
				$\beta$ has a $K$--representing measure.
		\item
			\label{040722-pt2} 
				$\beta$ has a $(\Rank M_k)$--atomic $K$--representing measure.
		\item
			\label{040722-pt3} 
				$M_k$ is positive semidefinite, 
				recursively generated, 
				satisfies the column relation $Y=q_2X^2+q_1X+q_0$ 
				and 
				$\Rank M_k\leq \Card \mathcal V(\beta),$
				where 
				$$\mathcal V(\beta):=\bigcap_{
				\substack{g\in \RR[x,y]_{\leq k},\\ 
							g(X,Y)=\mbf 0}} \mathcal Z(g).$$
		\item
			\label{290822-1554}  
				$M_k$ satisfies	$Y=q_2X^2+q_1X+q_0$ and 
				admits a positive semidefinite, 
				recursively generated extension $M_{k+1}$.
		\item
			\label{290822-1555}  
				$M_k$ satisfies	$Y=q_2X^2+q_1X+q_0$ and 
				admits a positive semidefinite extension $M_{k+1}$.
		\item
		\label{040722-pt4} 
			$M_k$ is positive semidefinite, the relations 
				$\beta_{i,j+1}=q_2\beta_{i+2,j}+q_1\beta_{i+1,j}+q_0\beta_{i,j}$ 
			hold for every $i,j\in \ZZ_+$ with $i+j\leq 2k-2$ 
			and, defining 		
			\begin{equation}
			\label{290822-2153}
				\cB=\{1,x,y,yx,\ldots,y^{k-1},y^{k-1}x,y^k\},
			\end{equation}
			 one of the following statements holds:
		\begin{enumerate}
			\item
				\label{040722-1401-pt1} 
					$(M_k)|_{\cB\setminus\{y^k\}}$ is positive definite.
				\smallskip 

			\item
				\label{040722-1401-pt2} 
					$\Rank (M_k)|_{\cB\setminus\{y^k\}}=\Rank M_k$.
		\end{enumerate}
	\end{enumerate}
\end{theorem}

\begin{proof}
	By applying an alt
		$\phi(x,y)=(x,\frac{1}{q_2}y-q_1x-q_0)$
	to the sequence $\beta$ we get a sequence $\widetilde \beta$ with the moment matrix $M_k(\widetilde \beta)$ satisfying $Y=X^2$. 
	Using Proposition \ref{251021-2254}, each of the statements \eqref{040722-pt1}--\eqref{040722-pt4} holds for the original sequence $\beta$ with the column relation 
	$Y=q_2X^2+q_1X+q_0$ 
	iff 
	it holds for $\widetilde\beta$ with the column relation $Y=X^2$.
	So we may assume $(q_2,q_1,q_0)=(1,0,0)$.
	Let us start by proving the equivalences $\eqref{040722-pt1} \Leftrightarrow\eqref{040722-pt2} \Leftrightarrow \eqref{040722-pt4}$.
	By \cite{Ric57}, $\eqref{040722-pt1}$ is equivalent to: 
	\begin{enumerate}[(1')]
		\item
		 	\label{040722-pt1-special-case}
				$\beta$ has a $s$--atomic $K$--representing measure for some $s\in \NN$.
	\end{enumerate}
	Let
	\begin{equation}
		\label{291122-1344}
		\gamma:=\gamma^{(0,4k)}=(\gamma_0,\gamma_1,\ldots,\gamma_{4k})\in 	
		\RR^{4k+1},
	\end{equation}
	where 
		$\gamma_{t}=\beta_{t\modulo 2,\lfloor \frac{t}{2}\rfloor}$, 
	which is a special case of definition \eqref{250822-1207} in the proof of Theorem \ref{240822-1317}.
	Claim 5
	in the proof of Theorem \ref{240822-1317} holds with the same proof also for $q(x)=x^2$.
	Using Claim 5 and \cite[Theorem 3.9]{CF91} for $\gamma$, the equivalences 
		$\ref{040722-pt1-special-case}\Leftrightarrow \eqref{040722-pt2} \Leftrightarrow \eqref{040722-pt4}$ 
	follow by noticing that
		$A_\gamma=(M_k)|_{\cB}$ and $A_{\gamma}(2k-1)=(M_k)|_{\cB\setminus\{Y^k\}}$.

	The implications $\eqref{040722-pt2}\Rightarrow \eqref{290822-1554}$ and $\eqref{290822-1554}\Rightarrow \eqref{290822-1555}$ are trivial.
	The implication $\eqref{040722-pt1}  \Rightarrow \eqref{040722-pt3}$ follows from the neccessary conditions for the existence of a $K$--rm
	(the variety condition follows from \cite[Proposition 3.1 and Corollary 3.7]{CF96}).

	Now we prove the implication
	$\eqref{290822-1555}  \Rightarrow  \eqref{040722-pt4}$.
	By \cite[Theorem 3.14]{CF96}, it follows that $M_k$ is rg.
	Defining the sequence 
		$\widetilde\gamma:= \gamma^{(0,4k+2)}=(\gamma_{0},\gamma_{1},\ldots, \gamma_{4k+1},\gamma_{4k+2})\in \RR^{4k+3},$
	where $\gamma_{t}=\beta_{t\modulo 2,\lfloor \frac{t}{2}\rfloor}$, 
	which is a special case of definition \eqref{250822-1207} in the proof of Theorem \ref{240822-1317},
	it follows by $M_{k+1}$ being psd that in particular
		$(M_{k+1})|_{\cB\cup\{y^{k+1}\}}=A_{\widetilde\gamma}$
	is also psd.
	If
		$A_{\gamma^{(0,4k)}}$ is pd,
	then $(M_k)|_{\cB\setminus \{y^k\}}$ is pd, 
	which is \eqref{040722-1401-pt1} of Theorem \ref{090622-1549}.
	Otherwise  $A_{\gamma^{(0,4k)}}$ is singular and prg by \cite[Theorem 2.6]{CF91}. 
	In particular,
	$\Rank A_{\gamma}=\Rank A_{\gamma}(2k-1)$,
	which, by noticing that $(M_k)|_{\cB\setminus \{y^k\}}=A_\gamma(2k-1)$,
	implies \eqref{040722-1401-pt2} of Theorem \ref{090622-1549}.
	This proves $\eqref{290822-1555}  \Rightarrow  \eqref{040722-pt4}$.

	It remains to prove the implication $\eqref{040722-pt3} \Rightarrow \eqref{040722-pt4}$.
	If $(M_k)|_{\cB\setminus\{y^k\}}$ is pd, we are done. Otherwise $(M_k)|_{\cB\setminus\{y^k\}}$ is not pd.
	We have to prove that in this case $\Rank (M_k)|_{\cB\setminus\{y^k\}}=\Rank M_k$.
	We asssume by contradiction that $\Rank (M_k)|_{\cB\setminus\{y^k\}}<\Rank M_k$.
 	Let $\gamma$ be as in \eqref{291122-1344}.
	The inequality  $\Rank (M_k)|_{\cB\setminus\{y^k\}}<\Rank M_k$ implies that $\Rank A_{\gamma}(2k-1)<\Rank A_{\gamma}$.
	Let $r=\Rank \gamma$. Then, 
	by \cite[Theorem 2.6]{CF91}, 
	$\Rank A_\gamma(2k-1)=r$ and hence, \cite[Theorems 3.9, 3.10]{CF91} imply that $\gamma^{(0,4k-2)}=(\gamma_0,\ldots,\gamma_{4k-2})$
	has a unique $r$--atomic $\RR$--rm with atoms $x_1,\ldots,x_r$. Hence, $\Rank A_\gamma=r+1$.
	Note that for every $g(x,y)\in \RR[x,y]$, which is a column relation of $M_k$, it follows that $g(x,x^2)\in \RR[x]$ is a column relation of $A_\gamma$ (where columns of $A_\gamma$ are $1,X,\ldots,X^{2k}$).
	Since $(x,y)\in \cV(\beta)$ and $(x,y')\in \cV(\beta)$, implies that $y=y'$ (due to $y=x^2$ and $y'=x^2$), it follows that $\cV(\beta)\subseteq\{(x_1,x_1^2),\ldots,(x_r,x_r^2)\}$. (This is true, since the atoms of a finitely atomic measure always satisfy all column relations of the moment matrix. Moreover, the sets are equal, but we do not need this in the rest of the proof.)
	Hence, $|\cV(\beta)|\leq r$.
	Since $\Rank A_{\gamma}=\Rank M_k$, this leads to a contradiction with the assumption
	$\Rank M_k\leq |\cV(\beta)|.$
\end{proof}

\begin{remark}
	\label{100323-1946}
\begin{enumerate}
\item
	\textbf{The proof of the implication $\eqref{040722-pt3}\Rightarrow \eqref{040722-pt2}$ of Theorem \ref{090622-1549} in \cite{CF04}:}
	\cite{CF04} considers 5 different cases according to the form of the relations between the columns
	of $\cB$ defined by \eqref{290822-2153}. The most demanding cases, which both use the FET as the main tool in the construction of
	a flat extension $M_{k+1}$ of $M_k$, are cases where there is only one relation and the column $Y^k$ occurs nontrivially in it or if there is no relation present.
\item 
	\label{100323-1946-pt2}
	\textbf{Complexity of checking conditions in the statements of Theorem \ref{090622-1549}:}
	Among conditions in \eqref{040722-pt3} the most demanding is the variety condition $\Rank M_k\leq \Card \mathcal V(\beta)$,
	since it requires solving a system of polynomial equation, which can be numerically difficult and unstable.
	However, \eqref{040722-pt4} of Theorem \ref{090622-1549} requires less work and can be stably checked numerically.
	Namely, one has to check
	if the column relations $Y^iX^j=q(X)^iX^j$, $i+j\leq k$, $i\in \NN$, hold, then try to compute 
	the Cholesky decomposition $VV^T$ of $(M_k)|_{\cB}$, compute the column echelon form $U$ of $V$
	and in case $U$ is of the form $\begin{pmatrix}I_r & B\end{pmatrix}^T$ for some matrix $B$, 
	then \eqref{040722-1401-pt1} or \eqref{040722-1401-pt2} holds and a rm exists.
	This follows by noticing that $(M_k)|_{\cB}$ corresponds to the Hankel matrix of a univariate sequence (see the proof above),
	for which the solution to the TMP is given in Algorithm \ref{alg-090323-1157}.
\end{enumerate}
\end{remark}

The following example shows that the variety condition	
$\Rank M_k\leq \Card \mathcal V(\beta)$ from Theorem \ref{090622-1549}.\eqref{040722-pt3} cannot be removed in contrast to the case of $K$ being a circle 
\cite[Theorem 2.1]{CF02} or a union of two parallel lines \cite[Theorem 3.1]{Zal22a}.
The \textit{Mathematica} file with numerical computations can be found on the link \url{https://github.com/ZalarA/TMP_quadratic_curves}.

\begin{example}\label{030822-2111}
Let $\beta=(\beta_{i,j})_{i,j\in \ZZ_+,i+j\leq 4}$
	be a bivariate sequence of degree 4
with the moment matrix $M_2$ equal to
\begin{equation*}
	M_2=
	\kbordermatrix{
		& 1 & X & Y & X^2 & XY & Y^2 \\
	 1 &
			3 & 0 & 2 & 2 & 0 & 2 \\
	X &  
			0 & 2 & 0 & 0 & 2 & 0 \\
	Y &  
			2 & 0 & 2 & 2 & 0 & 2 \\
	X^2 &
			2 & 0 & 2 & 2 & 0 & 2 \\
	XY & 
			0 & 2 & 0 & 0 & 2 & 0 \\
	Y^2 & 
			2 & 0 & 2 & 2 & 0 & 3 \\
	}.
	\end{equation*}
	$M_2$ is psd with the eigenvalues
		$\frac{1}{2}(9+\sqrt{65})\approx 8.53$, $4$, $1$, $\frac{1}{2}(9-\sqrt{65})\approx 0.47$, $0$, $0$,
	and the column relations 
		$Y=X^2$, $XY=X$.
	Hence, $M_2$ is psd, rg and satisfies $Y=X^2$.
	The variety $\cV(\beta)$ is equal to $\{(0,0),(-1,1),(1,1)\}$.
	So $4=\Rank M_2>\Card \cV(\beta)=3$ and the variety condition is not satisfied. 
	Thus, $\beta$ does not admit a representing measue supported on the parabola $y=x^2$.
	So $M_k$ being psd, satisfying $Y=X^2$ and rg does not imply the variety condition and the existence of a 
	representing measure.

	Note that by Remark \ref{100323-1946}.\eqref{100323-1946-pt2}
	it is cheaper and more stable to check only that $Y=X^2$ is a column relation and then
	solve the TMP for $\gamma$, such that $A_\gamma=(M_2)|_{\cB}$,
	using Algorithm \ref{alg-090323-1157}. Since the case 3.2 applies, a rm does not exist.
\end{example}

The following example demonstrates the solution of \cite[Example 1.6]{CF04} in the univariate setting.
The \textit{Mathematica} file with numerical computations can be found on the link \url{https://github.com/ZalarA/TMP_quadratic_curves}.

\begin{example}
	\label{290822-2056}
Let $\beta=(\beta_{i,j})_{i,j\in \ZZ_+,i+j\leq 6}$
	be a bivariate sequence of degree 6
with the moment matrix $M_3$ equal to
\begin{equation*}
	M_3=
	\kbordermatrix{
		& 1 & X & Y & X^2 & XY & Y^2 & X^3 & X^2Y & XY^2 & Y^3 \\
	 1 &
			1 & 0 & a & a & 0 & b & 0 & b & 0 & c \\
	X &  
			0 & a & 0 & 0 & b & 0 & b & 0 & c & 0 \\
	Y &  
			a & 0 & b & b & 0 & c & 0 & c & 0 & d \\
	X^2 &
			a & 0 & b & b & 0 & c & 0 & c & 0 & d \\
	XY & 
			0 & b & 0 & 0 & c & 0 & c & 0 & d & 0 \\
	Y^2 & 
			b & 0 & c & c & 0 & d & 0 & d & 0 & e \\
	X^3 & 
			0 & b & 0 & 0 & c & 0 & c & 0 & d & 0 \\
	X^2Y & 
			b & 0 & c & c & 0 & d & 0 & d & 0 & e \\
	XY^2 & 
			0 & c & 0 & 0 & d & 0 & d & 0 & e & 0 \\
	Y^3 & 
			c & 0 & d & d & 0 & e & 0 & e & 0 & f
	},
	\end{equation*}
	with the inequalities $a>0$, $b>a^2$, $c>\frac{b^2}{a}$, $d>\frac{b^3-2abc+c^2}{b-a^2}$,
	which ensure that $(M_2)|_{\{1,X,Y,XY,Y^2,X^2Y\}}$ is psd and $(M_2)|_{\{1,X,Y,XY,Y^2\}}$ is pd.
	Note that $M_3$ satisfies the column relations $Y=X^2$, $XY=X^3$ and $Y^2=X^2Y$.
	We introduce the univariate sequence 
		$$\gamma\in (1,0,a,0,b,0,c,0,d,0,e,0,f)\in \RR^{13}$$
	as in the proof of Theorem \ref{090622-1549}.	
 	We denote the rows and columns of $A_{\gamma}$ by $1,X,\ldots,X^6$.
	Since $(M_2)|_{\{1,X,Y,XY,Y^2\}}$ is pd, it follows that $A_{\gamma}(4)$ is pd.
	For 
		$$e=\big(\mathbf{v_\gamma}(5,4)\big)^T(A_{(1,0,a,0,b,0,c,0,d)})^{-1}\mathbf{v_\gamma}(5,4)
			=\frac{-c^3+2bcd-ad^2}{b^2-ac},$$
	we have that $A_{\gamma}(5)\succeq 0$ (e.g., using \cite[Theorem 1]{Alb69} for $A_{\gamma}(5)$)
	and $X^5\in \Span\{1,X,\ldots,X^4\}$ in $A_{\gamma}(5)$,
	where the vector 
	$
	\mathbf{v_\gamma}(5,4)=
	\left(
	\begin{array}{ccccc}	
		0 & c & 0 & d & 0
	\end{array}
	\right)^T
	$
	is the restriction of the column $X^5$
	to the rows indexed by $1,X,X^2,X^3,X^4$.
	Hence, for $\gamma$ to admit a $\RR$--rm,
	$A_\gamma\succeq 0$ and $X^i\in \Span\{1,X,\ldots,X^{i-1}\}$ for $i=5,6$ \cite[Theorem 3.9]{CF91}.
	Since $A_{\gamma}(5)\succeq 0$ and the last column of $A_{\gamma}(5)$ is a linear combination of the others,
	it only needs to hold by \cite[Theorem 1]{Alb69}, that
	\begin{align*}
		\mathbf{v_\gamma}(6,5)
		\in\cC(A_{\gamma}(5)) \quad \text{and}\quad
		f
		&=\big(\mathbf{v_\gamma}(6,5)\big)^T(A_{\gamma}(5))^{\dagger}\mathbf{v_\gamma}(6,5)\\
		&=\frac{-bc^4-b^2c^2d-2ac^3d-b^3d-b^3d^2+4abcd^2-a^2d^2}{(b^2-ac)^2},
	\end{align*}
	where $\mathbf{v_\gamma}(6,5)$ denotes the restriction of $X^6$ 
	to the rows indexed by $1,\ldots,X^5$ in $A_{\gamma}$ and $(A_{\gamma}(5))^\dagger$ denotes the Moore-Penrose inverse of $A_{\gamma}(5)$.
	Using \textit{Mathematica} we check that the equality $A_{\gamma}(5)(A_{\gamma}(5))^\dagger \mathbf{v_\gamma}(6,5)=\mathbf{v_\gamma}(6,5)$
	holds, which implies that 
	$\mathbf{v_\gamma}(6,5)\in\cC(A_{\gamma}(5))$ is true.
	By \cite[Theorem 3.10]{CF91}, in this case the $\RR$--rm is unique, 5--atomic and consists of the roots of the polynomial 
	\begin{align*}
		p(x)
		&=\left(\begin{array}{cccccc}1 & x & x^2 & x^3 & x^4 & x^5\end{array}\right)(A_{(1,0,a,0,b,0,c,0,d)})^{-1}\mathbf{v_\gamma}(5,4)\\
		&=x\Big(x^4+\frac{ad-bc}{b^2-ac}x^2+\frac{c^2-bd}{b^2-ac} \Big).
	\end{align*}
	Since $p(x)$ has roots $0,x_1,-x_1,x_2,-x_2$ and the atoms for the $K$--rm for $\beta$ are $(0,0)$, $(x_1,x_1^2)$, $(-x_1,x_1^2)$, $(x_2,x_2^2)$, $(-x_2,x_2^2)$.
\end{example}

The following theorem is a concrete solution to the parabolic TMP of odd degree, which is solved using	
the same technique as in the proof of Theorem \ref{240822-1317-odd}, but here we get explicit conditions
for the existence of the solution, similarly as in the even degree case.

\begin{theorem}[Solution to the parabolic TMP, odd case]
	\label{090622-1549-odd}
	Let
		 $K:=\{(x,y)\in \RR^2\colon y=x^2\}$
	be the parabola and 
	$\beta:=\beta^{(2k-1)}=(\beta_{i,j})_{i,j\in \ZZ_+,i+j\leq 2k-1}$, where $k\geq 2$. 
	Let
		$\gamma:=(\gamma_0,\gamma_1,\ldots,\gamma_{4k-2})$ 
	be a sequence,
			defined by
				$\gamma_t:=\beta_{t \modulo 2, \lfloor \frac{t}{2}\rfloor}$ for $t=0,1,\ldots,4k-2$.
	The following statements are equivalent:
	\begin{enumerate}
		\item
			\label{040722-pt1-odd} 
				$\beta$ has a $K$--representing measure.
		\item
			\label{040722-pt2-odd} 
				$\beta$ has a $(\Rank \gamma)$--atomic $K$--representing measure.
		\item
			\label{040722-pt3-odd}  
				$\beta$ can be extended to a sequence $\beta^{(2k)}$ such that $M_k$ is psd, rg, has a column relation $Y=X^2$
				and satisfies
				$\Rank M_k\leq \Card \mathcal V(\beta^{(2k)}),$
				where 
				$$\mathcal V(\beta^{(2k)}):=\bigcap_{
				\substack{g\in \RR[x,y]_{\leq k},\\ 
							g(X,Y)=\mbf 0\;\text{in}\;M_k}} \mathcal Z(g).$$ 
		\item
			\label{290822-1554-odd}   
				$\beta$ can be extended to a sequence $\beta^{(2k+2)}$ such that $M_{k+1}$ is psd and has a column relation $Y=X^2$.
		\item
		\label{040722-pt5-odd} 
			The relations $\beta_{i,j+1}=\beta_{i+2,j}$ hold for every $i,j\in \ZZ_+$ with $i+j\leq 2k-1$,
			$A_{\gamma}\succeq 0$ and
			the sequence 	
				$\gamma$ 
			is positively recursively generated.
		\item
		\label{040722-pt6-odd} 
			The relations $\beta_{i,j+1}=\beta_{i+2,j}$ hold for every $i,j\in \ZZ_+$ with $i+j\leq 2k-1$ 
			and defining 
				$\beta_{i,2k-i}=\beta_{i\modulo 2,2k-i+\lfloor \frac{i}{2}\rfloor}$
			for $2\leq i\leq 2k$,
			the moment matrix $(M_k)|_{\cB\setminus\{y^k\}}$,
			where 			
				$\cB=\{1,x,y,yx,\ldots,y^{k-1},y^{k-1}x,y^k\}$,
			is positive semidefinite and 
			\begin{equation}
				\label{181122-2315}
				\begin{pmatrix}
					\beta_{0,k} & \beta_{1,k} & \beta_{0,k+1} & \beta_{1,k+1} & \cdots & \beta_{0,2k-1}
				\end{pmatrix}^T\in \cC((M_k)|_{\cB\setminus\{xy^{k-1},y^k\},\cB\setminus \{y^k\}}).
			\end{equation}
	\end{enumerate}
\end{theorem}

\begin{proof}
	The equivalences $\eqref{040722-pt1-odd}\Leftrightarrow \eqref{040722-pt3-odd}\Leftrightarrow \eqref{290822-1554-odd}$
	follow by Theorem \ref{090622-1549}.
	By \cite{Ric57}, $\eqref{040722-pt1-odd}$ is equivalent to: 
	\begin{enumerate}[(1')]
		\item
		 	\label{040722-pt1-special-case-odd}
				$\beta$ has a $s$--atomic $K$--rm for some $s\in \NN$.
	\end{enumerate}
	Claim 5 of Theorem \ref{240822-1317} holds with the same proof also for $q(x)=x^2$ and odd degree sequence (i.e., $i+j\leq 2k-1$).
	Together with \cite[Theorem 3.9]{CF91}, the equivalences $\ref{040722-pt1-special-case-odd}\Leftrightarrow \eqref{040722-pt2-odd} \Leftrightarrow \eqref{040722-pt5-odd}$
	follow.
	Note that $(M_k)|_{\cB\setminus\{y^k\}}=A_{\gamma}$. 
	By \cite[Theorem 2.7.5]{BW11}, $\gamma$ is prg if and only if \eqref{181122-2315} holds.
	This establishes the equivalence $\eqref{040722-pt5-odd}\Leftrightarrow \eqref{040722-pt6-odd}$.
\end{proof}

\begin{remark}
	\begin{enumerate}
		\item 
			Note that $\Rank\gamma$ in Theorem \ref{090622-1549-odd} is at most $2k$ and it is $2k$ iff $A_\gamma$ is positive definite.
		\item
			Theorem \ref{090622-1549-odd} also solves the odd degree TMP on any curve	
			of the form $y=q_2x^2+q_1x+q_0\},$ where $q_0,q_1,q_2\in \RR$ and $q_2\neq 0$. 
			As in the proof of the even degree case one applies $\phi$ from the proof of Theorem \ref{090622-1549} to 
			$\beta$ to come into the case $y=x^2$ and then use Theorem \ref{090622-1549-odd}.
	\end{enumerate}
\end{remark}
 
\subsection{A solution to the TMP based on a feasibility of a linear matrix inequality}
\label{251122-2111}

In this subsection we give another alternative solution to the TMP on curves $y=q(x),$ where $q(x)\in\RR[x]$ and $\deg q\geq 3$, which is based on a
feasibility of a linear matrix inequality associated to the univariate sequence $\gamma$, obtained from 
the original sequence $\beta$ as in the proofs of the results of previous subsections. The feasibility question appears as a result of the fact 
that $\gamma$ is not fully determined by $\beta$, but $\beta$ admits a $K$--rm if and only if
$\gamma$ can be completed to a sequence admitting a $\RR$--rm. 

For $n\in \NN$, we denote by $[n]$ the set of all nonnegative integers smaller or equal to $n$.
Let the sets $N_1,N_2$ form a partition of $[n]$, i.e., $N_1,N_2\subseteq [n]$, $N_1\cup N_2=[n]$ and $N_1\cap N_2=\emptyset.$
Let $\Gamma_1:=(\gamma_t)_{t\in N_1}$  be a sequence of real numbers indexed by integers from $N_1$
and $\boldsymbol\Gamma_2:=(\boldsymbol\gamma_t)_{t\in N_2}$ a tuple of variables indexed by integers from $N_2$. 
Let 
\begin{equation}
	\label{230305-1808}
		F_{\Gamma_1}(\boldsymbol\Gamma_2):\RR^{|N_2|}\to \RR^{|N_1|+|N_2|}
\end{equation}
be a function with the output a sequence $(\widetilde \gamma_t)_{t\in [n]}$
where
$\widetilde \gamma_t=
\left\{
\begin{array}{rr}
\gamma_t,&	
	\text{if } t\in N_1,\\
\boldsymbol\gamma_t,&
	\text{if } t\in N_2.
\end{array}
\right.$

In Theorem \ref{121122-1341} below the set $N_1$ will be the set of indices, for which the corresponding univariate sequence $\gamma$
is determined by $\beta$, while the indices of the non-determined part will belong to $N_2$. Since we can either get an odd or an even sequence, for which the solutions to the $\RR$--TMP are slightly different, we separate two cases for $N_1\cup N_2=[n]$ (see \eqref{110323-1605}).

\begin{theorem}
	\label{121122-1341}
	Let $K:=\{(x,y)\in \RR^2\colon y=q(x)\}$, where $q(x)=\sum_{i=0}^\ell q_ix^i\in \RR[x]$, $\ell\geq 3$, $q_\ell\neq 0$, 
	and 
		$$\beta:=\beta^{(d)}=(\beta_{i,j})_{i,j\in \ZZ_+,i+j\leq d},$$ 
	where $\big\lceil \frac{d}{2}\big\rceil\geq \deg q$.
	For $i,j,s\in\ZZ_+$, such that $i+j\leq d$,
	define 	real numbers
	\begin{align*}	
	\displaystyle q_{i,j,s}
		&:=
			\left\{
			\begin{array}{rr}
				\displaystyle\sum_{\substack{0\leq i_1,\ldots,i_j\leq \ell,\\ i_1+\ldots+i_j=s-i}}
				q_{i_1}q_{i_2}\ldots q_{i_j},&		\text{if }i\leq s\leq i+j\ell,\\
				0,&		\text{otherwise}.
			\end{array}
			\right.
	\end{align*}
	Let
	\begin{equation*}
		N_1
		:=\Big\{t\in \ZZ_+\colon t\modulo \ell+\Big\lfloor{ \frac{t}{\ell}\Big\rfloor}\leq d\Big\},
	\end{equation*}
	\begin{equation}
		\label{121122-1732}
		\gamma_{t}
		=\frac{1}{(q_\ell)^{\lfloor \frac{t}{\ell}\rfloor}}\Big(\beta_{t\modulo \ell,\lfloor \frac{t}{\ell}\rfloor}-
		\sum_{s=0}^{t-1}q_{t\modulo \ell,\lfloor \frac{t}{\ell}\rfloor,s}\cdot \gamma_{s}\Big)\quad
		\text{for every }t\in N_1,
	\end{equation}
	and $\Gamma_1:=(\gamma_t)_{t\in N_1}$.
	Let
	\begin{equation}
	\label{110323-1605}
		[n]
		:=
		\left\{
		\begin{array}{rr}
		[d\ell+2],& \text{if }d\ell\text{ is even},\\[0.2em]
		[d\ell+1],& \text{if }d\ell \text{ is odd},
		\end{array}
		\right.
	\end{equation}
	$\boldsymbol\Gamma_2:=(\boldsymbol\gamma_t)_{t\in N_2}$ be a tuple of variables with $N_2=[n]\setminus N_1$
	and
	$F_{\Gamma_1}(\boldsymbol\Gamma_2)$ be defined as in \eqref{230305-1808}.
	Then the following statements are equivalent:
	\begin{enumerate}
		\item
		 	\label{121122-1341-pt1}
				$\beta$ has a $K$--representing measure.
		\item
		 	\label{121122-1341-pt3}
				$\beta_{i,j}=\sum_{p=0}^{\ell}q_p\beta_{i+p,j-1}$ for every $i,j\in \ZZ_+$, such that $i+j\leq d-\ell+1$
				and
				there exists a tuple
				 	$\Gamma_2=(\gamma_t)_{t\in N_2}\in \RR^{N_2}$
				such that
					$A_{F_{\Gamma_1}(\Gamma_2)}\succeq 0$.
	\end{enumerate}
\end{theorem}

\begin{proof}
Observing the proof of Theorem \ref{240822-1317} for a general $q(x)$
one can notice that 
		$F_{\Gamma_1}(\boldsymbol\Gamma_2)$
corresponds to the sequence $\widetilde\gamma$.
The original sequence $\beta$ determines only $\gamma_t$
for $t\in N_1$ by \eqref{121122-1732}, while for $t\in N_2$, $\boldsymbol\gamma_t$ are variables. By the proof of Theorem \ref{240822-1317},
$\beta$ will have a $K$--rm iff it satisfies the rg relations coming from the column relation $Y=q(X)$ and there exists 
$\widetilde \gamma$ such that $A_{\widetilde \gamma}\succeq 0$.
This proves Theorem \ref{121122-1341} for even $d$.

Observing the proof of Theorem \ref{240822-1317-odd} in case $d$ is odd one can notice that only 
	$\gamma^{(0,d\ell)}=(\gamma_0,\ldots,\gamma_{d\ell})$ 
needs to have a $\RR$--rm to obtain a $K$--rm for $\beta$. In case $d\ell$ is even, this is 
by \cite[Theorem 3.9]{CF91} equivalent to $A_{\gamma^{(0,d\ell+2)}}\succeq 0$, where 
	$\gamma^{(0,d\ell+2)}=(\gamma_0,\ldots,\gamma_{d\ell},\gamma_{d\ell+1},\gamma_{d\ell+2})$
for some  $\gamma_{d\ell+1}$, $\gamma_{d\ell+2}$.
Since $\gamma^{(0,d\ell+2)}$ corresponds to the sequence
		$F_{\Gamma_1}(\Gamma_2),$
this proves  Theorem \ref{121122-1341} for even $d\ell$ with $d$ being odd. If $d\ell$ is odd, then by \cite[Theorem 3.1]{CF91} 
it suffices that there is $\gamma_{d\ell+1}$
such that $A_{\gamma^{(0,d\ell+1)}}\succeq 0$, where $\gamma^{(0,d\ell+1)}=(\gamma_0,\ldots,\gamma_{d\ell},\gamma_{d\ell+1})$, 
and this proves Theorem \ref{121122-1341} for odd $d\ell$.
\end{proof}

We will present the statement of Theorem \ref{121122-1341} on a few examples.
The following example is for the case $\deg q=3$ and a sequence $\beta$ of even degree.

\begin{example}
	\label{121122-1941}
		Let $\beta=(\beta_{i,j})_{i,j\in \ZZ_+,i+j\leq 2k}$
	be a bivariate sequence of degree $2k$, $k\geq 3$,
	and $K:=\{(x,y)\in \RR^2\colon y=x^3\}$.
	For the existence of a $K$--rm $\beta$ must satisfy the relations $\beta_{i,j}=\beta_{i+3,j-1}$
	for every $i,j\in \ZZ_+$ such that $i+j+2\leq 2k$.
	In the notation of Theorem \ref{121122-1341} we have
	\begin{align*}
	q_{i,j,s}
		&:=
		\left\{
		\begin{array}{rr}
			1,&	\text{if }s=i+3j,\\
			0,& \text{otherwise},
		\end{array}	
		\right.\; \text{for }\; i,j,s\in \ZZ_+,\text{ such that }i+j\leq 2k,\\
	N_1
		&:=\left\{t\in \ZZ_+\colon t\modulo 3 + \Big\lfloor{ \frac{t}{3}\Big\rfloor} \leq 2k\right\}
		=\left\{t\in \ZZ_+\colon t\leq 6k,\; t\neq 6k-1\right\},\\
	[n]&:=[6k+2]
	\quad\text{and}\quad
	N_2
		:=\{6k-1,6k+1,6k+2\}.
	\end{align*}
	\noindent 
	The formula \eqref{121122-1732} is equal to 
		$$\gamma_t=\beta_{t\modulo t,\lfloor \frac{t}{3} \rfloor}\quad \text{for every }t\in N_1$$
	and
	the function $F_{\Gamma_1}:\RR^{3}\to \RR^{6k+3}$ is defined by
	\begin{equation*}
	F_{\Gamma_1}(\boldsymbol\Gamma_2)
	=
	F_{\Gamma_1}(\boldsymbol\gamma_{\mathbf{6k-1}},\boldsymbol\gamma_{\mathbf{6k+1}},\boldsymbol\gamma_{\mathbf{6k+2}})
		:=(\gamma_0,\gamma_1,\ldots,\gamma_{6k-2},\boldsymbol\gamma_{\mathbf{6k-1}},\gamma_{6k},\boldsymbol\gamma_{\mathbf{6k+1}},
			\boldsymbol\gamma_{\mathbf{6k+2}}).
	\end{equation*}
	The matrix 
	$A_{F_{\Gamma_1}(\boldsymbol\Gamma_2)}$ 
	is equal to
		\begin{equation*}\label{vector-v}
						\left(\begin{array}{cccccccc|c}
						\gamma_0 & \gamma_1 & \gamma_2 & \gamma_3 &\cdots & \cdots&&\gamma_{3k}&\gamma_{3k+1} \\
						\gamma_1 & \gamma_2 & \gamma_3 &\iddots&  &&&&\vdots\\
						\gamma_2 & \gamma_3 &\iddots &&&&&& \gamma_{6k-2}\\
						\gamma_3 & \iddots&&	&&&&\gamma_{6k-2}&\boldsymbol\gamma_{\mathbf{6k-1}}\\
						\vdots &&	 && &&\gamma_{6k-2} &\boldsymbol\gamma_{\mathbf{6k-1}} &\gamma_{6k}\\
						\gamma_{3k}&\cdots&&\cdots &\cdots &\gamma_{6k-2}&  \boldsymbol\gamma_{\mathbf{6k-1}}&\gamma_{6k}
& \boldsymbol\gamma_{\mathbf{6k+1}}\\\hline
						\gamma_{3k+1}&\cdots&\cdots&\cdots &\gamma_{6k-2} &\boldsymbol\gamma_{\mathbf{6k-1}}&  \gamma_{6k}&\boldsymbol\gamma_{\mathbf{6k+1}}
& \boldsymbol\gamma_{\mathbf{6k+2}}
						\end{array}\right),
	\end{equation*}
	The question of feasibility of $A_{F_{\Gamma_1}(\boldsymbol\Gamma_2)}\succeq 0$ can be answered analytically, 
	since the structure of the missing entries is simple enough. Actually it is even easier to work with 
	$A_{(\gamma_0,\ldots,\boldsymbol\gamma_{\mathbf{6k-1}},\gamma_{6k})}$
	and answer the feasibility question together with the condition from the solution of \cite[Theorem 3.9]{CF91} (see \cite[Theorem 3.1]{Zal21}).
\end{example}

\begin{remark}
\label{201122-2159}
 If $\deg q=3$ in Theorem \ref{240822-1317}, then a polynomial is of the form $y=q_3x^3+q_2x^2+q_1x+q_0\in \RR[x]$, where $q_3\neq 0$,
and using alt\textit{s}
it can be transformed to $y=x^3.$ Indeed, 
by first applying an alt as at the beginning of the proof of Claim 1 of Theorem \ref{240822-1317}, we can assume that $q_2=0$,
i.e., the polynomial becomes $y=q_3x^3+q_1x+q_0$. 
Now we apply an alt
$(x,y)\mapsto \big(x, y-q_1x-q_0\big)$,
followed by
$(x,y)\mapsto \big(\sqrt[3]{q_3}x,y\big)$
and get a polynomial $y=x^3$.
\end{remark}

The following example demonstrates the statement of Theorem \ref{121122-1341} for the case $\deg q=3$ and a sequence $\beta$ of odd degree.

\begin{example}
	\label{121122-1941-odd}
		Let $\beta=(\beta_{i,j})_{i,j\in \ZZ_+,i+j\leq 2k-1}$
	be a bivariate sequence of degree $2k-1$, $k\geq 3$,
	and $K:=\{(x,y)\in \RR^2\colon y=x^3\}$.
	For the existence of a $K$--rm $\beta$ must satisfy the relations $\beta_{i,j}=\beta_{i+3,j-1}$
	for every $i,j\in \ZZ_+$ such that $i+j+2\leq 2k-1$.
	In the notation of Theorem \ref{121122-1341}, we have
	\begin{align*}
		q_{i,j,s}
		&:=
		\left\{
		\begin{array}{rr}
			1,&	\text{if }s=i+3j,\\
			0,& \text{otherwise}.
		\end{array}	
		\right.\; \text{for }\; i,j,s\in \ZZ_+,\text{ such that }i+j\leq 2k-1,\\
		N_1
		&:=\left\{t\in \ZZ_+\colon t\modulo 3 + \Big\lfloor{ \frac{t}{3}\Big\rfloor} \leq 2k-1\right\}
		=\left\{t\in \ZZ_+\colon t\leq 6k-3 \text{ and } t\neq 6k-4\right\},\\
		[n]
		&=[6k-2]
		\quad\text{and}\quad
		N_2
		:=\{6k-4,6k-2\}.
	\end{align*}
	The formula \eqref{121122-1732} is equal to 
		$$\gamma_t=\beta_{t\modulo t,\lfloor \frac{t}{3} \rfloor}\quad \text{for every }t\in \NN,$$
	and the function $F_{\Gamma_1}:\RR^{2}\to \RR^{6k-1}$ is defined by
	\begin{equation*}
	F_{\Gamma_1}(\boldsymbol\Gamma_2)=
	F_{\Gamma_1}(\boldsymbol\gamma_{\mathbf{6k-4}},\boldsymbol\gamma_{\mathbf{6k-2}})
		:=(\gamma_0,\gamma_1,\ldots,\gamma_{6k-5},\boldsymbol\gamma_{\mathbf{6k-4}},\gamma_{6k-3},\boldsymbol\gamma_{\mathbf{6k-2}}).
	\end{equation*}
	Since $6k-3=(2k-1)\cdot 3$ is odd, only feasibility of the inequality
	$A_{F_{\Gamma_1}(\boldsymbol\Gamma_2)}\succeq 0$
	is important for the existence of the rm for $\beta$,
	where $A_{F_{\Gamma_1}(\boldsymbol\Gamma_2)}$ is equal to
		\begin{equation*}\label{vector-v}
						\left(\begin{array}{cccccccc}
						\gamma_0 & \gamma_1 & \gamma_2 & \gamma_3 &\cdots & \cdots&\gamma_{3k-1} \\
						\gamma_1 & \gamma_2 & \gamma_3 &\iddots&  &&\vdots\\
						\gamma_2 & \gamma_3 &\iddots &&&& \gamma_{6k-5}\\
						\gamma_3 & \iddots&&	&&\gamma_{6k-5}&\boldsymbol\gamma_{\mathbf{6k-4}}\\
						\vdots &&	 & &\gamma_{6k-5} &\boldsymbol\gamma_{\mathbf{6k-4}} &\gamma_{6k-3}\\
						\gamma_{3k-1}&\cdots &\cdots &\gamma_{6k-5}&  \boldsymbol\gamma_{\mathbf{6k-4}}&\gamma_{6k-3}
& \boldsymbol\gamma_{\mathbf{6k-2}}
						\end{array}\right).
	\end{equation*}
	This feasibility question 
	can be answered analytically, 
	since the structure of missing entries is simple enough. 
	See Theorem \ref{21122-0855-odd} below.
\end{example}

The following example demonstrates the statement of Theorem \ref{121122-1341} for the case $y=x^4$ and a sequence $\beta$ of even degree.

\begin{example}
	\label{121122-1804}
		Let $\beta=(\beta_{i,j})_{i,j\in \ZZ_+,i+j\leq 2k}$
	be a bivariate sequence of degree $2k$, $k\geq 4$,
	and $K:=\{(x,y)\in \RR^2\colon y=x^4\}$.
	For the existence of a $K$--rm $\beta$ must satisfy the relations $\beta_{i,j}=\beta_{i+4,j-1}$
	for every $i,j\in \ZZ_+$ such that $i+j+3\leq 2k$.
	In the notation of Theorem \ref{121122-1341}, we have
	\begin{align*}
		q_{i,j,s}
		&:=
		\left\{
		\begin{array}{rr}
			1,&	\text{if }s=i+4j,\\
			0,& \text{otherwise}.
		\end{array}	
		\right.\; \text{for }\; i,j,s\in \ZZ_+,\text{ such that }i+j\leq 2k,\\
		N_1
		&:=\left\{t\in \ZZ_+\colon t\modulo 4 + \Big\lfloor{ \frac{t}{4}\Big\rfloor} \leq 2k\right\}\\
		&=\left\{t\in \ZZ_+\colon t\leq 8k,\; t\notin \{8k-5,8k-2,8k-1\}\right\},\\[0.3em]
		[n]
		&:=[8k+2]
		\quad\text{and}\quad
		N_2
		:=\{8k-5,8k-2,8k-1,8k+1,8k+2\}.
	\end{align*}
	The formula \eqref{121122-1732} is equal to 
		$$\gamma_t=\beta_{t\modulo t,\big\lfloor \frac{t}{4} \big\rfloor}\quad \text{for every }t\in \NN,$$
	 the function $F_{\Gamma_1}:\RR^{5}\to \RR^{8k+3}$ is defined by
	\begin{align*}
	F_{\Gamma_1}(\boldsymbol\Gamma_2)
		&=
		F(\boldsymbol\gamma_{\mathbf{8k-5}},\boldsymbol\gamma_{\mathbf{8k-2}},\boldsymbol\gamma_{\mathbf{8k-1}},\boldsymbol\gamma_{\mathbf{8k+1}},\boldsymbol\gamma_{\mathbf{8k+2}})\\
		&:=(\gamma_0,\gamma_1,\ldots,\gamma_{8k-6},
			\boldsymbol\gamma_{\mathbf{8k-5}},\gamma_{8k-4},\gamma_{8k-3},
			\boldsymbol\gamma_{\mathbf{8k-2}},\boldsymbol\gamma_{\mathbf{8k-1}},\gamma_{8k},\boldsymbol\gamma_{\mathbf{8k+1}},\boldsymbol\gamma_{\mathbf{8k+2}}),
	\end{align*}
	and the matrix 
	$A_{F_{\Gamma_1}(\boldsymbol\Gamma_2)}$ 
	is equal to
		\begin{equation*}\label{vector-v}
						\left(\begin{array}{ccccccccc|c}
						\gamma_0 & \gamma_1 & \gamma_2 & \gamma_3 & \cdots&&&\cdots & \gamma_{4k} & \gamma_{4k+1}\\
						\gamma_1 & \gamma_2 & \gamma_3 &\iddots&&&&&  &\vdots\\
						\gamma_2 & \gamma_3 &\iddots &&&&& \gamma_{8k-6}&\boldsymbol\gamma_{\mathbf{8k-5}}&\gamma_{8k-4}\\
						\gamma_3 & \iddots&&	&&&\gamma_{8k-6}&\boldsymbol\gamma_{\mathbf{8k-5}}&\gamma_{8k-4}&\gamma_{8k-3}\\
						\vdots && &	 & &\gamma_{8k-6} & \boldsymbol\gamma_{\mathbf{8k-5}} &\gamma_{8k-4}&\gamma_{8k-3}& \boldsymbol\gamma_{\mathbf{8k-2}}\\
						&&&	&\gamma_{8k-6}& \boldsymbol\gamma_{\mathbf{8k-5}} &\gamma_{8k-4}& \gamma_{8k-3}& \boldsymbol\gamma_{\mathbf{8k-2}}
& \boldsymbol\gamma_{\mathbf{8k-1}}\\
					\vdots &&	&\gamma_{8k-6}	&\boldsymbol\gamma_{\mathbf{8k-5}} &	 \gamma_{8k-4}&\gamma_{8k-3} & \boldsymbol\gamma_{\mathbf{8k-2}}& \boldsymbol\gamma_{\mathbf{8k-1}} & \gamma_{8k}\\
\gamma_{4k}&  & \gamma_{8k-6} &\boldsymbol\gamma_{\mathbf{8k-5}} &\gamma_{8k-4} & \gamma_{8k-3} & \boldsymbol\gamma_{\mathbf{8k-2}} & \boldsymbol\gamma_{\mathbf{8k-1}} & \gamma_{8k} & \boldsymbol\gamma_{\mathbf{8k+1}} \\
\hline
\gamma_{4k+1}& \cdots & \boldsymbol\gamma_{\mathbf{8k-5}} &\gamma_{8k-4} & \gamma_{8k-3} & \boldsymbol\gamma_{\mathbf{8k-2}} & \boldsymbol\gamma_{\mathbf{8k-1}} & \gamma_{8k} & \boldsymbol\gamma_{\mathbf{8k+1}} & \boldsymbol\gamma_{\mathbf{8k+2}}
						\end{array}\right)
	\end{equation*}
	In contrast to the situation $y=x^3$ from Example 	
	\ref{121122-1941}, the structure of missing entries here is too complicated for the analytic approach and we believe the feasibility question can only be answered numerically 
	using linear matrix inequality solvers.
\end{example}

The following example demonstrates the statement of Theorem \ref{121122-1341} for the case $y=x^4$ and a sequence $\beta$ of odd degree.

\begin{example}
	\label{121122-1804-v2}
		Let $\beta=(\beta_{i,j})_{i,j\in \ZZ_+,i+j\leq 2k-1}$
	be a bivariate sequence of degree $2k-1$, $k\geq 4$,
	and $K:=\{(x,y)\in \RR^2\colon y=x^4\}$.
	For the existence of a $K$--rm $\beta$ must satisfy the relations $\beta_{i,j}=\beta_{i+4,j-1}$
	for every $i,j\in \ZZ_+$ such that $i+j+3\leq 2k-1$.
	In this case 
	\begin{align*}
		q_{i,j,s}
		&:=
		\left\{
		\begin{array}{rr}
			1,&	\text{if }s=i+4j,\\
			0,& \text{otherwise}.
		\end{array}	
		\right.\; \text{for }\; i,j,s\in \ZZ_+,\text{ such that }i+j\leq 2k-1,\\
		N_1
		&:=\left\{t\in \ZZ_+\colon t\modulo 4 + \Big\lfloor{ \frac{t}{4}\Big\rfloor} \leq 2k-1\right\}\\
		&=\left\{t\in \ZZ_+\colon t\leq 8k-4 \text{ and } t\notin \{8k-9,8k-6,8k-5\}\right\},\\[0.3em]
		n
		&:=[8k-2]
		\quad\text{and}\quad
		N_2
		:=\{8k-9,8k-6,8k-5,8k-3,8k-2\}.		
	\end{align*}
	The formula \eqref{121122-1732} is equal to 
		$$\gamma_t=\beta_{t\modulo t,\lfloor \frac{t}{4} \rfloor}\quad \text{for every }t\in \NN,$$
	the function $F_{\Gamma_1}:\RR^{5}\to \RR^{8k-1}$ is defined by
	\begin{align*}
		F_{\Gamma_1}(\boldsymbol\Gamma_2)
		&=
		F(\boldsymbol\gamma_{\mathbf{8k-9}},\boldsymbol\gamma_{\mathbf{8k-6}},\boldsymbol\gamma_{\mathbf{8k-5}},\boldsymbol\gamma_{\mathbf{8k-3}},\boldsymbol\gamma_{\mathbf{8k-2}})\\
		&:=(\gamma_0,\gamma_1,\ldots,\gamma_{8k-10},
			\boldsymbol\gamma_{\mathbf{8k-9}},\gamma_{8k-8},\gamma_{8k-7},
			\boldsymbol\gamma_{\mathbf{8k-6}},\boldsymbol\gamma_{\mathbf{8k-5}},\gamma_{8k-4},\boldsymbol\gamma_{\mathbf{8k-3}},\boldsymbol\gamma_{\mathbf{8k-2}}),
	\end{align*}
	and the matrix 
	$A_{F_{\Gamma_1}(\boldsymbol\Gamma_2)}$ 
	is equal to
		\begin{equation*}\label{vector-v}
						\left(\begin{array}{ccccccccc|c}
						\gamma_0 & \gamma_1 & \gamma_2 & \gamma_3 & \cdots&&&\cdots & \gamma_{4k-2} & \gamma_{4k-1}\\
						\gamma_1 & \gamma_2 & \gamma_3 &\iddots&&&&&  &\vdots\\
						\gamma_2 & \gamma_3 &\iddots &&&&& \gamma_{8k-10}&\boldsymbol\gamma_{\mathbf{8k-9}}&\gamma_{8k-8}\\
						\gamma_3 & \iddots&&	&&&\gamma_{8k-10}&\boldsymbol\gamma_{\mathbf{8k-9}}&\gamma_{8k-8}&\gamma_{8k-7}\\
						\vdots && &	 & &\gamma_{8k-10} & \boldsymbol\gamma_{\mathbf{8k-9}} &\gamma_{8k-8}&\gamma_{8k-7}& \boldsymbol\gamma_{\mathbf{8k-6}}\\
						&&&	&\gamma_{8k-10}& \boldsymbol\gamma_{\mathbf{8k-9}} &\gamma_{8k-8}& \gamma_{8k-7}& \boldsymbol\gamma_{\mathbf{8k-6}}
& \boldsymbol\gamma_{\mathbf{8k-5}}\\
					\vdots &&	&\gamma_{8k-10}	&\boldsymbol\gamma_{\mathbf{8k-9}} &	 \gamma_{8k-8}&\gamma_{8k-7} & \boldsymbol\gamma_{\mathbf{8k-6}}& \boldsymbol\gamma_{\mathbf{8k-5}} & \gamma_{8k-4}\\
\gamma_{4k-2}&  & \gamma_{8k-10} &\boldsymbol\gamma_{\mathbf{8k-9}} &\gamma_{8k-8} & \gamma_{8k-7} & \boldsymbol\gamma_{\mathbf{8k-6}} & \boldsymbol\gamma_{\mathbf{8k-5}} & \gamma_{8k-4} & \boldsymbol\gamma_{\mathbf{8k-3}} \\
\hline
\gamma_{4k-1}& \cdots & \boldsymbol\gamma_{\mathbf{8k-9}} &\gamma_{8k-8} & \gamma_{8k-7} & \boldsymbol\gamma_{\mathbf{8k-6}} & \boldsymbol\gamma_{\mathbf{8k-5}} & \gamma_{8k-4} & \boldsymbol\gamma_{\mathbf{8k-3}} & \boldsymbol\gamma_{\mathbf{8k-2}}
						\end{array}\right)
	\end{equation*}
	Since $8k-4=(2k-1)\cdot 4$ is even, the problem has the same structure as in the even degree case (see Example \ref{121122-1804}).
\end{example}

\subsection{A solution to the odd degree TMP on $y=x^3$}
\label{251122-2119} 

The following theorem is a concrete solution to the TMP of odd degree on the curve $y=x^3$, which can be solved using
the same technique as odd cases of the TMP on $y=x^\ell$, $\ell\geq 3$. However, for $\ell=3$ we get explicit conditions
for the existence of the solution, similarly as in the even degree case \cite[Theorem 3.1]{Zal21}.


\begin{theorem}[Solution to the TMP on $y=x^3$, odd case]
	\label{21122-0855-odd}  
	Let $K:=\{(x,y)\in \RR^2\colon y=x^3\}$ and 
	$\beta:=\beta^{(2k-1)}=(\beta_{i,j})_{i,j\in \ZZ_+,i+j\leq 2k-1}$, where $k\geq 3$. 
	Let
		$\gamma(\texttt{z}):=(\gamma_0,\gamma_1,\ldots,\gamma_{6k-5},\texttt{z},\gamma_{6k-3})$ 
	be a sequence,
			defined by
				$\gamma_t:=\beta_{t \modulo 3, \lfloor \frac{t}{3}\rfloor}$ for $t=0,1,\ldots,6k-5,6k-3$,
	and $\texttt{z}$ is a variable.
	The following statements are equivalent:
	\begin{enumerate}
		\item
			\label{211122-pt1-odd} 
				$\beta$ has a $K$--representing measure.
		\item
			\label{211122-pt2-odd} 
				$\beta$ has a $(\Rank M_{k-1})$--atomic or $(\Rank M_{k-1}+1)$--atomic $K$--representing measure.
		\item
		\label{211122-pt5-odd} 
			The relations $\beta_{i,j+1}=\beta_{i+3,j}$ hold for every $i,j\in \ZZ_+$ with $i+j\leq 2k-4$
			and 
			denoting 
				$\cB=\{1,x,x^2,y,yx,yx^2,\ldots,y^{k-1}\}$,
			one of the following holds:
			\begin{enumerate}
				\item 
					\label{211122-pt5-odd-a} 
					$(M_{k-1})|_{\cB}\succ 0$.	
					\smallskip

				\item 	
					\label{211122-pt5-odd-b} 
					$(M_{k-1})|_{\cB}\not\succ 0$, $(M_{k-1})|_{\cB}\succeq 0$, denoting
	 				$\gamma:=(\gamma_0,\gamma_1,\ldots,\gamma_{6k-6})$,
					$r:=\Rank \gamma$ and
				\begin{equation}
				\label{211222-1818}
				\left(\begin{array}{ccc}\varphi_0 & \cdots & \varphi_{r-1}\end{array}\right):=
						A_{\gamma}(r-1)^{-1}
						\left(\begin{array}{ccc}\gamma_r & \cdots &\gamma_{2r-1}\end{array}\right)^{T},
				\end{equation}
				it holds that	
				\begin{equation}
					\label{211122-1957}
						\gamma_{6k-u}=\sum_{i=0}^{r-1} \varphi_i \gamma_{6k-u-r+i}\quad \text{for }u=3,5,
				\end{equation}
				where $\gamma_{6k-4}$ is defined by \eqref{211122-1957} for $u=4$.
			\end{enumerate}
	\end{enumerate}
	Moreover, if a $K$--representing measure exists, then there does not exist a $(\Rank M_{k-1})$--atomic
	one if and only if $(M_{k-1})|_{\cB}\succ 0$ and $\gamma_{6k-3}$ does not satisfy \eqref{211122-1957} for $u=3$,
	where $\gamma_{6k-4}$ is obtained by \eqref{211122-1957} for $u=4$ and one uses \eqref{211222-1818} with $r=3k-2$.
\end{theorem}

\begin{proof}
	By Theorem \ref{121122-1341}, 
	$\eqref{211122-pt1-odd}$ is equivalent to the validity of the relations $\beta_{i,j+1}=\beta_{i+3,j}$ for every $i,j\in \ZZ_+$ with $i+j\leq 2k-4$
	and feasibility of $A_{F_{\Gamma_1}(\boldsymbol\gamma_{\mathbf{6k-4}},\boldsymbol\gamma_{\mathbf{6k-2}})}\succeq 0$,
	where the linear matrix function $A_{F_{\Gamma_1}(\boldsymbol\gamma_{\mathbf{6k-4}},\boldsymbol\gamma_{\mathbf{6k-2}})}$ is as in Example \ref{121122-1941-odd}.
	The latter is further equivalent to the existence of $\gamma_{6k-4}$ and $\gamma_{6k-2}$ such that
	$F_{\Gamma_1}(\gamma_{6k-4},\gamma_{6k-2})$ has a $\RR$--rm. Here we note that if 
	$A_{F_{\Gamma_1}(\gamma_{6k-4},\gamma_{6k-2})}\succeq 0$ is such that $\Rank F_{\Gamma_1}(\gamma_{6k-4},\gamma_{6k-2})<A_{F_{\Gamma_1}(\gamma_{6k-4},\gamma_{6k-2})}$,
	then 
	$\Rank F_{\Gamma_1}(\gamma_{6k-4},\gamma_{6k-2})=A_{F_{\Gamma_1}(\gamma_{6k-4},\gamma_{6k-2})}-1$
	by \cite[Corollary 2.5]{CF91}. Since $\gamma_{6k-2}$ occurs only in the bottom right corner of $A_{F_{\Gamma_1}(\gamma_{6k-4},\gamma_{6k-2})}$, we can replace it with $\widetilde\gamma_{6k-2}$ such that
	$A_{F_{\Gamma_1}(\gamma_{6k-4},\widetilde\gamma_{6k-2})}\succeq 0$ and $\Rank F_{\Gamma_1}(\gamma_{6k-4},\widetilde\gamma_{6k-2})=A_{F_{\Gamma_1}(\gamma_{6k-4},\widetilde\gamma_{6k-2})}$,
	which by \cite[Theorem 3.9]{CF91} indeed implies the existence of a $\RR$--rm.
	We have that
		$(M_{k-1})|_{\cB}=A_{\gamma}$.
	If $(M_{k-1})|_{\cB}\succ 0$, there exists $\gamma_{6k-4}$ such that
	$A_{(\gamma_0,\gamma_1,\ldots,\gamma_{6k-4})}\succ 0$ and by \cite[Theorem 3.1]{CF91}, 
	the sequence $(\gamma_0,\gamma_1,\ldots,\gamma_{6k-3})$ has a $(3k-1)$--atomic $\RR$--representing measure.
	Hence, one also finds $\gamma_{6k-2}$ such that $F_{\Gamma_1}(\gamma_{6k-4},\gamma_{6k-2})$ has a $\RR$--rm.
	If $(M_{k-1})|_{\cB}\succeq 0$ and $(M_{k-1})|_{\cB}\not \succ 0$, then by \cite[Theorem 3.8]{CF91}, 
	$(\gamma_0,\gamma_1,\ldots,\gamma_{6k-5})$ has a unique $\RR$--rm.
	This measure also represents $\gamma_{6k-3}$ iff \eqref{211122-1957} for $u=4$ and $u=3$ holds.
	This establishes the equivalence $\eqref{211122-pt1-odd}\Leftrightarrow \eqref{211122-pt5-odd}$.
	The equivalence of both with $\eqref{211122-pt2-odd}$ follows by observing that $F_{\Gamma_1}(\gamma_{6k-4},\gamma_{6k-2})$
	admits a $(\Rank \gamma)$--atomic or $(\Rank \gamma+1)$--atomic $\RR$--rm.
	The first case happens iff $\eqref{211122-pt5-odd-b}$ holds or $\eqref{211122-pt5-odd-a}$ holds and 
	$\gamma_{6k-4}$ is obtained by \eqref{211122-1957} for $u=4$, where one uses \eqref{211222-1818} with $r=3k-2$,
	and $\gamma_{6k-3}$ is obtained by \eqref{211122-1957} for $u=3$.
	Since $\Rank A_\gamma=\Rank M_{k-1}$, the equivalence follows.
\end{proof}

\section{The TMP on the curves $yx^\ell=1$}
\label{251122-2158}

In this section we study the $K$--TMP for $K$ being a curve of the form $yx^\ell=1$, $\ell\in \NN$, $\ell\geq 2$.
In Subsection \ref{251122-2138} we first give a solution to the $K$--TMP, 
based on the size of positive semidefinite extensions of the moment matrix needed and also bound the number of atoms in
the $K$--rm with the smallest number of atoms (see Theorem \ref{280822-2053} for the even degree and 
Theorem \ref{151122-1348} for the odd degree sequences). 
As a result we obtain a sum-of-squares representation for polynomials, which are strictly positive on $K$ (see Corollary \ref{290822-1052}).
This improves bounds in the previously known result \cite[Proposition 6.4]{Fia11}.
In Subsection \ref{251122-2140} we give a solution to the $K$--TMP, based on a feasibility of the corresponding linear matrix inequality 
(see Theorem \ref{14122-1729}).

\subsection{Solution to the TMP in terms of psd extensions of $M_k$, bounds on the number of atoms in the minimal measure and a Positivstellensatz}
\label{251122-2138}

\begin{theorem}[Even case]
	\label{280822-2053} 
	Let $K:=\{(x,y)\in \RR^2\colon yx^\ell=1\}$, where $\ell\in \NN\setminus\{1\}$,
	and 
		$\beta:=\beta^{(2k)}=(\beta_{i,j})_{i,j\in \ZZ_+,i+j\leq 2k}$, where $k\geq \ell+1$.
	The following statements are equivalent:
	\begin{enumerate}
		\item
		 	\label{280822-2053-pt1}
				$\beta$ has a $K$--representing measure.
		\item
		 	\label{280822-2053-pt4}
				$\beta$ has a $s$--atomic $K$--representing measure for some $s$ satisfying 
 					$$\Rank M_k\leq s\leq k(\ell+1).$$
		\item 
		 	\label{280822-2053-pt2}
				$M_k$ satisfies $Y=X^\ell$ 	
				and 
				admits a positive semidefinite, 
				recursively generated extension $M_{k+\ell}$.
		\item
		 	\label{280822-2053-pt3}
				$M_k$ satisfies $Y=X^\ell$ 
				and
				admits a positive semidefinite extension $M_{k+\ell+1}$.
	\end{enumerate}
\end{theorem}

\begin{theorem}[Odd case]
	\label{151122-1348}
	Let $K:=\{(x,y)\in \RR^2\colon yx^\ell=1\}$, where $\ell\in \NN\setminus\{1\}$,
	and 
		$\beta:=\beta^{(2k-1)}=(\beta_{i,j})_{i,j\in \ZZ_+,i+j\leq 2k-1}$, where $k\geq \ell+1$.
	The following statements are equivalent:
	\begin{enumerate}
		\item
		 	\label{151122-1348-pt1}
				$\beta$ has a $K$--representing measure.
		\item
		 	\label{151122-1348-pt4}
				$\beta$ has a $s$--atomic $K$--representing measure for some $s$ satisfying 
 					$$\Rank M_k\leq s\leq k(\ell+1)-\Big\lfloor\frac{\ell}{2}\Big\rfloor+1.$$
		\item 
		 	\label{151122-1348-pt2}		 	
				$\beta^{(2k-1)}$ can be extended to a sequence $\beta^{(2k)}$
				such that $M_{k}$ satisfies $YX^\ell=1$
				and admits a positive semidefinite, recursively generated extension $M_{k+\ell}$.
		\item
		 	\label{151122-1348-pt3}
				$\beta^{(2k-1)}$ can be extended to a sequence $\beta^{(2k)}$
				such that $M_{k}$ satisfies $YX^\ell=1$
				and admits a positive semidefinite extension $M_{k+\ell+1}$.
	\end{enumerate}
\end{theorem}

\begin{remark}
\label{211122-2118}
\begin{enumerate}
\item\textbf{Previous bounds on the size of extensions in \eqref{151122-1348-pt3} of Theorem \ref{280822-2053}:}
	In \cite[Section 6]{Fia11} the author considered TMPs on $\cZ(p)$
	in terms of the size of psd extensions of the moment matrix
	also for polynomials of the form	$p(x,y)=yq(x)$, 
	where $q\in \RR[x]$.
			Namely, by \cite[Propositions 6.1, 6.4]{Fia11},
			a sequence of degree $2k$ admits a $\cZ(p)$--rm, 
			if $M_k$ admits a psd extensions $M_{k+r}$,
			where 
					$r=(2k+2)(2+\deg q)-(1+\deg q+k).$
	The proof of this result relies on the truncated Riesz-Haviland theorem \cite[Theorem 1.2]{CF08}
	and a sum-of-squares representations for polynomials, strictly positive on $\cZ(p)$ 
	(\cite[Proposition 6.4]{Fia11} and \cite[Proposition 5.1]{Sto01}).
	Part \eqref{280822-2053-pt3} of Theorem \ref{280822-2053} improves Fialow's result in case $q(x)=x^\ell$, $\ell\geq 2$,
	by	decreasing the size of the extensions to $r=\ell+1$.

\item
\textbf{Known bounds on the number of atoms in \eqref{280822-2053-pt4} of Theorems \ref{280822-2053}, \ref{151122-1348}}:
	Similarly as in Remark \ref{191122-1055}.\eqref{191122-1055-pt2}, part \eqref{280822-2053-pt4} of 
	Theorem \ref{280822-2053} is a counterpart of \cite[Corollary 7.6]{RS18} for even degree sequences on curves $\cZ(yx^\ell-1)$,
	while part	\eqref{151122-1348-pt4} of Theorem \ref{151122-1348} 
	improves \cite[Corollary 7.6]{RS18} for curves $\cZ(yx^\ell-1)$
	by decreasing it for $\lfloor\frac{\ell}{2}\rfloor-1$.

\item
	\textbf{Uniqueness and description of all solutions in Theorems \ref{280822-2053} and \ref{151122-1348}:}
	The same comment as for Remark \ref{191122-1055}.\eqref{120323-0031} applies here.
	Beyond the case $yx^2=1$ (see Example \ref{141122-1921} below) the structure of the missing entries of $A_{\gamma}$ from the proof 
	of Theorem \ref{280822-2053} is too complicated to have control over all psd completions and consequently over the uniqueness and the description 
	of all solutions.
\item
	\textbf{Complexity of checking conditions in \eqref{240822-1317-pt3} of Theorems \ref{280822-2053} and \ref{151122-1348}:}
	The same comment as for Remark \ref{191122-1055}.\eqref{120323-0036} applies here.
	The main complexity question is the SDP feasibility question, which is cheaper when dealt with on a univariate sequence $A_{\widetilde \gamma}$ 
	defined in the proof. A precise SDP is stated in Theorem \ref{14122-1729} below.
\end{enumerate}
\end{remark}

\begin{proof}[Proof of Theorem \ref{280822-2053}]
	The implications 
		$\eqref{280822-2053-pt1}\Rightarrow\eqref{280822-2053-pt3}$ 
	and
		$\eqref{280822-2053-pt4}\Rightarrow\eqref{280822-2053-pt1}$ 	
	are trivial.
	The implication $\eqref{280822-2053-pt3}\Rightarrow\eqref{280822-2053-pt2}$ follows by \cite[Theorem 3.14]{CF96}.
	It remains to prove the implication $\eqref{280822-2053-pt2}\Rightarrow\eqref{280822-2053-pt4}$.
	Assume that $YX^\ell$ is a column relation and $M_k$ admits a psd, rg extension $M_{k+\ell}$.
Let 
	\begin{equation}	
		\label{290822-0930}
			\cB=\left\{y^{k+1}x^{\ell-1},y^{k},y^{k}x,\ldots,y^kx^{\ell-1},\ldots,y,yx,\ldots,yx^{\ell-1},1,x,\ldots,x^{k+1}\right\}
	\end{equation}
	be the set of monomials and $V$ the vector subspace in $\RR[x,y]_{k+\ell}$ generated by the set $\cB$.
	Since $M_{k+\ell}$ satisfies 
	$$
	X^iY^{j}=
	\left\{
	\begin{array}{rl}
	X^{i \modulo \ell}Y^{j-\lfloor\frac{i}{\ell}\rfloor},&\text{if }i,j\in \ZZ_+, \; i+j\leq k,\; j\geq\lfloor\frac{i}{\ell}\rfloor,\\[0.5em]
	X^{i -j\ell},&\text{if }i,j\in \ZZ_+, \; i+j\leq k,\; j<\lfloor\frac{i}{\ell}\rfloor,
	\end{array}
	\right.
	$$ 
	it follows that the columns from $\cB$
	span $\cC(M_{k+\ell})$.
	Let
		$p(x,y)=\sum_{i,j}p_{ij}x^iy^j\in V$
	be a polynomial and 
		$\widehat p$
	a vector of its coefficients ordered in the basis $\cB$.
	We define a univariate polynomial $q_p(x)$ corresponding to $p(x,y)$, by
	\begin{equation}
	\label{301122-2251}
		g_p(x)
			:=
			p(x,x^{-\ell})
			=
			\sum_{i,j}p_{ij} x^{i-\ell j}
			=:
			\sum_{s=-k\ell-1 }^{k+1} g_{p,s} x^s
			\in 
		\RR[x]_{k\ell+1}.
	\end{equation}
	Let $\widehat{g_p}$ be its vector of coefficients in the basis 
	\begin{equation}
		\label{211122-2139}
			\cB_1=\{x^{-k\ell-1},x^{-k\ell},\ldots,x^{k+1}\}.
	\end{equation}
	The monomials $x^{i_1}y^{j_1}$, $x^{i_2}y^{j_2}$ from $\cB$  correspond to the same monomial $x^s$
	by the correspondence \eqref{301122-2251} iff $i_1-\ell j_1=i_2-\ell j_2$, which is further equivalent to
	$i_1=i_2$ and $j_1=j_2$ (since $i_1$ and $i_2$ are at most $\ell-1$ in $\cB$).
	Therefore
	\begin{equation}
			\label{131122-2106}
				\widehat{g_p}=\widehat p.
	\end{equation}
	
	We define two univariate sequences
	\begin{align*}
		\gamma
		&:=\gamma^{(-2k\ell,2k)}
		=
		(\gamma_{-2k\ell},\gamma_{-2k\ell+1},\ldots,\gamma_0,\ldots,\gamma_{2k})\in \RR^{2k(1+\ell)+1},\\
		\widetilde\gamma
		&:=
		\gamma^{(-2k\ell-2,2k+2)}
		=
		(\gamma_{-2k\ell-2},\gamma_{-2k\ell-1},\gamma,\gamma_{2k+1},\gamma_{2k+2})\in \RR^{2k(1+\ell)+5},
	\end{align*}
	by the formula
	 \begin{equation}\label{280822-2111}
		\gamma_{t}
		=
		\left\{
		\begin{array}{rl}
			\beta_{t,0},& \text{if }t\geq 0,\\[0.5em]
			\beta_{t+\ell\big\lceil{\frac{|t|}{\ell}}\big\rceil,\big\lceil\frac{|t|}{\ell}\big\rceil},& \text{if } t<0.
		\end{array}
		\right.
	\end{equation}
	Note that for $t<0$ we have that 
		$t+\ell\big\lceil\frac{|t|}{\ell}\big\rceil\leq \ell-1$,
		$\big\lceil\frac{|t|}{\ell}\big\rceil\leq 2k+1$ (since $\ell\geq 3$)
	and hence
		$$
		t+\ell\Big\lceil\frac{|t|}{\ell}\Big\rceil+\Big\lceil\frac{|t|}{\ell}\Big\rceil
		\leq 
		\ell-1+2k+1
		=2k+\ell.
		$$
	Therefore 
		$\beta_{t+\ell\big\lceil{\frac{|t|}{\ell}}\big\rceil,\big\lceil\frac{|t|}{\ell}\big\rceil}$
	is well-defined being an element of the matrix $M_{k+\ell}$ (since $2k+\ell\geq 2k+2\ell$).\\

	By the following claim solving the $K$--TMP for $\beta$ is equivalent to solving the $(\RR\setminus\{0\})$--TMP for $\gamma$.\\

	\noindent\textbf{Claim 1.} 
		Let $u\in \NN$.
			A sequence $\gamma$ admits a $u$--atomic $(\RR\setminus \{0\})$--rm
		if and only if 
			$\beta$ admits a $u$--atomic $K$--rm.\\

\noindent \textit{Proof of Claim 1.} 
First we prove the implication $(\Rightarrow)$. 
Let $x_1,\ldots, x_u$, be the atoms in the $(\RR\setminus\{0\})$--rm for $\gamma$ with the corresponding densities $\rho_1,\ldots,\rho_u$.
We will prove that the atoms $(x_1,(x_1)^{-\ell}),\ldots, (x_u,(x_u)^{-\ell})$ with densities $\rho_1,\ldots,\rho_p$ are the $K$--rm for $\beta$.
We separate two cases:
\begin{enumerate}
\item $\lfloor\frac{i}{\ell}\rfloor \geq j$:
	\begin{align*}
		\beta_{i,j}
		&=
		\beta_{i-\ell j,0}
		=
		\gamma_{i-\ell j}
		=
		\sum_{p=0}^u \rho_p(x_{p})^{i-\ell j}
		=
		\sum_{p=0}^u \rho_p(x_{p})^{i}((x_p)^{-\ell})^{j},
	\end{align*}
	were we used the fact that $\beta$ is rg in the first equality,
	\eqref{280822-2111} in the second equality,
	the definitions of $\rho_p, x_p$ in the third equality
	and 
	split $(x_p)^{i-\ell j}$ into two parts in the last equality.
\item $\lfloor\frac{i}{\ell}\rfloor <j$:
\begin{align*}
		\beta_{i,j}
		=
		\beta_{i \modulo \ell,j-\lfloor\frac{i}{\ell}\rfloor}
		=
		\gamma_{-(j-\lfloor\frac{i}{\ell}\rfloor)\ell+i \modulo \ell}
		&=
		\sum_{p=0}^u \rho_p(x_{p})^{-(j-\lfloor\frac{i}{\ell}\rfloor)\ell+i \modulo \ell}\\
		&=
		\sum_{p=0}^u \rho_p(x_{p})^{\lfloor\frac{i}{\ell}\rfloor \ell+i \modulo \ell}((x_p)^{-\ell})^{j}\\
		&=
		\sum_{p=0}^u \rho_p(x_{p})^{i}((x_p)^{-\ell})^{j},
	\end{align*}
	were we used the fact that $\beta$ is rg in the first equality,
	\eqref{280822-2111} in the second equality,
	the definitions of $\rho_p, x_p$ in the third equality,
	split the exponent at $x_p$ into two parts in the fourth equality
	and 
	used that 
		$\lfloor\frac{i}{\ell}\rfloor \ell+i \modulo \ell=i$
	in the last equality.
\end{enumerate}
	This proves the implication $(\Rightarrow)$.

	It remains to prove the implication $(\Leftarrow)$.
	Let 	
		$(x_1,(x_1)^{-\ell}),\ldots, (x_u,(x_u)^{-\ell})$	
	be the atoms in the $K$--rm for $\beta$ with the corresponding densities 
		$\rho_1,\ldots,\rho_u$.
	We will prove that the atoms 
		$(x_1,\ldots,x_u)$ 
	with densities 
		$\rho_1,\ldots,\rho_p$ 
	are the $(\RR\setminus\{0\})$--rm for $\gamma$:
	\begin{itemize}
	\item
	For $t\geq 0$ we have that
	\begin{equation*}
	  \gamma_{t}
		=\beta_{t,0}
		=\sum_{p=0}^u\rho_p (x_p)^{t},
	\end{equation*}
	where we use the definition \eqref{280822-2111} in the first equality and the definitions of $\rho_p,x_p$ in the second.
	\item
	For $t<0$ we have that
	\begin{equation*}
	  \gamma_{t}
		=\beta_{t+\ell\big\lceil{\frac{|t|}{\ell}}\big\rceil,\big\lceil\frac{|t|}{\ell}\big\rceil} 
		=\sum_{p=0}^u\rho_p (x_p)^{t+\ell\big\lceil{\frac{|t|}{\ell}}\big\rceil}((x_p)^{-\ell})^{\big\lceil\frac{|t|}{\ell}\big\rceil}
		=\sum_{p=0}^u\rho_p (x_p)^{t},
	\end{equation*}
	where we use the definition \eqref{280822-2111} in the first equality and the definitions of $\rho_p,x_p$ in the second.
	\end{itemize}
	This proves the implication $(\Leftarrow)$.
	\hfill $\blacksquare$\\

	Let	$(M_{k+\ell})|_{\cB}$ be the restriction of $M_{k+\ell}$ to the rows and columns indexed by monomials (capitalized)
	from $\cB$. The following claim gives an explicit connection between $(M_{k+\ell})|_{\cB}$ and the Hankel matrix 	
	$A_{\widetilde\gamma}$ of the sequence $\widetilde\gamma$.\\
	
	\noindent\textbf{Claim 2.} 
	We have that
	\begin{equation}
		\label{131122-2155}
			(M_{k+\ell})|_{\cB}=A_{\widetilde\gamma}.
	\end{equation}

	\noindent \textit{Proof of Claim 2.}
	Let 
		$p(x,y)=\sum_{i,j}p_{ij}x^iy^j\in V$ 
	and
		$r(x,y)=\sum_{i,j}r_{ij}x^iy^j\in V$ 
	be polynomials from the vector subspace $V$ 
	and 
		$\widehat p$, $\widehat r$ 
	vectors of their coefficients ordered in the basis $\cB$.
	Let 
	$\widetilde \beta:=\beta^{(2(k+1)\ell+2(\ell-1))}$.
	Then we have
	\begin{align*}
		(\widehat r)^T\left((M_{k+\ell})|_{\cB}\right)\widehat p
		&=^1 L_{\widetilde\beta}(pr)
		=L_{\widetilde \beta}\big(\sum_{i_1,i_2,j_1,j_2} p_{i_1j_1}r_{i_2j_2}x^{i_1+i_2}y^{j_1+j_2}\big)\\
		&=^2 \sum_{i_1,i_2,j_1,j_2} p_{i_1j_1}r_{i_2j_2}\beta_{i_1+i_2,j_1+j_2}\\
		&=^3 \sum_{i_1,i_2,j_1,j_2} p_{i_1j_1}r_{i_2j_2}\gamma_{i_1+i_2-(j_1+j_2)\ell}\\
	 	&=^4 L_{\widetilde\gamma}\left(\sum_{i_1,i_2,j_1,j_2} 
			p_{i_1j_1}r_{i_2j_2}x^{i_1+i_2-(j_1+j_2)\ell}\right)\\
	 	&=^5 L_{\widetilde\gamma}\left(\sum_{i_1,i_2,j_1,j_2} p_{i_1j_1}x^{i_1-j_1\ell}\cdot r_{i_2j_2}x^{i_2-j_2\ell}\right)\\
	 	&=^6 L_{\widetilde\gamma}\Big( 
			\Big(\underbrace{\sum_{i_1,j_1}p_{i_1j_1} x^{i_1-j_1\ell}}_{g_p(x)}\Big)
			\Big(\underbrace{\sum_{i_2,j_2} r_{i_2j_2}x^{i_2-j_2\ell}}_{g_r(x)}\Big)\Big)\\
		&=^7 \widehat{g_r}^T A_{\widetilde\gamma}\widehat{g_p}
		=\widehat r^T A_{\widetilde\gamma} \widehat p,
	\end{align*}
	where 
		in the first line we used the correspondence between the moment matrix and the Riesz functional $L_{\widetilde \beta}$,
		the definition $L_{\widetilde \beta}$ in the second,
		\eqref{280822-2111} and the fact that $\beta$ is rg in the third (rg is needed if $i_1+i_2\geq \ell$),
		the definition of $L_{\widetilde\gamma}$ in the fourth,
		decomposed the exponent of $x$ into two parts in the fifth,
		decomposed a sum into the product of two sums in the sixth,
		in the seventh  we used the correspondence between $A_{\widetilde \gamma}$ and the Riesz functional $L_{\widetilde\gamma}$,
	where $\widehat{g_p}$, $\widehat{g_r}$ are the 
		vectors of coefficents of $g_p$ and $g_r$ in the basis $\cB_1$ (see \eqref{211122-2139})
		and \eqref{131122-2106}.
		Since $p$ and $q$ were arbitrary from $V$, this proves Claim 2.
	\hfill$\blacksquare$\\

	Since $(M_{k+\ell})|_{\cB}$ is psd, it follows from \eqref{131122-2155}
	that $A_{\widetilde\gamma}$ is psd. 
	We separate two cases. Either $A_{\gamma}$ is pd or $A_{\gamma}$ is psd, singular,
	prg by \cite[Theorem 2.6]{CF91} and nrg by \cite[Proposition 2.1.(5)]{Zal22b}.
	By \cite[Theorem 3.1]{Zal22b}, 
		$\gamma$ 
	admits a $(\Rank A_{\gamma})$--atomic $(\RR\setminus\{0\})$--rm.
	Since $\Rank M_k\leq \Rank A_{\gamma}\leq k(\ell+1)+1$,  
	using Claim 2 the following holds:
	\begin{enumerate}[(2')]
	\item
		 	\label{131122-2307}
				$\beta$ has a $s$--atomic $K$--rm for some $s$ satisfying 
 				\begin{equation}
					\label{131122-2308}
						\Rank M_k\leq s\leq \Rank A_\gamma\leq k(\ell+1)+1.
				\end{equation}
	\end{enumerate}
	To obtain \eqref{280822-2053-pt4} of Theorem \ref{280822-2053} we need to decrease the upper bound in
     \eqref{131122-2308} by 1. Note that the bound $k(\ell+1)+1$ occurs only in the case 
	$A_{\gamma}$	is pd. 
	We denote by $\gamma(\mathrm z)$ a sequence obtained from the sequence $\gamma$ by replacing $\gamma_{-2k\ell+1}$ with a variable $\mathrm z$.
	The matrix $A_{\gamma(\mathrm z)}$ is a partially positive definite matrix and by \cite[Lemma 2.11]{Zal21} there exist two choices of $\mathrm z$, 
	which we denote by $z^{\pm}$, such that $A_{\gamma(z^{\pm})}$ is psd and has rank $k(\ell+1)$.
	Using \cite[Proposition 2.5]{Zal22b} for the reversed sequence $(\gamma(z^{\pm}))^{(\text{rev})}$ of $\gamma(z^{\pm})$,
	we see that at least one of $(\gamma(z^{\pm}))^{(\text{rev})}$  admits a $(\RR\setminus \{0\})$--rm.
	Hence, at least one of $\gamma(z^{\pm})$ is prg and nrg, and admits a $k(\ell+1)$--atomic $(\RR\setminus \{0\})$--rm.
	If none of the moments $\beta_{i,j}$ of the sequence $\beta$ depends on $\gamma_{-2k\ell+1}$, 
	the $(\RR\setminus \{0\})$--rm for $(\gamma(z^{\pm}))^{(\text{rev})}$ will generate the $K$--rm
	for $\beta$ as in the proof of Claim 2. 
	But by definition \eqref{280822-2111}, there is indeed no moment from $\beta$, which depends on $\gamma_{-2k\ell+1}$  (since
	we need to represent moments of degree at least $-2k\ell$ and at most $2k$, while $\gamma_{-2k\ell+1}$ corresponds to $\beta_{\ell-1,2k}$ in some extension of $	
\beta$),
	which concludes the proof of Theorem \ref{280822-2053}.
\end{proof}

To prove Theorem \ref{151122-1348} only a little adaptation of the last part of the proof of Theorem \ref{280822-2053} is needed,
which we now explain.

\begin{proof}[Proof of Theorem \ref{151122-1348}]
The implications
		$\eqref{151122-1348-pt1}\Rightarrow\eqref{151122-1348-pt3}$ 
	and
		$\eqref{151122-1348-pt4}\Rightarrow\eqref{151122-1348-pt1}$ 	
	are trivial.
	The implication $\eqref{151122-1348-pt3}\Rightarrow\eqref{151122-1348-pt2}$ follows from \cite[Theorem 3.14]{CF96}.
	It remains to prove the implication $\eqref{151122-1348-pt2}\Rightarrow\eqref{151122-1348-pt4}$. 
	Following the proof of Theorem \ref{280822-2053} everything remains the same until \ref{131122-2307}.
It remains to justify that the upper bound in \eqref{131122-2308} can be decreased to $m:=k(\ell+1)-\lfloor\frac{\ell}{2}\rfloor+1$.
If $\Rank A_{\gamma}\leq m$, then we are already done. From now on we assume that $r:=\Rank A_{\gamma}>m$. 
Since $\gamma$ admits a $(\RR\setminus\{0\})$--representing measure, which we denote by $\mu$,
it is nrg and 	
	$\Rank \gamma=\Rank A_\gamma=\Rank A_{\gamma}[r-1]$.
Hence,
	$A_{\gamma^{(2k-2(r-1),2k)}}$ is pd
and in particular also its submatrix
	$A_{\gamma^{(2k-2(m-1),2k)}}$ is pd.
We denote by
	$\gamma(\texttt{z}_1,\ldots,\texttt{z}_\ell)$
a sequence obtained from the sequence $\gamma$ by replacing
the moments 
	$\gamma_{-2k\ell},\gamma_{-2k\ell+1},\ldots,\gamma_{-2k\ell+\ell-1}$
with variables $\texttt{z}_1,\ldots,\texttt{z}_\ell$.
By \cite[Theorem 3.1]{Zal22b}, the sequence $\gamma^{(2k-2(m-1),2k)}$ has 
a $m$--atomic $(\RR\setminus\{0\})$--rm
(to apply \cite[Theorem 3.1]{Zal22b} we used that 
	$2k-2(m-1)=-2k\ell+2\big\lfloor \frac{\ell}{2}\big\rfloor <0).$
We denote the measure obtained in this way by $\mu_1$
and generate its moment sequence $\gamma(z_1,\ldots,z_{\ell})$, where $z_1,\ldots,z_{\ell}$
are the moments of degrees $-2k\ell,-2k\ell+1,\ldots,-2k\ell+\ell-1$, respectively.
	If none of the moments $\beta_{i,j}$ of the sequence $\beta^{(2k-1)}$ depends on 
		$\gamma_{-2k\ell},\gamma_{-2k\ell+1},\ldots,\gamma_{-2k\ell+\ell-1}$, 
	then $\mu_1$ will generate the $K$--rm
	for $\beta^{(2k-1)}$ as in the proof of Claim 1 of Theorem \ref{280822-2053}. 
	But by definition \eqref{280822-2111}, there is indeed no moment from $\beta^{(2k-1)}$ depending on 
		$\gamma_{-2k\ell}, \gamma_{-2k\ell+1},\ldots, \gamma_{-2k\ell+\ell-1}$ ,
	which concludes the proof of Theorem \ref{151122-1348}.
\end{proof}

A corollary to Theorem \ref{280822-2053}  is an improvement of the bounds on the degrees of sums of squares in the Positivstellensatz \cite[Corollary 6.4]{Fia11} for the curves of the form $yx^\ell=1,$ $\ell\in \NN\setminus\{1\}$.

\begin{corollary}
	\label{290822-1052}
	Let $K:=\{(x,y)\in \RR^2\colon yx^\ell=1\}$, where $\ell\in \NN\setminus\{1\}$ and $k\geq \ell+1$.
	If $r(x,y)\in\RR[x,y]_{2k}$ is strictly positive on $K$, then $r$ admits a decomposition
		$$r(x,y)=\sum_{i=1}^{\ell_1} f_{i}(x,y)^2+(yx^\ell-1)\sum_{i=1}^{\ell_2}g_i(x,y)^2-(yx^\ell-1)\sum_{i=1}^{\ell_2}h_i(x,y)^2,$$
	where $\ell_1,\ell_2,\ell_3\in \ZZ_+$,
		$f_i,g_i,h_i\in \RR[x,y]$
	and 
		$$
			\deg f_i^2\leq 2m,
			\; \deg ((yx^\ell-1)g_i^2)\leq 2m,		
			\; \deg ((yx^\ell-1)h_i^2)\leq 2m
		$$
	with
		$m=k+\ell+1.$
	where $\ell_1,\ell_2,\ell_3\in \ZZ_+$.
\end{corollary}

\begin{proof}
	By the equivalence $\eqref{280822-2053-pt1}\Leftrightarrow\eqref{280822-2053-pt2}$
	of Theorem \ref{280822-2053}, the set $K$ has the property $(R_{k,\ell})$ in the notation of \cite[p.\ 2713]{CF08}.
	Now the result follows by \cite[Theorem 1.5]{CF08}.
\end{proof}

\begin{remark}
The bound on $m$ in Theorem \ref{290822-1052} in \cite[Corollary 6.4]{Fia11}
is quadratic in $k$ and $\ell$, namely $(2k+2)(2+\ell)-(1+\ell)$.
\end{remark}

\subsection{A solution to the TMP based on the feasibility of a linear matrix inequality}
\label{251122-2140}

In this subsection we give another alternative solution to the TMP on curves $yx^\ell,$ where $\ell\in \NN\setminus\{1\}$, which is based on the 
feasibility of a linear matrix inequality associated to the univariate sequence $\gamma$, obtained from 
the original sequence $\beta$ as in the proof of the results in the previous subsection.
The feasibility question appears as a result of the fact that $\gamma$ is not fully determined by $\beta$,
but $\beta$ admits a $K$--representing measure if and only if
$\gamma$ can be completed to a sequence
admitting a $(\RR\setminus \{0\})$--representing measure. 

For $n_1,n_2\in \ZZ$, $n_1\leq n_2$, we denote by $[n_1:n_2]$ the set of all integers between $n_1$ and $n_2$.
Let the sets $N_1,N_2$ form a partition of $[n_1:n_2]$, i.e., $N_1,N_2\subseteq [n_1:n_2]$, $N_1\cup N_2=[n_1:n_2]$ and $N_1\cap N_2=\emptyset.$
Let $\Gamma_1:=(\gamma_t)_{t\in N_1}$  be a sequence of real numbers indexed by integers from $N_1$
and $\boldsymbol\Gamma_2:=(\boldsymbol\gamma_t)_{t\in N_2}$ a tuple of variables indexed by integers from $N_2$. 
Let 
\begin{equation}
	\label{230305-1020}
		F_{\Gamma_1}(\boldsymbol\Gamma_2):\RR^{|N_2|}\to \RR^{|N_1|+|N_2|}
\end{equation}
be a function with the output a sequence $(\widetilde \gamma_t)_{t\in [n_1:n_2]}$
where
$\widetilde \gamma_t=
\left\{
\begin{array}{rr}
\gamma_t,&	
	\text{if } t\in N_1,\\
\boldsymbol\gamma_t,&
	\text{if } t\in N_2.
\end{array}
\right.$

In Theorem \ref{14122-1729} below the set $N_1$ will be the set of indices, for which the corresponding univariate sequence $\gamma$
is determined by $\beta$, while the indices of the non-determined part will belong to $N_2$. Since we can either get a sequence with the lowest and highest degree terms both of odd degree or both of even degree or only the highest term of odd degree, for which the solutions to the STHMP are slightly different, we separate three cases for $N_1\cup N_2=[n_1:n_2]$ (see \eqref{110323-1841}).

\begin{theorem}
	\label{14122-1729}
	Let $K:=\{(x,y)\in \RR^2\colon yx^\ell=1\}$, $\ell\in \NN\setminus\{1\}$,  
	and 
		$$\beta:=\beta^{(d)}=(\beta_{i,j})_{i,j\in \ZZ_+,i+j\leq d},$$ 
	where $\big\lceil \frac{d}{2}\big\rceil\geq \ell+1$.
	Define
	$$
		N_{1}:=\{t\in \ZZ_-\colon t=-i\ell+j\;\;\text{for some}\;\;0\leq j<\ell, i\in \ZZ_+\text{ and }i+j\leq d\},
	$$
	\begin{equation}
		\label{241122-2232}
			\gamma_t=
			\left\{
			\begin{array}{rr}
			\beta_{t,0},&	\text{if }t\in (\NN\cup \{0\})\cap N_{1},\\
			\beta_{t+\ell\big\lceil{\frac{|t|}{\ell}}\big\rceil,\big\lceil\frac{|t|}{\ell}\big\rceil},&\text{if }t\in N_{1}\setminus  (\NN\cup \{0\}),
			\end{array}
			\right.
	\end{equation}
	and $\Gamma_1:=(\gamma_t)_{t\in N_1}$.
	Let
	\begin{equation}
		\label{110323-1841}
		[n_1:n_2]
		:=
		\left\{
		\begin{array}{rr}
		[-d\ell-2:d+2],& \text{if }d\text{ is even},\\[0.2em]
		[-d\ell-2:d+1],& \text{if only }\ell \text{ is even},\\[0.2em]
		[-d\ell-1:d+1],&  \text{if }d,\ell \text{ are odd},
		\end{array}
		\right.
	\end{equation}
	$\boldsymbol\Gamma_2:=(\boldsymbol\gamma_t)_{t\in N_2}$ be a tuple of variables with $N_2=[n_1:n_2]\setminus N_1$
	and
	$F_{\Gamma_1}(\boldsymbol\Gamma_2)$ be defined as in \eqref{230305-1020}.
	Then the following statements are equivalent:
	\begin{enumerate}
		\item
		 	\label{14122-1729-pt1}
				$\beta$ has a $K$--representing measure.
		\item
		 	\label{14122-1729-pt2}
				$\beta_{i+\ell,j+1}=\beta_{i,j}$ for every $i,j\in \ZZ_+$ such that $i+j\leq d-\ell-1$ and 
				there exists a tuple
				 	$\Gamma_2=(\gamma_t)_{t\in N_2}\in \RR^{N_2}$
				such that
					$A_{F_{\Gamma_1}(\Gamma_2)}\succeq 0$.
	\end{enumerate}
\end{theorem}

\begin{proof}
	Assume that $d$ is even.
	Observing the proof of Theorem \ref{280822-2053} one can notice that
			$F_{\Gamma_1}(\boldsymbol\Gamma_2)$
	corresponds to the sequence $\widetilde\gamma$.
The original sequence $\beta$ determines only $\gamma_t$
for $t\in N_1$ by \eqref{241122-2232}, while for $t\in N_2$, $\boldsymbol\gamma_t$ are variables. By the proof of Theorem \ref{280822-2053},
$\beta$ will have a $K$--rm iff it satisfies the rg relations coming from the column relation $YX^\ell=1$ and there exists 
$\widetilde \gamma$ such that $A_{\widetilde \gamma}\succeq 0$.
This proves Theorem \ref{14122-1729} for even $d$.

Observing the proof of Theorem \ref{151122-1348} in case $d$ is odd one can notice that only 
	$$\gamma^{(-d\ell,d)}=(\gamma_{-d\ell},\gamma_{-d\ell+1},\ldots,\gamma_{d-1},\gamma_d)$$
needs to have a $(\RR\setminus\{0\})$--rm to obtain a $K$--rm for $\beta$. In case $d\ell$ is even, this is 
equivalent to $A_{\gamma^{(-d\ell-2,d+1)}}\succeq 0$, where 
	$$\gamma^{(-d\ell-2,d+1)}=(\gamma_{-d\ell-2},\gamma_{-d\ell-1},\ldots,\gamma_{d},\gamma_{d+1})$$
for some  $\gamma_{-d\ell-2}$, $\gamma_{-d\ell-1}$ and $\gamma_{d+1}$.
Since $\gamma^{(-d\ell-2,d+1)}$ corresponds to the sequence 
			$F_{\Gamma_1}(\boldsymbol\Gamma_2)$
for even $\ell$ and odd $d$, 
this proves  Theorem \ref{151122-1348} in this case.
If $d\ell$ is odd, then it suffices that there are $\gamma_{-d\ell-1}$
and $\gamma_{d+1}$ such that $A_{\gamma^{(-d\ell-1,d+1)}}\succeq 0$, where 
	$$\gamma^{(-d\ell-1,d+1)}=(\gamma_{-d\ell-1},\gamma_{-d\ell},\ldots,\gamma_{d},\gamma_{d+1}).$$
Since $\gamma^{(-d\ell-1,d+1)}$ corresponds to the sequence 
			$F_{\Gamma_1}(\boldsymbol\Gamma_2)$
for odd $d\ell$,
this proves Theorem \ref{151122-1348} in this case. 
\end{proof}

We will present the statement of Theorem \ref{14122-1729} on a few examples.
The following example is for $\ell=2$ and a sequence $\beta$ of even degree.

\begin{example}
	\label{141122-1921}
		Let $\beta=(\beta_{i,j})_{i,j\in \ZZ_+,i+j\leq 2k}$
	be a bivariate sequence of degree $2k$, $k\geq 3$,
	and $K:=\{(x,y)\in \RR^2\colon yx^2=1\}$.
	For the existence of a $K$--rm $\beta$
	must satisfy the relations $\beta_{i,j}=\beta_{i+2,j+1}$
	for every $i,j\in \ZZ_+$ such that $i+j\leq 2k-3$.
	In the notation of Theorem \ref{14122-1729},
	\begin{align*}
		N_1	
		&:=\{
		t\in \ZZ\colon -4k\leq t\leq 2k, t\neq-4k+1
		\},\\
		[n_1:n_2]
		&:=[-4k-2:2k+2],\quad
		N_2=\{-4k-2,-4k-1,-4k+1,2k+1,2k+2\},
	\end{align*}
	the formula 	\eqref{280822-2111} is equal to 
	\begin{equation*}
		\gamma_{t}
		=
		\left\{
		\begin{array}{rl}
			\beta_{t,0},& \text{if }t\in \ZZ_+\cap N_1,\\[0.5em]
			\beta_{t+2\big\lceil{\frac{|t|}{2}}\big\rceil,\big\lceil\frac{|t|}{2}\big\rceil},& \text{if } t\in N_1\setminus(\ZZ_+\cap N_1),
		\end{array}
		\right.
	\end{equation*}
	the function $F_{\Gamma_1}:\RR^{5}\to \RR^{6k+4}$ is defined by
	\begin{align*}
		F_{\Gamma_1}(\boldsymbol\Gamma_2)
		&=
		F_{\Gamma_1}(\boldsymbol\gamma_{\mathbf{-4k-2}},\boldsymbol\gamma_{\mathbf{-4k-1}},\boldsymbol\gamma_{\mathbf{-4k+1}},\boldsymbol\gamma_{\mathbf{2k+1}},\boldsymbol\gamma_{\mathbf{2k+2}})\\
		&:=(\boldsymbol\gamma_{\mathbf{-4k-2}},\boldsymbol\gamma_{\mathbf{-4k-1}},\gamma_{-4k},\boldsymbol\gamma_{\mathbf{-4k+1}},
			\gamma_{-4k+2},\ldots,\gamma_{2k},\boldsymbol\gamma_{\mathbf{2k+1}},\boldsymbol\gamma_{\mathbf{2k+2}}),
	\end{align*}
	and the matrix 
	$A_{F_{\Gamma_1}(\boldsymbol\Gamma_2)}$ 
	is equal to
\begin{equation*}\label{vector-v}
						\left(\begin{array}{c|cccccc|c}
						\boldsymbol\gamma_{\mathbf{-4k-2}} & \boldsymbol\gamma_{\mathbf{-4k-1}} & \gamma_{-4k}& \boldsymbol\gamma_{\mathbf{-4k+1}} & \gamma_{-4k+2}&\cdots& \gamma_{k} & \gamma_{k+1}\\
						\hline
						\boldsymbol\gamma_{\mathbf{-4k-1}} & \gamma_{-4k} & \boldsymbol\gamma_{\mathbf{-4k+1}} &\gamma_{-4k+2}&\iddots&&  \gamma_{k+1} &\gamma_{k+2}\\
						\gamma_{-4k} & \boldsymbol\gamma_{\mathbf{-4k+1}} &\gamma_{-4k+2}&\iddots && &\gamma_{k+2}&\vdots\\
						\boldsymbol\gamma_{\mathbf{-4k+1}} & \beta_{0,2k-1}&\iddots&	&&&\vdots&\vdots\\
						\gamma_{-4k+2} &\iddots& &	 & & &\vdots &\gamma_{2k}\\
						\vdots &&& && &\gamma_{2k}
& \boldsymbol\gamma_{\mathbf{2k+1}}\\
\hline
\gamma_{k+1}& \cdots & \cdots &\cdots & \gamma_{2k-1}&\gamma_{2k}& \boldsymbol\gamma_{\mathbf{2k+1}} & \boldsymbol\gamma_{\mathbf{2k+2}}
						\end{array}\right)
	\end{equation*}
	The question of feasibility of
	$A_{F_{\Gamma_1}(\boldsymbol\Gamma_2)}$ can be answered analytically, 
	since the structure of the missing entries is simple enough. Actually it is even easier to work with
	$A_{(\gamma_{-4k},\boldsymbol\gamma_{\mathbf{-4k+1}},\gamma_{-4k+2},\ldots,\gamma_{2k})}$ 
	and answer the feasibility question together with the conditions from the solution of \cite[Theorem 2.1]{Zal22b} (see \cite[Theorem 4.1]{Zal22b}).
\end{example}

The following example demonstrates the statement of Theorem \ref{14122-1729} 
for the case $yx^3=1$ and a sequence $\beta$ of even degree.

\begin{example}
	\label{141122-2021}
		Let $\beta=(\beta_{i,j})_{i,j\in \ZZ_+,i+j\leq 2k}$
	be a bivariate sequence of degree $2k$, $k\geq 4$,
	and $K:=\{(x,y)\in \RR^2\colon yx^3=1\}$.
	For the existence of a $K$--rm $\beta$ must satisfy the relations
	$\beta_{i,j}=\beta_{i+3,j+1}$ for every $i,j\in\ZZ_+$ such that $i+j\leq 2k-4$. 
	In the notation of Theorem \ref{14122-1729},
	\begin{align*}
		N_1	
		&:=\{
		t\in \ZZ\colon -6k\leq t\leq 2k, t\notin\{-6k+1,-6k+2,-6k+5\}
		\},\\
		[n_1:n_2]
		&:=[-6k-2:2k+2],\\
		N_2
		&:=\{-6k-2,-6k-1,-6k+1,-6k+2,-6k+5,2k+1,2k+2\},
	\end{align*}
	the formula 	\eqref{280822-2111} is equal to 
	\begin{equation*}
		\gamma_{t}
		=
		\left\{
		\begin{array}{rl}
			\beta_{t,0},& \text{if }t\in \ZZ_+\cap N_1,\\[0.5em]
			\beta_{t+3\big\lceil{\frac{|t|}{3}}\big\rceil,\big\lceil\frac{|t|}{3}\big\rceil},& \text{if } t\in N_1\setminus(\ZZ_+\cap N_1),		
		\end{array}
		\right.
	\end{equation*}
	the function $F_{\Gamma_1}:\RR^{7}\to \RR^{8k+4}$ is defined by
	\begin{align*}
		F_{\Gamma_1}(\boldsymbol\Gamma_2)
		&=F_{\Gamma_1}(\boldsymbol\gamma_{\mathbf{-6k-2}},\boldsymbol\gamma_{\mathbf{-6k-1}},\boldsymbol\gamma_{\mathbf{-6k+1}},\boldsymbol\gamma_{\mathbf{-6k+2}},\boldsymbol\gamma_{\mathbf{-6k+5}},\boldsymbol\gamma_{\mathbf{2k+1}},\boldsymbol\gamma_{\mathbf{2k+2}})\\
		&:=(\boldsymbol\gamma_{\mathbf{-6k-2}},\boldsymbol\gamma_{\mathbf{-6k-1}},{\gamma_{-6k}},\boldsymbol\gamma_{\mathbf{-6k+1}},\boldsymbol\gamma_{\mathbf{-6k+2}},\gamma_{-6k+3},\gamma_{-6k+4},\\
		&\hspace{2cm}
			\boldsymbol\gamma_{\mathbf{-6k+5}},\gamma_{-6k+6},\ldots,\gamma_{2k},\boldsymbol\gamma_{\mathbf{2k+1}},\boldsymbol\gamma_{\mathbf{2k+2}}),
	\end{align*}
	and the matrix 
	$A_{F_{\Gamma_1}(\boldsymbol\gamma_{\mathbf{-6k-2}},\boldsymbol\gamma_{\mathbf{-6k-1}},\boldsymbol\gamma_{\mathbf{-6k+1}},\boldsymbol\gamma_{\mathbf{-6k+2}},\boldsymbol\gamma_{\mathbf{-6k+5}},\boldsymbol\gamma_{\mathbf{2k+1}},\boldsymbol\gamma_{\mathbf{2k+2}})}$ 
	is equal to
\begin{footnotesize}
\begin{equation*}\label{vector-v}
						\left(\begin{array}{c|ccccccccc|c}
						\boldsymbol\gamma_{\mathbf{-6k-2}} & \boldsymbol\gamma_{\mathbf{-6k-1}} & \gamma_{-6k} & \boldsymbol\gamma_{\mathbf{-6k+1}} & \boldsymbol\gamma_{\mathbf{-6k+2}}&
						\gamma_{-6k+3}&\gamma_{-6k+4}&\boldsymbol\gamma_{\mathbf{-6k+5}}&
						\cdots& \gamma_{k} & \gamma_{k+1}\\
						\hline
						\boldsymbol\gamma_{\mathbf{-6k-1}} & \gamma_{-6k} & \boldsymbol\gamma_{\mathbf{-6k+1}} &\boldsymbol\gamma_{\mathbf{-6k+2}}&\gamma_{-6k+3}&\gamma_{-6k+4}&\boldsymbol\gamma_{\mathbf{-6k+5}}& \cdots&\cdots&\gamma_{k+1} &\gamma_{k+2}\\
						\gamma_{-6k} & \boldsymbol\gamma_{\mathbf{-6k+1}} &\boldsymbol\gamma_{\mathbf{-6k+2}} &\gamma_{-6k+3}&\gamma_{-6k+4} &\boldsymbol\gamma_{\mathbf{-6k+5}}& \iddots&&&\gamma_{k+2}&\vdots\\
						\boldsymbol\gamma_{\mathbf{-6k+1}} & \boldsymbol\gamma_{\mathbf{-6k+2}}&\gamma_{-6k+3}&\gamma_{-6k+4}&	\boldsymbol\gamma_{\mathbf{-6k+5}}&\iddots&&&&\vdots&\vdots\\
						\boldsymbol\gamma_{\mathbf{-6k+2}} &\gamma_{-6k+3}&\gamma_{-6k+4}&\boldsymbol\gamma_{\mathbf{-6k+5}}&\iddots& &	 & & &\vdots &\vdots\\
						\vdots &\gamma_{-6k+4}&\boldsymbol\gamma_{\mathbf{-6k+5}}&\iddots&&& && &\vdots& \vdots\\
						\vdots &\iddots&\iddots&&&& && &\vdots& \vdots\\
						\vdots &\iddots&&&&& && &\vdots& \gamma_{2k}\\
						\vdots &&&&&& && &\gamma_{2k}
& \boldsymbol\gamma_{\mathbf{2k+1}}\\
\hline
\gamma_{k+1}& \cdots & \cdots &\cdots & \cdots & \cdots & \cdots & \gamma_{2k-1}& \gamma_{2k} & \boldsymbol\gamma_{\mathbf{2k+1}} & \boldsymbol\gamma_{\mathbf{2k+2}}
						\end{array}\right)
	\end{equation*}	
\end{footnotesize}
	In contrast to the situation $yx^2=1$ from Example 	\ref{141122-1921}, the structure of the missing entries here is too complicated for the analytic approach and we believe the feasibility question can only be answered numerically 
	using linear matrix inequality solvers.
\end{example}

\begin{remark}
	It would be interesting to know, to what extent does the result \cite[Corollary 7.6]{RS18} (see Remark \ref{191122-1055}.\eqref{191122-1055-pt2}) extend to even degree sequences
	on all plane curves. As explained in Remark \ref{191122-1055}.\eqref{191122-1056}, one needs one more atom in the upper bound for curves of the form $y=q(x)$ with $\deg q=2$
	and the same is true if $\deg q\leq 1$ by Remark \ref{191122-1055}.\eqref{191122-1056-pt4}. On the other hand the results of the present paper suggest that for curves $\cZ(p)$, $\deg p\geq 3$, 	
	the upper bound could be $k\deg p$. Also from the concrete solution to the TMP on the curve $\cZ(y^2-x^3)$ \cite[Corollary 4.3]{Zal21} it follows that the same bound works. However,
	the forthcoming result of Bhardwaj \cite{Bha+} shows that also for degree 3 curves the upper bound has to be loosened, by constructing a truncated moment sequence of degree $2k=6$ on a curve 
	$\cZ(p)$, where $p(x,y)=y^2-x^3+ax-1$, $a=\frac{524287}{262144}$, with a minimal measure consisting of 10 atoms, which is $k\deg p+1$.
\end{remark}


\section*{Acknowledgements} 
I would like to thank the anonymous referees for very useful comments and suggestions for improving the manuscript.
Numerical examples in this paper were obtained using the software tool \textit{Mathematica}.

\section*{Funding} Supported by the Slovenian Research Agency grants J1-2453, J1-3004, P1-0288.

\end{document}